\documentclass[11pt, reqno]{amsart}

\usepackage{amsfonts}
\usepackage{amsmath}
\usepackage[utf8]{inputenc}
\usepackage[english]{babel}
\usepackage{amsthm}
\usepackage{graphicx}
\usepackage{subfigure}
\usepackage{color}
\usepackage{todonotes}
\usepackage{soul}

\usepackage{enumitem}
\usepackage{amssymb}
\usepackage{amsopn}
\usepackage{comment}
\usepackage{wrapfig}

\usepackage{mathtools}
\usepackage[normalem]{ulem}

\usepackage{pifont,bbding}
\usepackage{comment}
\usepackage{mathrsfs}

\usetikzlibrary{decorations.pathreplacing,arrows.meta,calc,positioning,fit}
\usetikzlibrary{backgrounds} 
\tikzset{%
    >=stealth, 
    nodelayer/.style={},
    edgelayer/.style={},
}

\usepackage{hyperref}

\usepackage{autonum}
\newcommand{\N}{\mathbb{N}}
\newcommand{\R}{\mathbb{R}}

\newcommand{\eps}{\varepsilon}
\newcommand{\oo}{\infty}

\newcommand{\bH}{\mathbb{H}}

\newcommand{\grad}{\nabla}
\newcommand{\hess}{\nabla\nabla}
\newcommand{\gradeps}{\nabla_\varepsilon}

\newcommand{\hesseps}{\nabla_\varepsilon\nabla_\varepsilon}

\newenvironment{bvprob}
  {\left\{\begin{aligned}
  }
  {\end{aligned}\right.}

\newcommand{\cond}[2]{&\quad \text{#1 } #2}

\usepackage{xcolor}
\usepackage{marginfix}
\usepackage{ragged2e}
\usepackage{ifoddpage}

\definecolor{vale}{rgb}{0,0,1}
\definecolor{mattia}{rgb}{0.5, 0.0, 0.5}
\definecolor{gaia}{rgb}{1.0, 0.31, 0.0}

\newcommand{\MarginNote}[3]{%
  \marginpar{%
  \begin{minipage}{.9\marginparwidth}
    \color{#1}%
    \tiny
    \checkoddpage
    \ifoddpage
      \RaggedRight
      $\longleftarrow$ #2:\par
    \else
      \RaggedLeft
      #2: $\longrightarrow$\par
    \fi
    \vspace{1pt}%
    \hspace{0pt}#3%
    \end{minipage}
  }%
}

\newcommand{\noteV}[1]{\MarginNote{vale}{Vale}{#1}}

\theoremstyle{plain}
\newtheorem{thm}{Theorem}[section]
\newtheorem{prop}[thm]{Proposition}
\newtheorem{cor}[thm]{Corollary}
\newtheorem{lem}[thm]{Lemma}
\theoremstyle{definition}
\newtheorem{defn}[thm]{Definition}

\theoremstyle{remark}
\newtheorem{rem}[thm]{Remark}

\numberwithin{equation}{section}

\newlist{checklist}{itemize}{1}
\setlist[checklist]{label=$\square$, leftmargin=*}

\title{Mean-convex Horizontal mean curvature flow in the Heisenberg group}
\author[G.~Bombardieri]{Gaia Bombardieri}
\address[G.~Bombardieri]{Dipartimento di Matematica ``Tullio Levi Civita'', Universit\`a di Padova, via Trieste 63, 35121 Padova (PD), Italy}
\email{gaia.bombardieri@studenti.unipd.it}

\author[M.~Fogagnolo]{Mattia Fogagnolo}
\address[M.~Fogagnolo]{Dipartimento di Matematica ``Tullio Levi Civita'', Universit\`a di Padova, via Trieste 63, 35121 Padova (PD), Italy}
\email{mattia.fogagnolo@unipd.it}

\author[V.~Franceschi]{Valentina Franceschi}
\address[V.~Franceschi]{Dipartimento di Matematica ``Tullio Levi Civita'', Universit\`a di Padova, via Trieste 63, 35121 Padova (PD), Italy}
\email{valentina.franceschi@unipd.it}

\begin{document}

\begin{abstract}
    We study the arrival time equation of the mean-convex horizontal mean curvature flow in the Heisenberg group. We prove existence of 
    viscosity solutions when the initial set is horizontally mean-convex and either non-characteristic, or axisymmetric and with horizontal mean curvature extending continuously across the characteristic set. A key ingredient is an approximation scheme by sets with prescribed Riemannian mean curvature, which preserves the mean-curvature bounds up to an arbitrarily small error. We further show that, when the evolving level sets form a generalized \(C^2\)-foliation, the horizontal mean curvature blows up around characteristic points. This provides an obstruction to self-similar evolution in the axisymmetric class considered above and, in particular, shows that the Pansu sphere cannot evolve self-similarly.
\end{abstract}

\maketitle


\section{Introduction}
Let $\mathbb H$ be the Heisenberg group, identified with $\R^3$ with coordinates $p=(x,y,z)$ and equipped with the \emph{horizontal distribution}
\begin{equation}\label{eq:campi}
    \mathcal{D}=\operatorname{span}\{X,Y\},
    \qquad 
    X = \partial_x - \frac{y}{2}\,\partial_z, 
    \quad 
    Y = \partial_y + \frac{x}{2}\,\partial_z,
\end{equation}
and the sub-Riemannian structure defined as the one making $\{X,Y\}$ an orthonormal frame for $\mathcal D$. In particular, a horizontal vector $\xi=\xi_X X+\xi_Y Y\in\mathcal D$ has norm $|\xi|=\sqrt{\xi_X^2+\xi_Y^2}$.

In this paper, we study mean curvature flow evolutions of surfaces in $\bH$, governed by the horizontal geometry induced throughout the evolution. A key feature of the induced geometry of a surface is its set of characteristic points: if $\Omega\subset\mathbb R^3$ is a bounded open set with $C^2$ boundary, we call \emph{characteristic points} of $\partial\Omega$ those at which the surface is tangent to the distribution, and we set
\begin{equation}
    \mathcal C(\partial\Omega)=\{p\in \partial \Omega : \mathcal D_p\subset T_p\partial\Omega\}.
\end{equation}
At non-characteristic points, the Euclidean outer unit normal admits a non-trivial horizontal projection. After normalization, this projection defines the \emph{horizontal outer unit normal}
$\nu$, a unit vector field in $\mathcal{D}$. The \emph{horizontal mean curvature} of $\partial\Omega$ is then defined as the divergence of this vector field,
\begin{equation}\label{eq:H-intro}
H=\operatorname{div}\nu
\qquad\text{on } \partial\Omega\setminus\mathcal C(\partial\Omega).    
\end{equation}

In the Euclidean setting, the mean curvature flow evolves a smooth hypersurface in its normal direction with normal velocity equal to its mean curvature. A formal horizontal analogue of the classical mean curvature flow is obtained by prescribing that a family of hypersurfaces $\{\partial\Omega_t\}_{t\in[0,T)}$ with $T>0$, parametrized on a reference surface $\Gamma$ by
\begin{equation}
    F:\Gamma\times[0,T)\longrightarrow\mathbb R^3,
    \quad
    \partial\Omega_t=F(\Gamma,t),
\end{equation}
 evolves according to
\begin{equation}\label{eq:HMCF_formal}
\partial_tF(q,t)
=
-H(F(q,t))\nu(F(q,t)),
\end{equation}
with initial datum $\partial\Omega_0=\partial\Omega$. 
This evolution law is not defined at characteristic points, where the horizontal normal is not well defined.
For this reason, weak notions of horizontal mean curvature flow have been developed, for instance via viscosity solutions in \cite{capognacitti09,Ferrariliumanfredi_2014,DirrDragoniVonRenne10,dirrdragonigrande20}, in analogy with the classical level-set formulation developed in~\cite{evansspruck91,ChenGigaGoto1991} or via variational techniques in the recent preprint~\cite{CNP26}, in analogy with the classical minimizing movements scheme developed in~\cite{ATW93,LS95}.

A major difficulty in weak formulations is that fine geometric information along the flow is hard to retain. In this paper, we impose some structure on the initial sets, aiming to obtain more geometric information along the flow. More precisely, we assume that the horizontal mean curvature of the initial set $\Omega$ is  positive and bounded away from zero on the non-characteristic set $\partial \Omega\setminus\mathcal C(\partial \Omega)$. We refer to sets satisfying this condition as \emph{(strictly) horizontally mean convex}.

The appropriate formulation in this context is given in terms of the \emph{arrival time function}, introduced in~\cite{evansspruck91}. Suppose formally that there exists a function $u \in C^2(\mathbb{R}^3)$ such that $\Omega_t=\{ p\in \R^3 : u(p)> t \}$ for $t>0$. 
Thus, $u(p)$ represents the time at which the evolving boundary reaches the point $p$. 
At a regular point of a level set $\partial\Omega_t$ the horizontal outer unit normal is
\begin{equation}
    \nu=-\frac{\nabla u}{|\nabla u|},
\end{equation}
where $\nabla u$ denotes the horizontal gradient, see~\eqref{eq:grad}. Differentiating the identity
$u(F(q,t))=t$ and using \eqref{eq:HMCF_formal}, we formally obtain
\begin{equation}
\label{eq:formal}
1 
= \partial_t\, u(F(q,t))
=  \grad u \cdot (- H(F(q,t)) \, \nu(F(q,t)))
= H (F(q,t))\, |\grad u(F(q,t))|,
\end{equation}
so that, using the explicit expression \eqref{eq:H} for $H$ in terms of $u$, the arrival time function formally satisfies 
\begin{equation}\label{eq:HMCFeq}
   -\Delta u+\frac{\grad\grad u(\grad u,\grad u)}{|\grad u|^2}=1.
\end{equation}
Here $\Delta$ and $\nabla\nabla$ denote, respectively, the horizontal sub-Laplacian and the horizontal Hessian recalled in Section~\ref{ss:Hoperators}. We are therefore led to the boundary value problem
\begin{equation}
    \label{HMCFprobl}\tag{MCF}
    \begin{bvprob}
   -\Delta u+\frac{\grad\grad u(\grad u,\grad u)}{|\grad u|^2}=1 \cond{in}{\Omega},\\
   u=0 \cond{on}{\partial\Omega},
\end{bvprob}
\end{equation}
which is referred to as the \emph{level-set formulation of the horizontal mean curvature flow with initial datum $\Omega$}, understood in the viscosity sense at points where the horizontal gradient vanishes.

Identity \eqref{eq:formal} also explains the role of horizontal mean convexity: wherever the arrival time is smooth and non-characteristic,
\(H={|\nabla u|^{-1}}>0\).
Thus, solving \eqref{HMCFprobl} provides a natural weak notion of mean-convex horizontal mean curvature flow.

\medskip 

Our first result establishes the existence of a mean-convex horizontal mean curvature flow under two alternative geometric assumptions: either the boundary is non-characteristic and has strictly positive horizontal mean curvature, or the domain is axisymmetric with sphere-like boundary and its horizontal mean curvature admits a strictly positive continuous extension across characteristic points; see Section~\ref{s:axisym}.

The existence statement below involves the right-invariant distance $\tilde d$, introduced in Section~\ref{ss:rightinv}, which is the sub-Riemannian distance associated with the right-invariant distribution
\begin{equation}\label{eq:campi_right}
\mathcal{\tilde D}=\operatorname{span}{\tilde X,\tilde Y},
\qquad
\tilde X = \partial_x + \frac{y}{2}\partial_z,
\quad
\tilde Y = \partial_y - \frac{x}{2},\partial_z.
\end{equation}

\begin{thm}
    \label{thm:existence}
    Let $\Omega \subset \R^3$ be a bounded open set with $C^2$ boundary.
    Assume that $\Omega$ satisfies one of the following geometric conditions:
    \begin{itemize}
        \item[(A)] $\Omega$ is axisymmetric, $\partial \Omega$ is homeomorphic
        to a sphere, and the horizontal mean curvature of $\partial\Omega$
        extends continuously to $\partial\Omega$ and is strictly positive
        there.
        \item[(B)] $\partial \Omega$ is non-characteristic and has strictly
        positive horizontal mean curvature.
    \end{itemize}
    Then $\Omega$ admits a mean-convex horizontal mean curvature flow, i.e.,
    there exists a viscosity solution $u \in C^0(\overline{\Omega})$ to
    \eqref{HMCFprobl}, which is right-Lipschitz, i.e., Lipschitz with respect to the right-invariant distance $\tilde d$. 
\end{thm}

We note that in case $(A)$, right-Lipschitz continuity is equivalent to left-Lipschitz continuity, namely Lipschitz continuity with respect to the sub-Riemannian distance $d$; see Remark~\ref{rem:rotinvariancegrad}. 

The two assumptions in Theorem~\ref{thm:existence} are independent. The Pansu sphere, described below, satisfies assumption $(A)$ and has two characteristic points.
Conversely, assumption $(B)$ allows for different topologies and does not require axisymmetry. To illustrate this, consider axisymmetic tori, which always have empty characteristic sets but need not to be horizontally mean convex. For instance, 
the tori 
${T} = \big\{(x,y,z) : (r(x,y)-r_0)^2+z^2=a^2\big\}$, $0<a<r_0$ 
are horizontally mean convex if the ratio $a/r_0$ is sufficiently small.
Moreover, sufficiently small $C^2$ perturbations of such tori still satisfy assumption $(B)$, providing a rich family of non-axisymmetric tori to which the theorem applies.

Besides establishing existence for this variety of initial sets, Theorem~\ref{thm:existence} yields an evolution that is \emph{non-fattening}, in the sense that
$\operatorname{int}\{u = t\} = \emptyset$ for every $t > 0$.
Indeed, \eqref{eq:HMCFeq} is manifestly incompatible with $u$ being constant on an open set, and
 in fact, testing the supersolution condition~\eqref{eq:supersol_eta}
esasily shows it.

\medskip 

Our proof of existence for~\eqref{HMCFprobl} relies on an approximation procedure.
A key technical step, which may be of independent interest, consists in approximating $\Omega$ by $C^2$ sets whose mean-curvature bounds are preserved, up to an arbitrarily small error, in a family of approximating Riemannian manifolds. These sets will serve as initial data for Riemannian strictly mean-convex mean curvature flows approximating the sub-Riemannian evolution, see Section~\ref{ss:approxsol}.

For $\varepsilon>0$, let $g_\varepsilon$ be the Riemannian metric making
\begin{equation}
\label{eq:campi_eps}
X_\varepsilon=X,
\qquad
Y_\varepsilon=Y,
\qquad
Z_\varepsilon=\varepsilon\partial_z
\end{equation}
an orthonormal frame, and set $\mathbb H_\varepsilon=(\R^3,g_\varepsilon)$. As $\varepsilon\to0^+$, the $g_\varepsilon$-mean curvature $H_\varepsilon$ of $\partial\Omega$ converges to the horizontal mean curvature $H$ away from the characteristic set $\mathcal C(\partial\Omega)$. Thus, under assumption $(B)$, $H_\varepsilon$ is already close to $H$ for sufficiently small $\varepsilon$. 

In case $(A)$, instead, horizontal mean convexity does not in general imply strict mean convexity in $\mathbb H_\varepsilon$. This obstruction already appears for the most natural model surface: the model constant mean curvature sphere $\bH$, known as the \emph{Pansu sphere}~\cite{Pansu82}. For $R>0$, the Pansu sphere of radius $R$ bounds the set
\begin{equation}
\label{eq:pansub}
\mathcal P_R
=
\left\{
p=(x,y,z)\in\R^3:
|z|<f_R(r(p)),\  r(p)<R
\right\},
\end{equation}
where
\begin{equation}
f_R(r)
=
\frac{1}{4}\left(
R^2\arccos\left(\frac{r}{R}\right)
+
r\sqrt{R^2-r^2}
\right)
.
\end{equation}
Its boundary has two characteristic points, corresponding to $r=0$, and a direct computation gives
\begin{equation}
\label{eq:H_Pansu}
H=\frac{2}{R}
\qquad
\text{on }
\partial\mathcal P_R
\setminus
\mathcal C(\partial\mathcal P_R).
\end{equation}
Nevertheless, its $g_\varepsilon$-mean curvature is
\begin{equation}
\label{eq:Heps_Pansu}
H_\varepsilon
=
\frac{
2r\bigl(6R^2\varepsilon^2+r^2(R^2-4\varepsilon^2)\bigr)
}{
\bigl(4R^2\varepsilon^2+r^2(R^2-4\varepsilon^2)\bigr)^{3/2}
}.
\end{equation}
In particular, although $\partial\mathcal P_R$ has constant positive horizontal mean curvature away from its characteristic points, its $g_\varepsilon$-mean curvature vanishes at those points for every $\varepsilon>0$. 

We therefore construct a suitable perturbation $\Omega_\varepsilon$ by prescribing its $g_\varepsilon$-mean curvature in terms of the horizontal mean curvature of $\partial\Omega$. We refer to this procedure as \emph{approximation by prescribed mean curvature}.
\begin{thm}
    \label{thm:approx}
Let $\Omega$ be a bounded open $C^2$-axisymmetric set, homeomorphic to a sphere, such that $H(\partial\Omega)$ extends continuously to $\mathcal C(\partial\Omega)$.
Let $c_1, c_2 \in \R$ be such that the continuous horizontal mean curvature $H$ of $\partial \Omega$ satisfies $c_1 \leq H \leq c_2$ on $\partial \Omega$. Then, for any $\varepsilon>0$, there exists a bounded open axisymmetric set $\Omega_\varepsilon \subset \R^3$ with $C^2$-boundary such that $\partial \Omega_\varepsilon \to \partial \Omega$ in Hausdorff distance as $\varepsilon\to 0$.
Moreover, for any $\eta>0$ there exists $\bar \varepsilon>0$ such that the $g_\varepsilon$-mean curvature of $\partial\Omega_\varepsilon$ satisfies $c_1 - \eta \leq H_\varepsilon \leq c_2 + \eta$ for $\varepsilon<\bar \varepsilon$. 
\end{thm}

Theorem~\ref{thm:approx} covers assumption $(A)$ of Theorem~\ref{thm:existence}.
Axisymmetry of $\Omega$ is used here to reduce the prescribed \(g_\varepsilon\)-mean curvature equation to an ODE. Solving this equation yields the generating profiles of the approximating sets. This construction is carried out in Section~\ref{s:axisym}.

We are finally interested in understanding whether, in the distinguished setting given by assumption $(A)$ of Theorem~\ref{thm:existence}, the horizontal mean curvature flow admits self-similar solutions. 
Here self-similarity is naturally understood with respect to the anisotropic dilations $\{\delta_\lambda\}_{\lambda>0}$ defined as 
\begin{equation}
\label{eq:dilations}
    \delta_\lambda(x,y,z)=(\lambda x,\lambda y,\lambda^2 z),\quad (x,y,z)\in\R^3,\quad \lambda>0.
\end{equation} 

\begin{defn}[Self-similar solution]\label{defn:ss}
Let $\Omega\subset\R^3$ be a bounded open set with $C^2$ boundary, and let $u\in C(\overline{\Omega})$ be a viscosity solution to~\eqref{HMCFprobl} with initial datum $\Omega$. Let $T^\ast=\max_{\bar{\Omega}} u>0$ denote its extinction time. We say that $u$ is a \emph{self-similar solution} if there exists a continuous function
\begin{equation}\label{eq:selfsimilar-scaling-function}
    \lambda:[0,T^\ast)\longrightarrow(0,1],
    \qquad
    \lambda(0)=1, \qquad \lambda(t)\xrightarrow[t\to T^\ast]{}0
\end{equation}
such that for any $t\in[0,T^*)$, the level set 
\begin{equation}
    \label{eq:Gammat}
    \Gamma_t=\bigl\{
        p\in\overline{\Omega}:u(p)=t
    \bigr\},
\end{equation}
satisfies
\begin{equation}\label{eq:defselfsimlambda}
    \Gamma_t
    =
    \delta_{\lambda(t)}(\partial\Omega).
\end{equation}
\end{defn}

In the Euclidean setting, the round sphere is simultaneously an isoperimetric set and a self-similarly shrinking solution to mean curvature flow. More generally, mean curvature flow is well known to have strong connections with the isoperimetry problem; see, for instance~\cite{RitoreSinestrari10} for an account, and \cite{huisken_isoperimetricconceptmassgeneral_2006, jauregui_lowersemicontinuitymass0_2019, agostiniani_sharpgeometricinequalitiesclosed_2020} for other applications of such connection. It is then natural to ask whether a similar picture holds in the Heisenberg group.
The Pansu sphere is the natural candidate for such a role due to its isoperimetric properties~\cite{RitoreRosales08,Monti08,CKL26-H1}. In fact, characteristic points obstruct self-similar evolution through any right-Lipschitz solution. 
\begin{thm}\label{cor:noselfsim}
Let $\Omega\subset\R^3$ be a bounded open set with $C^2$ boundary satisfying assumption $(A)$ of Theorem~\ref{thm:existence}. Then $\Omega$ cannot evolve self-similarly through a right-Lipschitz (or equivalently left-Lipschitz) solution of the mean-convex horizontal mean curvature flow. In particular, the Pansu sphere does not evolve self-similarly.
\end{thm}
The failure of the Pansu sphere in fulfilling natural isoperimetric-related properties was recently observed also in \cite{fogagnolo-pinamonti-verzellesi}.

At the core of Theorem~\ref{cor:noselfsim} lies the following key observation: if $u\in C^0(\bar \Omega)$ is a solution to \eqref{HMCFprobl} whose level sets $\Gamma_t$ are regular enough, then  near every characteristic point belonging to a positive-time level, the horizontal mean curvature becomes unbounded along a sequence of regular points. We call this phenomenon \emph{instantaneous curvature blow-up} and make it rigorous in Theorem~\ref{thm:hmcfoliation}.

This blow-up can also be explicitly observed along the evolution of the Kor\'anyi sphere.  
This is the boundary of the set $K_R=\{p\in\R^3:k_R(p)<0\}$, $R>0$ where
\begin{equation}\label{eq:koranyi-ball-intro}
    k_R(x,y,z)
    =
    (x^2+y^2)^2+16z^2-R^4.
\end{equation}
In terms of the radial distance $r$, its 
horizontal mean curvature is given by 
\begin{equation}\label{eq:H_Korany}
H_{\partial K_R}
=
\frac{3r}{R^2},
\end{equation}
away from the characteristic set, which consists of the two poles. 
Although $H_{\partial K_R}> 0$ at every non-characteristic point, Theorem~\ref{thm:existence} does not apply to $\Omega= K_R$, since the curvature tends to zero when approaching the characteristic points. 

Nevertheless, in \cite{Ferrariliumanfredi_2014}, the authors show that the unique axisymmetric solution to 
\begin{equation}\label{eq:HMCF-evo-gen}
    \begin{bvprob}
        \partial_t w =\Delta w -\frac{\grad\grad w(\grad w,\grad w)}{|\grad w|^2}
        \cond{in}{\R^3\times(0,+\infty)},
        \\
        w(\cdot,0)=k_R
        \cond{in}{\R^3},
    \end{bvprob}
\end{equation}
understood in the viscosity sense at points where $\grad w=0$, is given by
\begin{equation}\label{eq:koranyi-evo}
    w(p,t)
    =
    (x^2+y^2)^2
    +
    12t(x^2+y^2)
    +
    16z^2
    +
    12t^2
    -
    R^4.
\end{equation}
The horizontal mean curvature of the corresponding evolved level sets $\Sigma_t=\{w(\cdot,t)=0\}$ for $0\leq t<T_R^*$ and $T_R^*=\frac{R^2}{\sqrt{12}}$ is given by 
\begin{equation}\label{eq:koranyi-evo_H}
    H_{\Sigma_t}
    =
    \frac{\partial_t w}{|\grad w_t|}
    =
    \frac{12(x^2+y^2)+24t}{|\grad w_t|}.
\end{equation}
As a point of $\Sigma_t$ approaches one of the characteristic poles, its horizontal gradient vanishes, making $H_t$ blow up for any $t>0$, unlike the case $t=0$. 
Thus, although the horizontal mean curvature of the initial Korányi sphere vanishes at its characteristic points, the evolved hypersurfaces immediately develop curvature blow-up near the corresponding characteristic poles. We point out that \eqref{eq:HMCF-evo-gen} is the general parabolic formulation of the horizontal mean curvature flow~\cite{capognacitti09,Ferrariliumanfredi_2014}, when the initial set is the Kor\'anyi sphere. 

The explicit formula~\eqref{eq:koranyi-evo} reveals another interesting geometric aspect of this blow up. Indeed, setting \(\lambda_t = \sqrt{R^4-12t^2}\),
one obtains
\begin{equation}\label{eq:koranyi-cylas}
    \delta_{1/\lambda_t}(\Sigma_t)
    \longrightarrow
    \mathrm{Cyl}_R
    =
    \left\{
        p\in\R^3:
        r(p)^2=\frac{1}{R^2\sqrt{12}}
    \right\}    \qquad
    \text{as }t\to T^*_R
\end{equation}
locally in Hausdorff distance.
Thus, under this anisotropic rescaling, the asymptotic profile of the Korányi evolution is a vertical cylinder.

The instantaneous curvature blow-up, together with the cylindrical asymptotics of the Korányi evolution and the non-existence of compact right-Lipschitz self-similar solutions established above, leads us to conjecture that  horizontal mean curvature flows starting from a suitable class of compact horizontally mean-convex hypersurfaces and becoming extinct at a single point do not approach a compact self-similar profile. Instead, their asymptotic geometry should be cylindrical. 
In this sense, vertical cylinders, which are shown to be self-similar for the general evolution in~\cite{capognacitti09}, would provide the natural asymptotic models for compact mean-convex horizontal mean curvature flow evolutions in $\bH$.

\subsection{Structure of the paper}
Section~\ref{s:heis} introduces the geometric analysis of the Heisenberg group needed in the paper: the horizontal notation for vectors, the right- and left-invariant sub-Riemannian distances, the horizontal gradient and the Laplacian of a function, and the curvature of hypersurfaces, together with some of their properties. It also collects the definitions of the Riemannian approximants and of viscosity solutions.

Section~\ref{s:axisym} contains the proof of Theorem~\ref{thm:approx}. Its main building block is Proposition~\ref{thm:prescrmcanalit}, which drives the approximation by prescribed mean curvature; this is then combined with a polynomial approximation, carried out in Section~\ref{ss:analyticapprox}.

The proof of Theorem~\ref{thm:existence} is given in Section~\ref{s:existence}. It is based on the Riemannian approximation discussed above, at least in the axisymmetric case (Theorem~\ref{thm:approx}), and relies on two further ingredients: stable gradient estimates for the solutions of the approximating Riemannian problems (Propositions~\ref{stimagradeps} and~\ref{thm:gradest_ednoncar}), and the extraction of a limit viscosity solution (Proposition~\ref{prop:limitsueps}).

Section~\ref{s:blowup} shows that, whenever a horizontal mean curvature flow forms a foliation that is (slightly less than) $C^2$, the horizontal mean curvature of the evolving sets blows up at characteristic points (Theorem~\ref{thm:hmcfoliation}). This result leads to the proof of Theorem~\ref{cor:noselfsim}.

\subsection{Acknowledgments} The authors are members of INdAM--GNAMPA. The first and second authors are supported by the INdAM--GNAMPA projects ``ANGELS'' (CUP code E53C25002010001) and ``GIANTS'' (CUP code E53C23001740001).
The second author is also supported by the STARS project ``DEFORM'' of the University of Padova (CUP code C93C25004860003).
The third author is also affiliated with the Padova Neuroscience Center (PNC), University of Padova, and is supported by the French National Research Agency (ANR) through the project ``SCOPE''.

The authors are grateful to Luca Capogna for his interest in this work and for helpful discussions, to Pierre Pansu for an inspiring discussion about the Koranyi evolution, and to Simone Verzellesi for pointing out an error in an argument in an earlier version of the paper.

\subsection{AI Statement}
The present paper is part of the Ph.D. project of the first author. All arguments presented herein were conceived and developed entirely by the authors. Large language models (ChatGPT and Claude) were used solely to assist with proof verification, the generation of TikZ figures, and linguistic editing.

\section{The Heisenberg group}\label{s:heis}
\subsection{Horizontal distribution}\label{ss:horizdistr}
The \emph{horizontal distribution} $\mathcal{D}$ defined in~\eqref{eq:campi} is bracket-generating of step $2$, i.e., the whole $\R^3$ is generated by commutators of $\mathcal D$ of length one and two, as the commutation relation
$[X,Y] = Z$ holds, where $Z = \partial_z$. Moreover, this is the only non-trivial commutation relation.

\begin{rem}[Properties of the horizontal distribution]\label{rem:propertiesofD}
    Horizontal vector fields are
    \begin{enumerate}[label=\roman*.]
        \item {\em left-translations invariant} with respect to the Lie group operation
        \begin{equation}
        \label{eq:oper}
        (x,y,z)*(x',y',z')=\left(x+x',y+y',z+z'+\frac{1}{2}(xy'-x'y)\right).
        \end{equation}
        That is, $V(f(p*q))=(Vf)(p*q)$ for any $V\in\mathcal D$ and $p,q\in\R^3$.
        We call \emph{vertical translation} any map $q\in\mathbb R^3\mapsto (0,0,z)*q$ for $z\in \R$.
        \item {\em $1$-homogeneous} with respect to the family of anisotropic  dilations defined in formula~\eqref{eq:dilations}
        That is, $V(u\circ\delta_\lambda)=\lambda (Vu)\circ \delta_\lambda$  for any $u\in C^1(\R^3)$.
    \end{enumerate}
    We explicit how horizontal vector fields change under rotations. For $\theta\in\mathbb S^1$ let
        \begin{equation}\label{eq:rotations}
            \mathcal{R}_\theta(x,y,z) = (x\cos\theta - y\sin\theta, x\sin\theta + y\cos\theta, z)
        \end{equation}
        denote the rotation about the $z$-axis of angle $\theta$. Given any $u\in C^1(\R^3)$ and $\theta\in\mathbb S^1$, we have
        \begin{align}
            X(u\circ \mathcal{R}_\theta) = \cos\theta\,(Xu)\circ \mathcal{R}_\theta+\sin\theta\,(Yu)\circ \mathcal{R}_\theta,\\
            Y(u\circ \mathcal{R}_\theta) = -\sin\theta\,(Xu)\circ \mathcal{R}_\theta+\cos\theta\,(Yu)\circ \mathcal{R}_\theta.
        \end{align}
        Moreover, the vertical field $Z=\partial_z$ satisfies $Z(u\circ \mathcal{R}_\theta)=(Zu)\circ \mathcal{R}_\theta$ (i.e., it is \emph{rotation invariant}).
\end{rem}

\subsection{Sub-Riemannian distance} 
A curve $\gamma:[0,1]\to\R^3$ is said to be {\em horizontal} if it is absolutely continuous and $\dot\gamma(t)\in\mathcal D_{\gamma(t)}$ for a.e.\ $t\in [0,1]$. Since being horizontal is preserved under monotone reparametrization, this normalization to $[0,1]$ entails no loss of generality.
The {\em sub-Riemannian distance} from $p\in\R^3$ to $q\in\R^3$ is defined as
\begin{equation}
    \label{eq:CCdist}
    d(p,q)=\inf\left\{\int_0^1 |\dot\gamma(t)|\,dt \;\middle|\; 
    \gamma:[0,1]\to\R^3 \text{ is horizontal, } \gamma(0)=p,\ \gamma(1)=q \right\}.
\end{equation}
Since $\mathcal D$ is bracket generating, the Chow--Rashevskii theorem implies that 
$d:\R^3\times\R^3\to [0,+\infty]$ is finite and induces the Euclidean topology on $\R^3$ 
(see, e.g., \cite[Theorem~3.31]{Agrachev_Barilari_Boscain_2019}). 
By Remark~\ref{rem:propertiesofD}, the distance $d$ is
\begin{enumerate}[label=\roman*.]
    \item left-translations invariant with respect to the group operation~\eqref{eq:oper}, i.e., $d(q*p,q*p')=d(p,p')$ for all $p,p'\in\R^3$; 
    \item $1$-homogeneous with respect to the family of  anisotropic  dilations $\{\delta_\lambda\}_{\lambda>0}$, i.e., $d(\delta_\lambda p, \delta_\lambda p')=\lambda d(p,p')$ for all $\lambda>0$, $p,p'\in\R^3$.
\end{enumerate}

\begin{rem}
The sub-Riemannian distance is not equivalent to the Euclidean one. The following estimates hold: for every compact set $K\subset\R^3$ there exists a constant $C_K>0$ such that
\begin{equation}
    C_K^{-1}|p-q|_{\R^3}\le d(p,q)\le C_K |p-q|_{\R^3}^{\frac{1}{2}},\quad \forall p,q\in K,
\end{equation}
see, e.g., \cite[Proposition~5.15.1]{BLU07}.
\end{rem}

\subsection{Horizontal gradient, Horizontal Hessian and Heisenberg sub-Laplacian}\label{ss:Hoperators} Let $u\in C^1(\R^3)$. We call {\em horizontal gradient} of $u$ the vector field 
\begin{equation}
\label{eq:grad}
    \grad u=(Xu) X+(Yu) Y\in \mathcal D.    
\end{equation}
If $u\in C^2$, the {\em horizontal Hessian of $u$} is the quadratic form on $\mathcal D\times\mathcal D$ represented, with respect to the basis $\{X,Y\}$ by the matrix
\begin{equation}
    \label{eq:hessian}
    \hess u=\begin{pmatrix}
        XXu & XYu\\ YXu &YYu
    \end{pmatrix}.
\end{equation}
The {\em Heisenberg sub-Laplacian} of $u$ is defined as the trace of the horizontal Hessian 
\begin{equation}
    \Delta u=\operatorname{tr}(\hess u)=X^2u+Y^2u,
\end{equation}
and it is related to the horizontal gradient via the identity $\Delta u=\operatorname{div}(\grad u)$, where the divergence is computed with respect to the Lebesgue measure.
\begin{rem}\label{rem:invariance}
By Remark~\ref{rem:propertiesofD}, the horizontal gradient, the horizontal Hessian, and the Heisenberg sub-Laplacian are left-invariant with respect to the group law~\eqref{eq:oper} and also invariant under rotations about the $z$-axis. 
Moreover, the vertical vector field $Z=\partial_z$ is also left-invariant and invariant under rotations about the $z$-axis.
\end{rem}

\subsection{Right-invariant structure}\label{ss:rightinv}
In this paper, {\em right-invariant structure} will also come into play. The Lie algebra of right-invariant vector fields with respect to the group law~\eqref{eq:oper} is generated by $\tilde X$ and $\tilde Y$ defined in~\eqref{eq:campi_right} and by $Z=\partial_z$, with the only nontrivial commutation relation given by $[\tilde X,\tilde Y]=-Z$. Observe that left-invariant and right-invariant vector fields commute. In particular,
\begin{equation}\label{eq:vfcomm}
    [X,\tilde{X}] = [X,\tilde{Y}] = [Y,\tilde{X}] = [Y,\tilde{Y}] = 0.
\end{equation}

We define the right-invariant distance $\tilde d$, by substituting horizontal curves in~\eqref{eq:CCdist} with curves that are a.e.\ tangent to $\tilde{\mathcal D} $. {We similarly put a $\ \tilde{}\ $ on any other quantity associated to the right invariant distribution, in full analogy with the previous subsections.} 

For a regular function $u$, we introduce the notation
\begin{equation}
    \tilde \grad u=\tilde X u\, X+\tilde Yu\, Y.
\end{equation}

\subsection{Hypersurfaces in the Heisenberg group and mean curvature} \label{sec:hypersurfaces}
The {\em set of characteristic points} of a $C^1$-regular surface $\mathcal S$ is 
$\mathcal C(\mathcal S)=\{p\in\mathcal S : \mathcal D_p\subset T_p\mathcal S\}$.
If $\mathcal S$ is given as the zero level set of $u \in C^k(\mathbb{R}^3)$, $k\ge 1$, i.e., $\mathcal S=\{p\in\R^3 : u(p)=0\}$ with $\grad_{E} u\neq 0$, (here the gradient is computed with respect to the Euclidean structure), then we have
\begin{equation}
\label{eq:char_u}
\mathcal{C}(\mathcal S) = \{p \in \mathcal S \mid X u(p) = Y u(p) = 0\}.
\end{equation}

\begin{rem}
\label{rem:frobeniuschar}
Since the horizontal distribution $\mathcal D$ is bracket generating, the characteristic locus $\mathcal C(\Sigma)$ of any $C^2$-hypersurface $\Sigma$ has empty interior in $\Sigma$, so that non-characteristic points are dense in $\Sigma$; see~\cite[p.\,65]{Balogh2003}. The same holds for the right-invariant distribution $\tilde{\mathcal D}$, which is also bracket generating.
\end{rem}

Let $\Omega=\{p\in\R^3 \mid u(p)>0\}$ with $u\in C^k(\R^3)$, $k\ge 1$ and $\grad_{E} u(p)\neq 0$ for any $p \in \partial \Omega$. In this case we say that $\Omega$ is given as the (regular) zero super-level set of $u$ and we define the {\em horizontal outer unit normal $\nu$ to $\Omega$} as
\begin{equation}
\label{eq:nu_u}
\nu = -\frac{\grad u}{\lvert \grad u \rvert}
\quad \text{on } \partial\Omega \setminus \mathcal{C}(\partial\Omega).
\end{equation}
 One can check that this definition does not depend on the choice of $u$, as one can infer  e.g. from~\cite[Theorem~1.1]{ArcozziFerrari07}).

If $\Omega$ is the zero super-level set of $u\in C^2(\R^3)$, we define the {\em horizontal mean curvature of $\partial\Omega$} at noncharacteristic points $p \in \partial\Omega \setminus \mathcal{C}(\partial \Omega)$ as
\begin{equation}
\label{eq:H}H=\operatorname{div}(\nu)
    =-\frac{1}{|\grad u|}\left(\Delta u
    -\frac{\hess u\left(\grad u,\grad u\right)}{|\grad u|^2}\right)
    \quad \text{at } p \in \partial\Omega \setminus \mathcal{C}(\partial\Omega).
\end{equation}
\begin{defn}[Horizontal mean convexity and continuity of $H$]
\label{def:hormeanconve}
Let $\Omega$ be a bounded open set with $C^2$ boundary, and let $H$ be defined as in \eqref{eq:H} at noncharacteristic point.
We say that $\partial \Omega$ has \emph{strictly positive horizontal mean curvature} if $H > c >0$ at any $p \in \partial\Omega \setminus \mathcal{C}(\partial \Omega)$. In this case we say that $\partial \Omega$ is \emph{horizontally mean convex}.

Moreover, we say that $H \in C^0(\partial \Omega)$ if $H$ extends to a continuous function on the whole $\partial \Omega$.
If $\partial \Omega$ happens to be noncharacteristic, then  $H\in C^0{(\partial \Omega)}$.
\end{defn}


\subsection{Riemannian approximation} \label{ss:Riemapprox}

For $\varepsilon>0$, consider the Riemannian manifold $\bH_\varepsilon=(\R^3, g_\eps)$, where $g_\eps$ is the unique Riemannian metric that makes the vector fields~\eqref{eq:campi_eps} orthonormal.
{We denote by $d_\varepsilon$ distance induced by $g_\eps$. We record the following well-known observation. 


\begin{prop}\label{prop:deps-d}
    Let $\varepsilon>0$, and $p,q\in\R^3$. Then $d_\varepsilon(p,q)\le d(p,q)$. Similarly, $\tilde{d}_\varepsilon(p,q)\le \tilde{d}(p,q)$. 
\end{prop}
\begin{proof}
Recall that an absolutely continuous curve $\gamma=(\gamma_x,\gamma_y,\gamma_z)$ on $[0,1]$ is horizontal if and only if $\dot \gamma_z=\frac12(\gamma_x\dot \gamma_y-\gamma_y\dot \gamma_x)$ a.e.\ on $[0,1]$. If $\gamma$ is horizontal, then
\begin{equation}
    |\dot\gamma|_\varepsilon
    =\left|\left(\dot\gamma_x,\dot\gamma_y,0\right)\right|_{\R^3}
    =|(\dot\gamma_x,\dot\gamma_y)|_{\R^2}=|\dot\gamma|,
\end{equation}
and the conclusion follows from the definition of the sub-Riemannian distance~\eqref{eq:CCdist}.
\end{proof}}

We use the subscript $\varepsilon$ to denote the Riemannian quantities norms and tensors induced by $g_\eps$. In particular, $\gradeps$ denotes the gradient, $\hesseps$ the Hessian and $\Delta_\eps$ the Laplace-Beltrami operator. When considering the norm of a $g_\eps$ gradient vector field, we agree to denote $|\nabla_\eps f|^2 = g_\eps(\nabla_\eps f, \nabla_\eps f)$, that is explicitly given by $(Xf)^2 + (Yf)^2 + \eps^2(Zf)^2$.
Accordingly, the $g_\eps$-mean curvature of $\partial \Omega$, with $\Omega$ given as the zero super-level set of $u\in C^2$ with $|\nabla_\eps u| \neq 0$ on $\partial \Omega$ is given by 
\begin{equation}
    \label{eq:Heps}
    H_\varepsilon
=-\operatorname{div}\left(\frac{\gradeps u}{|\gradeps u |_\eps}\right)
    =-\frac{1}{|\gradeps u|_\eps}\left(\Delta_\varepsilon u-\frac{\hesseps u\left(\gradeps u,\gradeps u\right)}{|\gradeps u|_\eps^2}\right) \quad \text{at} \,\, p  \in \partial \Omega.
\end{equation}
One can finally check that, if $f_\eps \to f$ in $C^2$, then $\nabla_\eps f_\eps \to \nabla f$ and $\nabla_\eps\nabla_\eps f_\eps(\nabla_\eps f_\eps, \nabla_\eps f_\eps) \to \nabla\nabla f(\nabla f, \nabla f)$ uniformly. In particular, one deduces that 
\[
\operatorname{div}\left(\frac{\gradeps f_\eps}{|\gradeps f_\eps |}\right) \to \operatorname{div}\left(\frac{\nabla f}{|\nabla f|}\right)
\]
locally uniformly in $\{|\nabla f| \neq 0\}$. As a consequence, given a bounded domain $\Omega$ with $C^2$ boundary, we have that the mean curvatures satisfy 
\begin{equation}\label{eq:HepstoH}
    H_\eps \to H \text{ as } \eps \to 0 \text{ locally uniformly in } \partial \Omega \setminus \mathcal{C}(\partial \Omega).
\end{equation} 
This last observation is also recorded in~\cite[Corollary~4.21]{capdanieltys2007}. We are repeatedly using these simple observations without explicitly mentioning it.
In full analogy one can consider Riemannian approximations for the sub-Riemannian structure given by the right-invariant distribution described above. Namely, one can define $\tilde{g}_\eps$ as the unique Riemannian metric on $\R^3$ such that the vector fields $\tilde X, \tilde Y$ and $\eps Z$ are orthonormal. We define accordingly $\tilde\nabla_\eps f$ the $\tilde g_\eps$-gradient of a function $f$ and agree to denote $|\tilde\nabla_\eps f|^2 = \tilde g_\eps(\tilde\nabla_\eps f, \tilde\nabla_\eps f)$.
 The following  basic  observation will crucially allow us to handle axisymmetric sets in the a priori estimates below, thereby allowing us to establish Theorem \ref{thm:existence} under these assumptions. 

 \begin{rem}\label{rem:rotinvariancegrad}
Let $u\in C^1(\R^3)$ be invariant under rotations, i.e., $u\circ R_\theta=u$ for all $\theta\in\mathbb S^1$, where $R_\theta$ is defined in~\eqref{eq:rotations}. Then, we have
\begin{equation}\label{eq:rotinvar}
     |\nabla u|=|\tilde\nabla u|,    
\end{equation}
and consequently, one also gets
\begin{equation}\label{rem:rotinvariancegradeps}
    |\gradeps u|=|\tilde\nabla_\eps u|.
\end{equation}
\end{rem}
\subsection{Viscosity solutions to~\eqref{eq:HMCFeq} and \eqref{HMCFprobl}}

The formal derivation described in the introduction leads us to the level-set formulation of the horizontal mean-convex mean curvature flow equation~\eqref{eq:HMCFeq}.
To include characteristic initial data in the study of horizontal mean curvature flow evolutions we will study viscosity solutions to \eqref{HMCFprobl}.
\begin{defn}\label{dfn:viscsol}
    A function $u\in C^0(\Omega)$ is a \textit{viscosity subsolution} of \eqref{eq:HMCFeq} in $\Omega\subset\R^3$ if, for any $p\in\Omega$ and any $\Phi\in C^2(\Omega)$ such that $u-\Phi$ attains a local maximum at $p$, then
    \begin{equation}\label{eq:subsol}
    \begin{cases}
        1\ge-\Delta \Phi+\frac{\hess \Phi(\grad \Phi, \grad \Phi)}{|\grad \Phi|^2}\ \ \ \ \mbox{at } p\\
        \mbox{if } \ \grad \Phi(p)\not =0
    \end{cases}
    \end{equation}
    and 
    \begin{equation}\label{eq:subsol_eta}
    \begin{cases}
        1\ge-\Delta \Phi+\hess \Phi(\eta, \eta)\ \ \ \ \mbox{at } p\\
        \mbox{for some } \eta\in\mathcal D\mbox{ with } |\eta|\le1, \mbox{ if } \grad \Phi(p) =0.
    \end{cases}
    \end{equation}
    A function $u\in C^0(\Omega)$ is a \textit{viscosity supersolution} of \eqref{eq:HMCFeq} in $\Omega\subset\R^3$ if, for any $p_0\in\Omega$ and any function $\Phi\in C^2(\Omega)$ such that $u-\Phi$ attains a local minimum at $p$ then
    \begin{equation}\label{eq:supersol}
    \begin{cases}
        1\le-\Delta \Phi+\frac{\hess \Phi(\grad \Phi, \grad \Phi)}{|\grad \Phi|^2}\ \ \ \ \mbox{at } p\\
        \mbox{if } \ \grad \Phi(p)\not =0
    \end{cases}
    \end{equation}
    and 
    \begin{equation}\label{eq:supersol_eta}
    \begin{cases}
        1\le-\Delta \Phi+\hess \Phi(\eta, \eta)\ \ \ \ \mbox{at } p\\
        \mbox{for some } \eta\in\mathcal D \mbox{ with } |\eta|\le1, \mbox{ if } \grad \Phi(p) =0.
    \end{cases}
    \end{equation}
    A function $u \in C^0({\Omega})$ is a \textit{viscosity solution} to equation~\eqref{eq:HMCFeq} if it is both a viscosity supersolution and a viscosity subsolution. A function $u \in C^0(\bar{\Omega})$ is a {\em viscosity solution} to \eqref{HMCFprobl} if it is a viscosity solution of \eqref{eq:HMCFeq} in $\Omega$, and it satisfies the boundary condition continuously:
    \begin{equation}
        \lim_{x \to p} u(x) = 0 \quad \forall p \in \partial\Omega.
    \end{equation}
\end{defn}

\subsubsection{Approximating solutions}\label{ss:approxsol} Our strategy to prove the existence of solutions to~\eqref{HMCFprobl} (see Theorem~\ref{thm:existence}) is based on an approximation procedure. 
We first apply Theorem~\ref{thm:approx} to approximate $\Omega$ by a family of sets $\{\Omega_\varepsilon\}_{\varepsilon>0}$, each of which has strictly mean-convex boundary in $\bH_\varepsilon$. 
The sets $\Omega_\varepsilon$ are then used as initial domains for a family of Riemannian mean curvature flows with mean-convex boundary, which approximate~\eqref{HMCFprobl}. More precisely, the prescribed mean curvature equation~\eqref{eq:HMCFeq} is approximated by the corresponding equation in $\bH_\varepsilon$:
\begin{equation}
    \label{RepsMCFeq}
       -\Delta_\eps u_\eps + \frac{\grad_\eps \grad_\eps u_\eps(\grad_\eps u_\eps, \grad_\eps u_\eps)}{|\grad_\eps u_\eps|^2} = 1.
\end{equation}
The resulting approximating problems correspond to the level set formulation for the mean curvature flow on $\Omega_\varepsilon$ whose solutions in a viscosity sense are denoted by $u_\varepsilon\in C^0(\overline \Omega_\varepsilon)$:
\begin{equation}\label{RepsMCFprobl}\tag{MCF$_\eps$}
\begin{bvprob}
   -\Delta_\eps u_\eps + \frac{\grad_\eps \grad_\eps u_\eps(\grad_\eps u_\eps, \grad_\eps u_\eps)}{|\grad_\eps u_\eps|^2} = 1 \cond{in }{\Omega_\varepsilon}, \\
   u_\eps = 0 \cond{on}{\partial \Omega_\varepsilon.}
   \end{bvprob}
\end{equation}
 Viscosity solutions to~\eqref{RepsMCFprobl} are obtained as in \cite{Schulze2008} as the limit as $\delta\to 0^+$ of the unique solution $u_\varepsilon^\delta$ of the following elliptic regularized problem:
\begin{equation}\label{RepsMCFproblreg}\tag{MCF$_\eps^\delta$}
    \begin{bvprob}
  -\Delta_\eps u_\eps^\delta + \frac{\grad_\eps \grad_\eps u_\eps^\delta(\grad_\eps u_\eps^\delta, \grad_\eps u_\eps^\delta)}{|\grad_\eps u_\eps^\delta|^2+\delta^2} = 1 \cond{in}{\Omega_\varepsilon}, \\
   u_\eps^\delta = 0 \cond{on}{\partial \Omega_\varepsilon}.
\end{bvprob}
\end{equation}

To construct a solution of~\eqref{HMCFprobl}, we first obtain a family
of solutions \(\{u_\varepsilon^\delta\}_{\varepsilon,\delta>0}\) to
the elliptic problems~\eqref{RepsMCFproblreg}, and then pass to the
limits as \(\delta\to0\) and \(\varepsilon\to0\). To this end, we
establish suitable gradient bounds for \(u_\varepsilon^\delta\), which
ensure the uniform boundedness and equicontinuity of the family in
\(C^0(\bar{\Omega})\).

This is the main technical step in which our assumptions are used. 
First, a Bochner-type identity based on \cite[Proposition~3.4]{Garofalo23} allows to obtain \emph{right-invariant gradient estimates} for \(u_\varepsilon^\delta\). Such identity is based on the commutation relations between the left-invariant distribution \(\mathcal D\) and the right-invariant
distribution \(\tilde{\mathcal D}\): a similar behavior is not expected to hold for the left-invariant gradient $\nabla$ and explains the role of $\tilde d$. 

Axisymmetry in the characteristic case is crucial for converting the right-invariant gradient estimate into a left-invariant one by Remark~\ref{rem:rotinvariancegrad}, allowing to exploit the boundary information $H_\eps |\nabla_\eps u_\eps| = 1$ on $\partial \Omega_\varepsilon$ coupled with the $g_\eps$-mean convexity of $\partial \Omega_\varepsilon$ to get a gradient bound on $\overline{\Omega_\varepsilon}$ that is uniform in $\eps$.

In the non-characteristic case  we obtain an upper bound on the norm of right-invariant gradient, arguing similarly to the Euclidean case treated in \cite{evansspruck91}. Indeed, comparing $u_\varepsilon^\delta$ with an explicit barrier built from the Euclidean distance $d_E$ to $\partial\Omega_\varepsilon$, via the comparison principle for the regularized operator acting on it from~\cite[Proposition~4.3]{FernandezRealRosOton22}, one obtains a bound on $\|\tilde\nabla_\varepsilon u_\varepsilon^\delta\|_{L^\infty(\partial\Omega_\varepsilon)}$ uniform in $\varepsilon$ and $\delta$. Here the left-invariant conversion step is unnecessary: it is the right-invariant distance $\tilde d$, rather than $d$, that governs the compactness argument in this case.

\begin{rem}
By Proposition~\ref{lem:topologiaaxsymm}, regular axisymmetric sets are topologically either spheres or tori. Theorem~\ref{thm:approx} deals only with the spherical case, since axisymmetric tori are non-characteristic by Corollary~\ref{cor:axisimcarpoint}. 

Consequently, by \eqref{eq:HepstoH}, $\partial\Omega$ is $g_\eps$-strictly mean convex for $\varepsilon>0$ sufficiently small, and thus serves as admissible initial data for the mean-convex mean curvature flow~\eqref{RepsMCFprobl} (cf.~\cite{Schulze2008}).
\end{rem}

\section{Approximation by prescribed mean curvature and proof of theorem \ref{thm:approx}}\label{s:axisym}
In this section we collect some facts on axisymmetric sets and prove Theorem~\ref{thm:approx}.

\subsection{$C^2$-regular connected axisymmetric sets} 
Axisymmetric connected sets are those generated by rotation of a $C^k$-curve around the $z$-axis.

\begin{defn}\label{defn:gencurve}
Let $\gamma=(\gamma_1,\gamma_2)\in C^k([0,1],\R^2_+)$, $k\in\mathbb N$ be a regular, simple curve.
Recalling the rotation $R_\theta$ defined in~\eqref{eq:rotations} for $\theta\in\mathbb S^1$, we define the rotation map $\Phi_\gamma:[0,1]\times\mathbb S^1\to\R^3$ as 
\begin{equation}\label{eq:mappasupedirot}
    \Phi_\gamma(\tau,\theta) = R_\theta(\gamma_1(\tau),0,\gamma_2(\tau)),
\end{equation}

We say that $\gamma$ is a {\em generating profile} for a bounded, open, connected set  $\Omega \subset \R^3$ with $C^k$-regular boundary, $k\ge 1$ if  $\partial \Omega=\Phi_\gamma([0,1]\times\mathbb S^1)$.
In this case, $\Omega$ is called \emph{$C^k$-regular axisymmetric}.
\end{defn}

\begin{rem}
    Every $C^k$-regular axisymmetric set admits a $C^k$ generating profile, obtained by intersecting its boundary with the half-plane $\{y=0, \ x\ge 0\}$. While the profile of the curve can be re-parametrized, its regularity class is an {intrinsic property}.
\end{rem}

In the following proposition we discuss the topology of $C^k$-regular axisymmetric sets.
\begin{prop}\label{lem:topologiaaxsymm}
Let $\Omega\subset\R^3$ be a $C^1$-regular open connected and bounded set which is axisymmetric.
Then, if $\partial\Omega$ intersects the $z$-axis, it is homeomorphic to a sphere, otherwise it is homeomorphic to a torus.
\end{prop}

\begin{rem} If $\partial\Omega$ is  not of class $C^1$, different topologies may appear. This is the case of a \textit{pinched torus}, which is given by the rotation around the $z$-axis of a circle touching the origin. This is the boundary of a compact axisymmetric set, but it has a singularity on the rotation axis and fails to be homeomorphic to a sphere or to a torus. 
\end{rem}
\begin{proof}
Let $\gamma\in C^1([0,1])$ be a generating profile for $\Omega$. 
If $\gamma$ does not meet the $z$-axis, then $\gamma$ is a closed simple curve. Hence it is homeomorphic to $\mathbb S^1$ so that $\partial \Omega$ is homeomorphic to $\mathbb S^1\times\mathbb S^1$.

If $\gamma$ meets the $z$-axis, regularity of $\partial\Omega$ implies that the profile meets the axis only at its two endpoints. Therefore $\partial\Omega$ is obtained from $[0,1]\times\mathbb S^1$ by collapsing each boundary circle $\{0\}\times\mathbb S^1$ and $\{1\}\times\mathbb S^1$ to a point. This quotient is therefore homeomorphic to $\mathbb S^2$. 
\end{proof}

\begin{cor}\label{cor:axisimcarpoint}
    Let $\Omega \subset \R^3$ be an open, $C^1$-regular, axisymmetric set. Then
    \begin{enumerate}[label=\roman*.]
        \item if $\partial\Omega$ is homeomorphic to a sphere, then $\mathcal C(\partial\Omega)$ consists of exactly two points where the surface meets the $z$-axis;
        \item if $\partial\Omega$ is homeomorphic to a torus, then $\mathcal C(\partial\Omega)=\emptyset$.
    \end{enumerate}
\end{cor}
\begin{proof}
By Proposition~\ref{lem:topologiaaxsymm}, the statement follows once we prove that 
    \begin{equation}
        \mathcal C(\partial \Omega) = \partial \Omega \cap \{(x,y,z) \in \mathbb{R}^3 \mid x = 0,\, y = 0 \}.
    \end{equation}
In fact, if $p=(x,y,z)\in\partial\Omega$ satisfies $\mathcal D\subset T_p\partial\Omega$, locally around $p$, the boundary $\partial \Omega$  can be represented as the graph of a function $z=f(x,y)$. Moreover, being $\Omega$ axisymmetric, $f$ only depends on $r=r(x,y)$. 
Characteristic points are therefore those at which $|\grad (z-f(r(x,y)))|^2=f'(r)^2 + \frac{r^2}{4}=0$. This occurs if and only if $r=0$ and $f'(0)=0$.
\end{proof} 


\begin{rem}[Curvatures of axisymmetric $z$-graphs]
    Let \(f\in C^2(I,\R)\), with \(I\subset [0,\infty)\) an open interval. 
If \(\partial\Omega\) is the zero level-set of the function 
\[
u=z-f\circ r \in C^2(\mathbb R^3),
\]
then a direct computation shows that the horizontal mean curvature and the $g_\varepsilon$-mean curvature of \(\partial\Omega\) are respectively given by \(H=\varphi\circ r\), and \(H_\varepsilon=\varphi_\varepsilon\circ r\), where
\begin{align}
    \label{eq:MCHeis}
    \varphi(r)&=-\frac{\sigma}{r}\,\partial_r\left(\frac{r f'(r)}{\sqrt{f'(r)^2+\frac{r^2}{4}}}\right),\\
    \label{eq:MCHeis-eps}
    \varphi_\varepsilon(r)&=-
    \frac{\sigma}{r} \, \partial_r\left( \frac{r f'(r)}{\sqrt{f'(r)^2 + \frac{r^2}{4} + \varepsilon^2}} \right).
\end{align}
Here $\sigma=1$ if $\Omega=\{u<0\}$ and $\sigma=-1$ if $\Omega=\{u>0\}$.
While $\varphi$ is not defined at $r=0$ by Corollary~\ref{cor:axisimcarpoint}, in the definition of $\varphi_\varepsilon$ the value at $r=0$ is intended in the limit as $r\to 0^+$.
\end{rem}


\subsection{Approximation by prescribed mean curvature.}
Throughout this section, let $\Omega\subset\mathbb R^3$ be a $C^2$-axisymmetric set, homeomorphic to a sphere, such that $H(\partial\Omega)$ extends continuously to $\mathcal C(\partial\Omega)$. By Proposition~\ref{lem:topologiaaxsymm} and Corollary~\ref{cor:axisimcarpoint}, there exists a continuous function $\varphi$ on $(0,R]$ such that $H=\varphi\circ r$ on $\partial\Omega\setminus\mathcal C(\partial\Omega)$. The continuity of $H$ across $\mathcal C(\partial\Omega)$ allows us to extend $\varphi$ continuously to $[0,R]$, with the same notation. Henceforth, we write 
\begin{equation}
\label{eq:notationphi}
    H=\varphi\circ r\text{ on }\partial\Omega.
\end{equation}
The arguments of this section do not require any mean convexity assumption of $\Omega$.

\subsubsection{The case of a finite vertical set of the generating profile.} 
We first carry out the construction under the assumption that the
generating profile has \emph{finitely many vertical points}.
\begin{defn}
    Given a $C^1$ curve $\gamma=(\gamma_1,\gamma_2):[0,1]\to \mathbb{R}^2_+$, we define the \emph{$\mathcal{V}$-set} as
    \[
    \mathcal{V}(\gamma)= \left\{ \tau\in [0,1] \mid \gamma_1'(\tau)=0 \right\}.
    \]
    Accordingly, the \emph{vertical part} of the boundary of the $C^1$-regular axisymmetric set $\Omega$ generated by $\gamma$ is defined as 
    \[
      \mathcal{V}(\partial\Omega) = \left\{ \Phi_\gamma(\tau,\theta) \mid \tau\in \mathcal{V}(\gamma),\ \theta\in\mathbb S^1 \right\},
    \]
    where $\Phi_\gamma$ is the parametrization given in \eqref{eq:mappasupedirot}.
\end{defn}

\begin{rem}
    The vertical part $\mathcal{V}(\partial\Omega)$ is independent of the parametrization of the generating profile, since regular reparametrizations multiply $\gamma_1'$ by a non-zero factor. Geometrically, $\mathcal{V}(\partial\Omega)$ is the set of points where the tangent plane to $\partial \Omega$ is parallel to the vertical axis.
\end{rem}

\begin{figure}
    \centering
    \begin{tikzpicture}[scale=.9]

\draw[->] (-0.5,0) -- (6,0);
\draw[->] (0,-4) -- (0,5);

\coordinate (ga0) at (0,4);
\coordinate (ga1) at (3.15,2.75);
\coordinate (ga2) at (2.1,1.9);
\coordinate (R)   at (5.5,0);
\coordinate (c41) at (4.15,-1.25);
\coordinate (c42) at (2.85,-0.85);
\coordinate (ga4) at (1.5,-1.5);
\coordinate (ga5) at (3.5,-2.7);
\coordinate (ga6) at (0,-3.5);

\draw[thick]
(ga0)
.. controls (1.2,4) and (3.15,3.3) .. (ga1)

.. controls (3.15,2.2) and (2.1,2.45) .. (ga2)

.. controls (2.1,1.35) and (5.5,0.7) .. (R)

.. controls (5.5,-0.7) and (4.7,-1.25) .. (c41)
.. controls (3.6,-1.25) and (3.35,-0.85) .. (c42)
.. controls (2.35,-0.85) and (1.5,-1.0) .. (ga4)

.. controls (1.5,-2.0) and (3.5,-2.2) .. (ga5)

.. controls (3.5,-3.2) and (1.2,-3.5) .. (ga6);

\fill (ga0) circle (2pt);
\fill (ga1) circle (2pt);
\fill (ga2) circle (2pt);
\fill (R)   circle (2pt);
\fill (ga4) circle (2pt);
\fill (ga5) circle (2pt);
\fill (ga6) circle (2pt);

\fill (c41) circle (1.6pt);
\fill (c42) circle (1.6pt);

\draw[dashed] (3.15,0) -- (ga1);
\draw[dashed] (2.1,0)  -- (ga2);
\draw[dashed] (1.5,0)  -- (ga4);
\draw[dashed] (3.5,0)  -- (ga5);

\node[below right] at (c41) {$\gamma_2'=0$};
\node[above left]  at (c42) {$\gamma_1'=0$};

\node[left]  at (ga0) {$\gamma(\alpha_0)$};
\node[right] at (ga1) {$\gamma(\alpha_1)$};
\node[left]  at (ga2) {$\gamma(\alpha_2)$};
\node[below] at (R)   {$a_3=a_{j_R}=R$};
\node[left]  at (ga4) {$\gamma(\alpha_4)$};
\node[right] at (ga5) {$\gamma(\alpha_5)$};
\node[left]  at (ga6) {$\gamma(\alpha_6)$};

\node[below] at (0,0)    {$a_0=0$};
\node[below] at (3.15,0) {$a_1$};
\node[below] at (2.1,0)  {$a_2$};
\node[below] at (1.5,0)  {$a_4$};
\node[below] at (3.5,0)  {$a_5$};

\node[above] at (2.3,3.6)    {$\Sigma_1$};
\node[above] at (2.9,1.7)    {$\Sigma_2$};
\node[above] at (4,1)        {$\Sigma_3=\Sigma_{j_R}$};
\node[below] at (3.8,-1.35)  {$\Sigma_4$};
\node[above] at (2.9,-1.9)   {$\Sigma_5$};
\node[above] at (1.5,-4)     {$\Sigma_6$};

    \end{tikzpicture}
    \caption{Decomposition of the generating profile into the graphical
    pieces \(\Sigma_j\).}
    \label{fig:decomposition}
\end{figure}
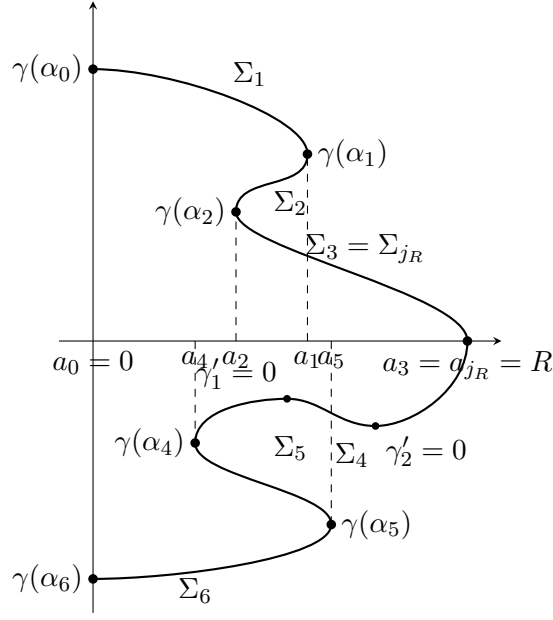

The construction of \(\{\Omega_\varepsilon\}_{\varepsilon>0}\) begins
by decomposing \(\partial\Omega\) at the vertical points of its generating profile, as illustrated
in Figure~\ref{fig:decomposition}. Each resulting portion is a radial
graph, whose horizontal mean curvature is given by~\eqref{eq:MCHeis},
with prescribed curvature \(\varphi\) as in~\eqref{eq:notationphi} and
a sign determined by the local graphical parametrization. These pieces
will serve as local building blocks for the approximating boundaries,
obtained by solving the corresponding prescribed \(g_\varepsilon\)-mean
curvature equation~\eqref{eq:MCHeis-eps} with
\(\varphi_\varepsilon=\varphi\). The following lemma makes this
decomposition precise.

\begin{lem}[Radial decomposition of the boundary]
\label{lem:decomposition}
Let \(\Omega\) be a bounded, open, \(C^2\), axisymmetric set,
homeomorphic to a sphere, such that \(H(\partial\Omega)\) extends
continuously to \(\mathcal C(\partial\Omega)\). Let
\(\gamma\in C^2([0,1],\R^2_+)\) be a generating profile for
\(\partial\Omega\), with \(\#\mathcal V(\gamma)<\infty\), and assume that
\begin{equation}
\label{eq:ipgamma}
    \gamma_1(0)=\gamma_1(1)=0,
    \qquad
    \gamma_1(\tau)>0
    \quad\text{for every }\tau\in(0,1),
    \qquad
    \gamma_2(0)>\gamma_2(1).
\end{equation}
Let $0=\alpha_0<\alpha_1<\cdots<\alpha_N=1$ be the elements of \(\mathcal V(\gamma)\cup\{0,1\}\), and set
\[
    \tilde I_j=(\alpha_{j-1},\alpha_j),
    \qquad
    a_j=\gamma_1(\alpha_j),
    \qquad
    j=0,\ldots,N.
\]
Then the following properties hold.

\begin{enumerate}
\item[\emph{(i)}] \emph{Height functions.}
For every \(j=1,\ldots,N\), the restriction
\[
    \gamma_1|_{\tilde I_j}\colon \tilde I_j\longrightarrow I_j
\]
is a \(C^2\)-diffeomorphism onto
\begin{equation}
\label{eq:Ij}
    I_j
    =
    \gamma_1(\tilde I_j)
    =
    \bigl(
        \min\{a_{j-1},a_j\},
        \max\{a_{j-1},a_j\}
    \bigr).
\end{equation}
Letting 
\(\tau_j=(\gamma_1|_{\tilde I_j})^{-1}\),
the functions 
\begin{equation}
\label{eq:height}
    f_j(r)=\gamma_2(\tau_j(r)),
    \qquad r\in I_j
\end{equation}
are \(C^2(I_j)\), and the sets
\begin{equation}
\label{eq:graphical-pieces}
    \Sigma_j
    =
    \Phi_\gamma(\tilde I_j\times\mathbb S^1)
    =
    \bigl\{
        (x,y,z)\in\R^3:
        r(x,y)\in I_j,\ 
        z=f_j(r(x,y))
    \bigr\}
\end{equation}
are pairwise disjoint radial graphs satisfying
\begin{equation}
\label{eq:boundary-decomposition}
    \partial\Omega
    =
    \bigcup_{j=1}^N\overline{\Sigma_j}.
\end{equation}
Moreover, consecutive pieces meet along the corresponding vertical
circle:
\begin{equation}
\label{eq:adjacent-piece-intersection}
    \overline{\Sigma_j}\cap\overline{\Sigma_{j+1}}
    =
    \Phi_\gamma(\{\alpha_j\}\times\mathbb S^1),
    \qquad
    j=1,\ldots,N-1.
\end{equation}

\item[\emph{(ii)}] \emph{Poles and junctions.}
At the poles,
\[
    f_1\in C^2([0,a_1)),
    \qquad
    f_N\in C^2([0,a_{N-1})),
\]
and
\begin{equation}
\label{eq:bc_0}
    \lim_{r\to0^+}f_1'(r)
    =
    \lim_{r\to0^+}f_N'(r)
    =
    0.
\end{equation}
For \(j=1,\ldots,N\), set
\begin{equation}
\label{eq:sigmaj}
    \sigma_j
    =
    \operatorname{sgn}
    \bigl(\gamma_1'|_{\tilde I_j}\bigr)
    \in\{-1,1\}.
\end{equation}
At every junction point \(a_j\), \(j=1,\ldots,N-1\), we have
\begin{equation}
\label{eq:bc_int}
    \lim_{r\to a_j^{-\sigma_j}}|f_j'(r)|
    =
    \lim_{r\to a_j^{\sigma_{j+1}}}|f_{j+1}'(r)|
    =
    +\infty,
\end{equation}
where \(r\to a^\sigma\) denotes the one-sided limit at \(a\) from the
side specified by \(\sigma\).

For every nonzero endpoint \(a\in\partial I_j\), there exist
\(\rho>0\) and \(\delta_j(a)\in\{-1,1\}\) such that
\begin{equation}
\label{eq:monotonaj}
    \delta_j(a)f_j'(r)>0
    \qquad
    \text{for every }
    r\in I_j\cap(a-\rho,a+\rho).
\end{equation}
At a common endpoint \(a_j\) of two consecutive pieces, these signs
satisfy
\begin{equation}
\label{eq:junc_deltasigma}
    \sigma_j\delta_j(a_j)
    =
    \sigma_{j+1}\delta_{j+1}(a_j),
    \qquad
    j=1,\ldots,N-1.
\end{equation}

\item[\emph{(iii)}] \emph{Equators.}
Let
\[
    R=\max_{\tau\in[0,1]}\gamma_1(\tau)
\]
and choose \(\tau_R\in(0,1)\) such that \(\gamma_1(\tau_R)=R\).
Then \(\tau_R\in\mathcal V(\gamma)\). Denote by
\(j_R\in\{1,\ldots,N-1\}\) the index such that
\(\tau_R=\alpha_{j_R}\). Then
\begin{equation}
\label{eq:atjR}
    a_{j_R}=R,
    \qquad
    \sigma_{j_R}=1,
    \qquad
    \sigma_{j_R+1}=-1,
\end{equation}
and
\begin{equation}
\label{eq:bcbR}
    \lim_{r\to R^-}f_{j_R}'(r)=-\lim_{r\to R^-}f_{j_R+1}'(r)=-\infty.
\end{equation}

\item[\emph{(iv)}] \emph{Curvatures.}
For \(j=1,\ldots,N\), define
\begin{equation}
\label{eq:varphij}
    \varphi_j(r)
    =
    H\bigl(r,0,f_j(r)\bigr),
    \qquad r\in I_j.
\end{equation}
Then \(\varphi_j\) extends continuously to \(\overline I_j\), and
\begin{equation}
\label{eq:Hvarphij}
    H(p)=\varphi_j(r(p)),
    \qquad p\in\Sigma_j.
\end{equation}
On \(I_j\), the function \(\varphi_j\) satisfies equation~\eqref{eq:MCHeis}
with \(\sigma=\sigma_j\). Moreover, the continuity of \(H\) at each
junction gives
\begin{equation}
\label{eq:curvature-junction-compatibility}
    \lim_{r\to a_j^{-\sigma_j}}\varphi_j(r)
    =
    \lim_{r\to a_j^{\sigma_{j+1}}}\varphi_{j+1}(r),
    \qquad
    j=1,\ldots,N-1.
\end{equation}
\end{enumerate}
\end{lem}

\begin{proof}
Most of the assertions follow directly from the construction. We
provide details only where needed. Up to reparametrization,
condition~\eqref{eq:ipgamma} follows from the fact that
\(\partial\Omega\) is homeomorphic to a sphere.

\begin{enumerate}
\item[\emph{(i)}]
For every \(j=1,\ldots,N\), the derivative \(\gamma_1'\) has constant
sign on \(\tilde I_j\). Hence
\(\gamma_1|_{\tilde I_j}\) is a \(C^2\)-diffeomorphism onto \(I_j\).
The remaining assertions follow from the definitions and the regularity
of \(\partial\Omega\).

\item[\emph{(ii)}]
The assertions at the poles follow from the \(C^2\)-regularity and
axisymmetry of \(\partial\Omega\). Notice that \(a_0=a_N=0\).

By~\eqref{eq:height}, for every \(j=1,\ldots,N\),
\begin{equation}
\label{eq:fprimegamma}
    f_j'(r)
    =
    \frac{\gamma_2'(\tau_j(r))}
         {\gamma_1'(\tau_j(r))},
    \qquad r\in I_j.
\end{equation}
Since \(\gamma_1'=0\) on \(\mathcal V(\gamma)\), the regularity of
\(\gamma\) implies that \(\gamma_2'\neq0\) there.
Formula~\eqref{eq:bc_int} follows.

In particular, \(f_j'\) does not vanish in a sufficiently small
one-sided neighborhood of each nonzero endpoint of \(I_j\). Therefore, \(f_j'\) has constant sign in
such a neighborhood, proving~\eqref{eq:monotonaj}. Finally,
at the junction point corresponding to \(\alpha_j\), we observe that 
\begin{equation}
    \sigma_j\delta_j(a_j)
    =
    \operatorname{sgn}\bigl(\gamma_2'(\alpha_j)\bigr)
    =
    \sigma_{j+1}\delta_{j+1}(a_j),
\end{equation}
which proves~\eqref{eq:junc_deltasigma}.

\item[\emph{(iii)}]
Since \(\gamma_1(0)=\gamma_1(1)=0<R\), we have
\(\tau_R\in(0,1)\). As \(\tau_R\) is a maximum point of \(\gamma_1\),
it belongs to \(\mathcal V(\gamma)\). By maximality and the chosen
orientation,
\[
    \sigma_{j_R}=1,
    \qquad
    \sigma_{j_R+1}=-1.
\]
Moreover, regularity gives
\(\gamma_2'(\alpha_{j_R})\neq0\), while the chosen orientation gives
\(\gamma_2'(\alpha_{j_R})<0\). Since
\(\tau_{j_R}(r)\to\alpha_{j_R}^-\) as \(r\to R^-\),
formula~\eqref{eq:fprimegamma} yields
\[
    \lim_{r\to R^-}f_{j_R}'(r)
    =
    \lim_{r\to R^-}
    \frac{\gamma_2'(\tau_{j_R}(r))}
         {\gamma_1'(\tau_{j_R}(r))}
    =
    -\infty.
\]
Arguing in the same way for \(f_{j_R+1}\), we obtain the second limit in~\eqref{eq:bcbR}, and the conclusion follows.

\item[\emph{(iv)}]
The representation of \(H\) on each \(\Sigma_j\) and the formula for
\(\varphi_j\) follow from~\eqref{eq:MCHeis}. Its continuous extension
to \(\overline I_j\), as well as
\eqref{eq:curvature-junction-compatibility}, follows from the assumed
continuous extension of \(H\) to \(\mathcal C(\partial\Omega)\).
\end{enumerate}
\end{proof}

The next proposition contains the solution to the \emph{prescribed $g_\varepsilon$-mean curvature equation} that is needed to build the sets $\Omega_\varepsilon$.

\begin{prop}[Prescribing $g_\varepsilon$-mean curvature]
\label{thm:prescrmcanalit}
Let $\Omega$ be a bounded open $C^2$-axisymmetric set, homeomorphic to a sphere, such that $H(\partial\Omega)$ extends continuously to $\mathcal C(\partial\Omega)$.
    Let $\gamma\in C^2([0,1],\R^2_+)$ be a generating profile for $\partial\Omega$ with $\#\mathcal V(\gamma)<\infty$, and let $\Sigma_1,\dots,\Sigma_N$ be the surfaces defined in Lemma~\ref{lem:decomposition}.
    Then there exists \(\varepsilon_0>0\) such that, for every \(\varepsilon\in(0,\varepsilon_0)\), there exists a $C^2$-regular axisymmetric set $\Omega_\varepsilon\subset\R^3$, whose boundary decomposes correspondingly into pieces $\Sigma_1^\varepsilon,\dots,\Sigma_N^\varepsilon$, such that its $g_\varepsilon$-mean curvature satisfies for each $j=1,\dots,N$
    \begin{equation}\label{eq:prescrHeps}
        H_\eps (p) = H(p') \quad\text{for}\quad p \in \Sigma_j^\varepsilon \setminus{\{r=0\}},\quad p'\in\Sigma_j \setminus{\{r=0\}}\quad \text{such that}\quad r (p) = r (p').
    \end{equation} 
    In particular, $H_\eps\in C^0(\partial \Omega_\varepsilon)$.
    
    Moreover, as $\varepsilon\to 0$, the sets $\partial\Omega_\varepsilon$ converge to $\partial\Omega$ with respect to the Hausdorff distance $d_{\mathcal H}$.
\end{prop}

We first prove a few auxiliary lemmas containing information on $f$ satisfying~\eqref{eq:MCHeis} for a continuous function $\varphi$, that will be used to integrate the prescribed mean curvature equation~\eqref{eq:prescrHeps}.

\begin{lem}\label{prop:proprietaperodeversgenerale}
Let $\sigma\in\{-1,1\}$. For $0\le a<b$, let $\varphi\in C([a,b])$ and $f\in C([a,b])$ be such that
    \begin{enumerate}[label=\alph*)]
    \item $f\in C^2([0,b))$ if $a=0$, and $f\in C^2((a,b))$ if $a>0$;
    \item for $r\in (a,b)$ we have
     \begin{equation}\label{eq:ODEgeneralprop}
        -\frac{\sigma}{r} \frac{\mathrm{d}}{\mathrm{d}r}\left(
        \frac{r f'(r)}{\sqrt{f'(r)^2+\frac{r^2}{4}}}
        \right) = \varphi(r);
    \end{equation}
    \item there exist $\delta_a,\delta_b\in\{-1,1\}$ such that 
    \begin{equation}
    \begin{cases}
        \displaystyle\lim_{r\to0^+}f'(r)=0,
            &\text{if }a=0,\\[1ex]
        \displaystyle\lim_{r\to a^+}\delta_a f'(r)=+\infty,
            &\text{if }a>0,
    \end{cases}
    \qquad
    \lim_{r\to b^-}\delta_b f'(r)=+\infty.
    \end{equation}
\end{enumerate}
Then the function $\psi:[a,b]\to \R$ defined by
\begin{equation}\label{eq:psi}
    \psi(r)=\int_a^r t\,\varphi(t)\,dt,
    \qquad a\le r\le b,
\end{equation}
satisfies the following:
    \begin{enumerate}[label=\roman*.]
        \item $\psi\in C^1((a,b)) \cap C^0([a,b])$;
        
        \item $r^2-(\psi(r) - a\delta_a\sigma)^2\ge 0$ in $[a,b]$ and 
        \begin{equation}\label{eq:existence_condition}
        r^2-(\psi(r) - a\delta_a\sigma)^2 = 0 \iff r\in\{a,b\};
        \end{equation}
        
        \item Let $\Psi:(a,b)\to \R$ be the function defined by 
        \begin{equation}\label{eq:Psi}
            \Psi(r)=\frac{|\psi(r)-a\delta_a\sigma|\,\frac{r}{2}}{\sqrt{r^2-(\psi(r)-a\delta_a\sigma)^2}}.
        \end{equation} 
        Then $\Psi\in L^1([a,b])$ and for $r\in (a,b)$ it satisfies
        \begin{equation}\label{eq:tesiiii}
            f'(r)= \delta(r)\Psi(r).
        \end{equation}        
        where $\delta:[a,b]\to \R$ is the function defined by
        \begin{equation}\label{eq:delta}
            \delta(r)=
            \begin{cases}
                \operatorname{sgn}(f'(r)),&\quad \text{if }f'(r)\neq 0;\\
                0,&\quad \text{if }f'(r)= 0.\\
            \end{cases} 
        \end{equation}
   \end{enumerate}
\end{lem}

\begin{proof}
    Since $\varphi$ is continuous, then $\psi\in C^1((a,b))\cap C^0([a,b])$.  This proves i). 

    We prove ii). 
    Integrating \eqref{eq:ODEgeneralprop} on $(a,r)$ for $r\in (a,b)$, we obtain
    \begin{equation}
        \psi(r) 
        = -\sigma \frac{r f'(r)}{\sqrt{f'(r)^2 + \frac{r^2}{4}}} +\sigma  C_a,\quad a< r < b,
    \end{equation}
    where $C_a = \lim_{r \to a^+} \frac{r f'(r)}{\sqrt{f'(r)^2 + r^2/4}}$.
    If $a=0$, assumptions a) and c) imply that
   \begin{equation}
       C_0=\lim_{r\to 0^+}\frac{r(f''(0)r+o(r))}{\sqrt{(f''(0)r+o(r))^2+\frac{r^2}{4}}} =\lim_{r\to 0^+}\frac{f''(0)r^2+o(r^2)}{r\sqrt{f''(0)^2+\frac{1}{4}+o(1)}}=0.
   \end{equation}
    If $a>0$, assumption c) implies that 
    \begin{equation}
        C_a = \lim_{r \to a^+} \frac{r \delta_a}{\sqrt{1 + \frac{r^2}{4f'(r)^2}}} = a\delta_a.
    \end{equation}
    We therefore obtain
    \begin{equation}\label{eq:psi_identity_proof}
        \psi(r) = -\sigma \frac{r f'(r)}{\sqrt{f'(r)^2 + \frac{r^2}{4}}}+\sigma a\delta_a,
    \quad r\in (a,b),
    \end{equation}
    that can be rewritten as
\begin{equation}\label{eq:fsigmaab}
    f'^2(r)\bigl[r^2-(\psi(r)-a\delta_a\sigma)^2\bigr]
    =(\psi(r)-a\delta_a\sigma)^2\frac{r^2}{4},
    \quad r\in (a,b).
\end{equation}
Since the right-hand side is nonnegative we have that
\[
    r^2-(\psi(r)-a\delta_a\sigma)^2\ge 0
    \qquad \text{for every } r\in(a,b).
\]

We show that equality holds at the endpoints. At $r=a$, this follows directly from the definition of $\psi$. At $r\to b^-$, using assumption~c) in~\eqref{eq:fsigmaab}, we get
\begin{equation}
    \lim_{r\to b^-}\bigl[r^2-(\psi(r)-a\delta_a\sigma)^2\bigr]
    =
    \lim_{r\to b^-}
    \frac{(\psi(r)-a\delta_a\sigma)^2r^2}{4f'^2(r)}
    =0,
\end{equation}
because $(\psi(r)-a\delta_a\sigma)^2r^2$ is continuous, and hence bounded, on $[a,b]$.

On the other hand, if $r\in(a,b)$ satisfies
\begin{equation}\label{eq:utile}
    r^2=(\psi(r)-a\delta_a\sigma)^2
\end{equation}
equation~\eqref{eq:fsigmaab} yields $\frac{r^2}{4}(\psi(r)-a\delta_a\sigma)^2=0$. Plugging~\eqref{eq:utile} in the latter we thus obtain $r=0$ which is absurd. This proves ii).

    We conclude by proving iii). Observe that $\Psi$ is well defined on the whole $(a,b)$ by~\eqref{eq:existence_condition}.
    By \eqref{eq:fsigmaab} and ii) we therefore deduce that $|f'(r)|=\Psi(r)$, yielding~\eqref{eq:tesiiii}.
    We deduce that for any $c,d\in(a,b)$, $c<d$, we have $\Psi\in L^1([c,d])$ being $f'\in C^0([c,d])$. 

    If $a=0$, assumption c) guarantees that $f'\in C^0([0,b))$, so that $\Psi\in L^1([0,d])$ for any $d\in (0,b)$.
    
    If $a>0$, by assumption c) there exists $r_a>a$ such that for $r\in (a,r_a)$, we have $|f'(r)|=\delta_a f'(r)$. Therefore, for any $r\in (a,r_a)$ we can compute
    \begin{equation}
        \lim_{\eta\to 0^+}\int_{a+\eta}^{r} |\Psi(s)|\,ds
        =\lim_{\eta\to 0^+}\int_{a+\eta}^{r} \delta_a f'(s)\;ds
        =\lim_{\eta\to0^+}\delta_a(f(r)-f(a+\eta))
        \in \R,
    \end{equation}
    thus proving $\Psi\in L^1([a,r])$. 
    
    The same argument proves integrability of $|\Psi|$ on $[r,b]$. This concludes the proof.
\end{proof}

\begin{rem}
    An explicit formula for $f$ satisfying the assumptions of Lemma~\ref{prop:proprietaperodeversgenerale} in terms of $\psi$ can be deduced by integrating  equation~\eqref{eq:tesiiii}. 
    For $r\in[a,b]$ this yields
    \begin{equation}\label{eq:formulaf}
        f(r)=f(b)-\int_r^b\delta(t)\Psi(t)\; dt=f(a)+\int_a^r\delta(t)\Psi(t)\; dt.
    \end{equation}
\end{rem}

\begin{lem}\label{lem:fprime}
Let $\sigma\in\{-1,1\}$. For $b>0$, let $\varphi\in C([0,b])$ and $f\in C([0,b])\cap C^2([0,b))$ be such that \eqref{eq:ODEgeneralprop} is satisfied on $(0,b)$ and 
    \begin{equation}
        \lim_{r\to 0^+}f'(r)=0
    \end{equation}    
Then as $r\to 0^+$ we have
    \begin{equation}\label{eq:asympf'0}
        f'(r)=\frac{-\sigma \varphi(0)}{4}r^2+o(r^2).
    \end{equation}
\end{lem}

\begin{proof}[Proof of Lemma~\ref{lem:fprime}]
    Since $f\in C^2([0,b))$ we have that 
\begin{equation}
    \lim_{r\to 0^+}\frac{rf'(r)}{\sqrt{f'(r)^2+\frac{r^2}{4}}}=\lim_{r\to 0^+}\frac{r^2f''(0)}{r\sqrt{f''(0)^2+\frac{1}{4}+o(1)}}=0.
\end{equation}
Using also~\eqref{eq:ODEgeneralprop}, this implies that for $r\to 0^+$ we have
\begin{equation}
    \frac{rf'(r)}{\sqrt{f'(r)^2+\frac{r^2}{4}}}=\int_0^r\frac{d}{dt}\left(\frac{tf'(t)}{\sqrt{f'(t)^2+\frac{t^2}{4}}}\right)\;dt
    =\int_0^r(-\sigma\varphi(0)t+o(t))\;dt,
\end{equation}
where in the last equality we used the continuity of $\varphi$ at $0$.
Computing the integral and dividing by $r\to 0^+$ therefore yields:
\begin{equation}
    \frac{f'(r)}{\sqrt{f'(r)^2+\frac{r^2}{4}}}=-\sigma \varphi(0)\frac{r}{2}+o(r),
\end{equation}
If $\varphi(0)\neq 0$ this implies that $\delta(r)=\operatorname{sgn}(-\sigma\varphi(0))$, so that we can rewrite the above as 
\begin{equation}
    f'(r)=\frac{\frac{r}{2}\left(-\sigma\varphi(0)\frac{r}{2}+o(r)\right)}{\sqrt{1-\left(\sigma\varphi(0)\frac{r}{2}+o(r)\right)^2}}=-\sigma \varphi(0)\frac{r^2}{4}+o(r^2).
\end{equation}
The computation directly extends to the case $\varphi(0)=0$ thus proving the statement.
\end{proof}

In the next proposition, we solve the local boundary value problems for the prescribed $g_\varepsilon$-mean curvature equation~\eqref{eq:MCHeis-eps} on intervals $[a,b]\subset[0,R]$, with $\varphi_\varepsilon=\varphi$. In view of the global construction outlined above, the boundary conditions are chosen to match the behavior of $\partial\Omega$ at the endpoints $a$ and $b$ and the proof makes use of the properties in Lemma~\ref{prop:proprietaperodeversgenerale}.

\begin{prop}\label{prop:existence_explicit_epsilon}
    Let $\sigma\in\{-1,1\}$. For $0\le a<b$, let $\varphi\in C([a,b])$ and $f\in C([a,b])$ be such that
    \begin{enumerate}[label=\alph*)]
    \item $f\in C^2([0,b))$ if $a=0$, and $f\in C^2((a,b))$ if $a>0$;
    \item for $r\in (a,b)$ we have
     \begin{equation}\label{eq:ODEgeneralprop2}
        -\frac{\sigma}{r} \frac{\mathrm{d}}{\mathrm{d}r}\left(
        \frac{r f'(r)}{\sqrt{f'(r)^2+\frac{r^2}{4}}}
        \right) = \varphi(r);
    \end{equation}
    \item there exist $\delta_a,\delta_b\in\{-1,1\}$ such that 
    \begin{equation}\label{eq:bcf}
    \begin{cases}
        \displaystyle\lim_{r\to0^+}f'(r)=0,
            &\text{if }a=0,\\[1ex]
        \displaystyle\lim_{r\to a^+}\delta_a f'(r)=+\infty,
            &\text{if }a>0,
    \end{cases}
    \qquad
    \lim_{r\to b^-}\delta_b f'(r)=+\infty.
    \end{equation}
\end{enumerate}

Let $z_a, z_b \in \mathbb{R}$. For any $\varepsilon>0$ consider the prescribed $g_\varepsilon$-mean curvature equation on $(a,b)$ given by
\begin{equation}
    \label{eq:MCHeis-eps-lem}
    -
    \frac{\sigma}{r} \, \partial_r\left( \frac{r f_\varepsilon'(r)}{\sqrt{f_\varepsilon'(r)^2 + \frac{r^2}{4} + \varepsilon^2}} \right)=\varphi,
\end{equation}
together with one of the following boundary conditions 
\begin{align}
    \label{eq:BC0} \tag{BC$_0$}
    f_\varepsilon'(0)=0, 
    & \quad \lim_{r\to b^-} \delta_b f_\varepsilon'(r)=+\infty, 
     \quad f_\varepsilon(b)=z_b\quad\text{with}\quad a=0;\\
    \label{eq:BCa} \tag{BC$_a$}
    \lim_{r\to a^+} \delta_a f_\varepsilon'(r)=+\infty, 
    & \quad  \lim_{r\to b^-} \delta_b f_\varepsilon'(r)=+\infty, 
    \quad f_\varepsilon(a)=z_a\quad\text{with}\quad a>0;\\
    \label{eq:BCb} \tag{BC$_b$}
    \lim_{r\to a^+} \delta_af_\varepsilon'(r)=+\infty, 
    & \quad \lim_{r\to b^-} \delta_b f_\varepsilon'(r)=+\infty,
    \quad f_\varepsilon(b)=z_b\quad\text{with}\quad a>0.
\end{align}
Then the corresponding boundary value problem admits a unique solution 
\(f_\varepsilon\in C^2((a,b))\cap C^0([a,b])\) if $a>0$ and \(f_\varepsilon\in C^2([0,b))\cap C^0([0,b])\) if $a=0$, given by the formula
\begin{equation}\label{eq:fepsb}
    f_\varepsilon(r)
    =
    z_b-\int_r^b f'(t)\sqrt{1+\frac{4\varepsilon^2}{t^2}}\,dt
\end{equation}
in cases~\eqref{eq:BC0} and~\eqref{eq:BCb}, and by
\begin{equation}\label{eq:fepsa}
    f_\varepsilon(r)
    =
    z_a+\int_a^r f'(t)\sqrt{1+\frac{4\varepsilon^2}{t^2}}\,dt
\end{equation}
in case~\eqref{eq:BCa}. 

Moreover, if $a=0$ the following Taylor expansion as $r\to 0^+$ holds true
\begin{equation}\label{eq:tayloreps0}
        f_\varepsilon'(r)=-\frac{\sigma\varphi(0)\varepsilon}{2}r+o(r).
\end{equation}

Finally, as $\varepsilon\to 0^+$, $f_\varepsilon$ converges uniformly on $[a,b]$ to the function $f_a=f-f(a)+z_a$ in case \eqref{eq:BCa}, and to the function $f_b=f-f(b)+z_b$ in case \eqref{eq:BC0} and \eqref{eq:BCb}.

\end{prop}
\begin{proof}
\textit{Step I: Existence of a solution.}
Let $\psi:[a,b]\to\mathbb R$, and
$\Psi:(a,b)\to\mathbb R$ be the functions defined in~\eqref{eq:psi}, and~\eqref{eq:Psi}, respectively. For $\varepsilon>0$, set
\begin{equation}
    \Psi_\varepsilon(r)
    =
    \frac{2\Psi(r)}{r}\sqrt{\frac{r^2}{4}+\varepsilon^2},
    \qquad r\in(a,b).
\end{equation}
We first prove that
\begin{equation}\label{eq:claimPsieps}
    \Psi_\varepsilon\in L^1([a,b]).
\end{equation}
Fix $c\in(a,b)$. Since
\begin{equation}
    \frac{1}{r}\sqrt{\frac{r^2}{4}+\varepsilon^2}
    \le
    \frac{1}{c}\sqrt{\frac{b^2}{4}+\varepsilon^2}
    \qquad\text{for } r\in[c,b],
\end{equation}
and since $\Psi\in L^1([a,b])$ by Lemma~\ref{prop:proprietaperodeversgenerale},
statement~iii), we have $\Psi_\varepsilon\in L^1([c,b])$.

It remains to consider the behavior near $a$. If $a>0$, the same estimate applies on
$[a,c]$, and therefore $\Psi_\varepsilon\in L^1([a,c])$. Assume instead that $a=0$.
Then
\begin{equation}
    \label{eq:fepsprime0}
    \lim_{r\to0^+}\frac{\psi(r)}{r}
    =
    \lim_{r\to0^+}\frac{1}{r}\int_0^r t\varphi(t)\,dt
    =0,
\end{equation}
because $\varphi$ is continuous at $0$. Hence
\begin{equation}\label{eq:Psisur}
    \lim_{r\to0^+}\frac{2\Psi(r)}{r}
    =
    \lim_{r\to0^+}
    \frac{|\psi(r)|}{\sqrt{r^2-\psi(r)^2}}
    =
    \lim_{r\to0^+}
    \frac{|\psi(r)|}{r}
    \frac{1}{\sqrt{1-\left(\frac{\psi(r)}{r}\right)^2}}
    =0.
\end{equation}
Thus $\Psi_\varepsilon$ extends continuously to $r=0$ by setting
$\Psi_\varepsilon(0)=0$, and in particular $\Psi_\varepsilon\in L^1([0,c])$.
This proves~\eqref{eq:claimPsieps}.

Let $\delta:[a,b]\to\{-1,0,1\}$ be the function defined in~\eqref{eq:delta}. For $\varepsilon>0$, the candidate solution $f_\varepsilon:[a,b]\to\mathbb R$ given by formula~\eqref{eq:fepsb} or \eqref{eq:fepsa} reads 
\begin{equation}
    f_\varepsilon(r)
    =
    z_b-\int_r^b \delta(t)\Psi_\varepsilon(t)\,dt
\end{equation}
in cases~\eqref{eq:BC0} and~\eqref{eq:BCb}, and by
\begin{equation}
    f_\varepsilon(r)
    =
    z_a+\int_a^r \delta(t)\Psi_\varepsilon(t)\,dt
\end{equation}
in case~\eqref{eq:BCa}. By~\eqref{eq:claimPsieps} $f_\varepsilon$ is well defined and continuous on $[a,b]$. Moreover, by~\eqref{eq:tesiiii} and assumption~a), $\delta\Psi\in C^1((a,b))$, hence also $\delta\Psi_\varepsilon\in C^1((a,b))$. Therefore $f_\varepsilon\in C^2((a,b))$ and we can 
differentiate its defining formula, obtaining
\begin{equation}\label{eq:fepsprime}
    f_\varepsilon'(r)=\delta(r)\Psi_\varepsilon(r),
    \qquad r\in(a,b).
 \end{equation}

If $a=0$, we show that $f_\varepsilon$ is $C^2$ at $0$. In fact, the comparison of~\eqref{eq:tesiiii} and \eqref{eq:fepsprime} yields
\begin{equation} f_\varepsilon'(r)=f'(r)
\sqrt{1+\frac{4\varepsilon^2}{r^2}}.
\end{equation}
By~\eqref{eq:asympf'0}, we therefore obtain~\eqref{eq:tayloreps0}, that implies that 
\begin{equation}\label{eq:cond0}
    \lim_{r\to 0^+}f_\varepsilon'(r)=0,\quad 
    {f_\varepsilon}''(0)=\lim_{r\to 0^+}\frac{f_\varepsilon'(r)}{r}=-\frac{\sigma\varphi(0)\varepsilon}{2},
\end{equation}
hence yielding  $C^2$-regularity at $0$.
Moreover, rewriting~\eqref{eq:MCHeis-eps-lem} as
\begin{equation}\label{eq:fepssecond}
    f_\varepsilon''(r)
    =
    -\frac{1}{\frac{r^2}{4}+\varepsilon^2}
    \left\{
        \sigma\varphi(r)
        \left(
            f_\varepsilon'(r)^2+\frac{r^2}{4}+\varepsilon^2
        \right)^{3/2}
        +
        \frac{f_\varepsilon'(r)^3}{r}
        +
        \varepsilon^2\frac{f_\varepsilon'(r)}{r}
    \right\},
\end{equation}
and using~\eqref{eq:cond0}, we obtain
\[
    \lim_{r\to0^+}f_\varepsilon''(r)
    =
    -\frac{1}{\varepsilon^2}
    \left(
        \sigma\varphi(0)\varepsilon^3
        -
        \frac{\sigma\varphi(0)\varepsilon^3}{2}
    \right)
    =
    -\frac{\sigma\varphi(0)\varepsilon}{2}
    =
    f_\varepsilon''(0).
\]
Therefore \(f_\varepsilon\in C^2([0,b))\).

We show that $f_\varepsilon$ solves~\eqref{eq:MCHeis-eps-lem}. 
By the definition of $\Psi_\varepsilon$, a direct computation gives
\begin{equation}
    f_\varepsilon'(r)^2+\frac{r^2}{4}+\varepsilon^2
    =
    \frac{r^2\left(\frac{r^2}{4}+\varepsilon^2\right)}
    {r^2-\left(\psi(r)-a\delta_a\sigma\right)^2},
    \qquad r\in(a,b).
\end{equation}
In this computation, the case $\delta(r)=0$ is covered by
formula~\eqref{eq:tesiiii}, which gives $\psi(r)-a\delta_a\sigma=0$. Hence
\[
    \frac{r f_\varepsilon'(r)}
    {\sqrt{f_\varepsilon'(r)^2+\frac{r^2}{4}+\varepsilon^2}}
    =
    \delta(r)\left|\psi(r)-a\delta_a\sigma\right|,
    \qquad r\in(a,b).
\]
By~\eqref{eq:tesiiii}, the right-hand side agrees with the corresponding
quantity for the limiting profile, namely
\begin{equation}
    \frac{r f'(r)}{\sqrt{f'(r)^2+\frac{r^2}{4}}}=\delta(r)\left|\psi(r)-a\delta_a\sigma\right|,
    \qquad r\in(a,b).
\end{equation}
Hence, using~\eqref{eq:ODEgeneralprop2}, we get
\begin{equation}
    -\frac{\sigma}{r}\frac{\mathrm d}{\mathrm dr}
    \left(
    \frac{r f_\varepsilon'(r)}
    {\sqrt{f_\varepsilon'(r)^2+\frac{r^2}{4}+\varepsilon^2}}
    \right)
    =
    -\frac{\sigma}{r}\frac{\mathrm d}{\mathrm dr}
    \left(
    \frac{r f'(r)}
    {\sqrt{f'(r)^2+\frac{r^2}{4}}}
    \right)
    =
    \varphi(r),
    \qquad r\in(a,b).
\end{equation}
Thus $f_\varepsilon$ satisfies~\eqref{eq:MCHeis-eps-lem}.


We finally check the boundary conditions. All derivatives are taken in the open
interval $(a,b)$, and the conditions involving $f_\varepsilon'$ are understood
as one-sided limits. In case~\eqref{eq:BC0}, the condition at $a=0$ follows
from~\eqref{eq:cond0}. At $b$, condition~\eqref{eq:bcf} gives
$\delta(r)=\delta_b$ for $r$ sufficiently close to $b$. Moreover, by
Lemma~\ref{prop:proprietaperodeversgenerale}, $\psi$ is continuous at $b$ and
\eqref{eq:existence_condition} holds. Hence
\begin{equation}
    \lim_{r\to b^-}\delta_b f_\varepsilon'(r)
    =
    \lim_{r\to b^-}
    \frac{|\psi(r)|}{\sqrt{r^2-\psi(r)^2}}
    \sqrt{\frac{r^2}{4}+\varepsilon^2}
    =
    +\infty.
\end{equation}
The prescribed value $f_\varepsilon(b)=z_b$ follows directly from the definition
of $f_\varepsilon$.

The same argument gives the boundary condition at $b$ in case~\eqref{eq:BCb},
with $\psi(r)$ replaced by $\psi(r)-a\delta_a\sigma$. Similarly,
\eqref{eq:bcf} gives $\delta(r)=\delta_a$ for $r$ sufficiently close to $a$,
and the same computation yields the boundary condition at $a$ in
cases~\eqref{eq:BCa} and~\eqref{eq:BCb}. The prescribed values
$f_\varepsilon(a)=z_a$ in case~\eqref{eq:BCa} and $f_\varepsilon(b)=z_b$ in
case~\eqref{eq:BCb} follow from the definition of $f_\varepsilon$.
\smallskip

\textit{Step II: Uniqueness of the solution.} 
We prove uniqueness. Let \(h\) be any solution of the boundary value problem
under consideration. Set
\[
    W_h(r)=
    \frac{r h'(r)}
    {\sqrt{h'(r)^2+\frac{r^2}{4}+\varepsilon^2}},
    \qquad r\in(a,b).
\]
Then the ODE \eqref{eq:MCHeis-eps-lem} rewrites as 
\(W_h'(r)=-\sigma r\varphi(r)\). Hence, for every \(r\in(a,b)\), we have 
\[
    W_h(r)=C_a-\sigma\int_{a}^r t\varphi(t)\,dt,
\]
where 
\begin{equation}
    C_a = \lim_{r \to a^+} \frac{r h'(r)}{\sqrt{h'(r)^2 + \frac{r^2}{4}+\varepsilon^2}}=a\delta_a.
\end{equation}
Here, the computation of $C_a$ is independent of $h$ and only depends on the boundary conditions under consideration. 
Consequently, 
\[
    W_h(r)
    = W_{f_\varepsilon}(r)\quad \text{ for any } r\in (a,b).
\]

It remains to observe that, for each fixed \(r\in(a,b)\), the map
\[
    p\mapsto
    \frac{r p}{\sqrt{p^2+\frac{r^2}{4}+\varepsilon^2}}
\]
is strictly increasing from \(\mathbb R\) onto \((-r,r)\). Therefore $h'(r)=f_\varepsilon'(r)$ for any $r\in (a,b)$, and thus $h$ and $f_\varepsilon$ may differ only by an additive
constant. The normalization condition \(h(a)=z_a\) or \(h(b)=z_b\) respectively,
fixes this constant. This proves uniqueness.
\smallskip

\textit{Step III: Uniform convergence as $\varepsilon\to 0^+$.}
We prove the uniform convergence in the case of \eqref{eq:BC0} and
\eqref{eq:BCb}; the proof for \eqref{eq:BCa} is analogous.

We first claim that
\begin{equation}
\label{eq:claim_uniform}
    \lim_{\varepsilon\to 0^+}
    \int_a^b |\Psi(t)|
    \left(
    \frac{2}{t}\sqrt{\frac{t^2}{4}+\varepsilon^2}-1
    \right)\,dt=0 .
\end{equation}
Indeed, this implies
\begin{equation}
    f_\varepsilon \to f-f(b)+z_b
    \qquad\text{uniformly on }[a,b],
\end{equation}
since
\begin{equation}
\begin{split}
    \sup_{r\in[a,b]}
    |f_\varepsilon(r)-f(r)+f(b)-z_b|
    &=
    \sup_{r\in[a,b]}
    \left|
    \int_r^b \delta(t)(\Psi_\varepsilon(t)-\Psi(t))\,dt
    \right|  \\
    &\le
    \int_a^b |\Psi(t)|
    \left(
    \frac{2}{t}\sqrt{\frac{t^2}{4}+\varepsilon^2}-1
    \right)\,dt .
\end{split}
\end{equation}

It remains to prove \eqref{eq:claim_uniform}. For every fixed
$t\in(a,b]$, one has
\begin{equation}
    \frac{2}{t}\sqrt{\frac{t^2}{4}+\varepsilon^2}-1
    \longrightarrow 0
    \qquad\text{as }\varepsilon\to0^+ .
\end{equation}
Thus, \eqref{eq:claim_uniform} follows by dominated convergence once we
find an $L^1$-dominating function independent of $\varepsilon$, for
$\varepsilon$ sufficiently small.

If $a>0$, then for $0<\varepsilon\le1$ and every $t\in[a,b]$,
\begin{equation}
    0\le
    \frac{2}{t}\sqrt{\frac{t^2}{4}+\varepsilon^2}-1
    \le
    \frac{2}{a}\sqrt{\frac{b^2}{4}+1}-1 .
\end{equation}
Therefore the integrand is bounded by a constant multiple of $|\Psi|$.
Since $\Psi\in L^1([a,b])$
by Lemma~\ref{prop:proprietaperodeversgenerale}, statement iii), the
conclusion follows by dominated convergence.

Let now $a=0$. By \eqref{eq:Psisur}, the map
$t\mapsto \Psi(t)/t$ extends continuously at $t=0$. In particular,
\begin{equation}\label{eq:Psisut}
    \frac{\Psi(t)}{t}\in L^1([0,b]).
\end{equation}
For $0<\varepsilon\le1$, we have
\begin{equation}
\begin{split}
    0\le
    \frac{2}{t}\sqrt{\frac{t^2}{4}+\varepsilon^2}-1
    &=
    \frac{2}{t}
    \left(
    \sqrt{\frac{t^2}{4}+\varepsilon^2}-\frac t2
    \right) \le
    \frac{2}{t}
    \left(
    \sqrt{\frac{t^2}{4}+1}-\frac t2
    \right).
\end{split}
\end{equation}
Moreover,
\begin{equation}
    \sqrt{\frac{t^2}{4}+1}-\frac t2
    =
    \frac{1}{
    \sqrt{\frac{t^2}{4}+1}+\frac t2
    }
    \le 1 .
\end{equation}
Hence
\begin{equation}
    |\Psi(t)|
    \left(
    \frac{2}{t}\sqrt{\frac{t^2}{4}+\varepsilon^2}-1
    \right)
    \le
    2\left|\frac{\Psi(t)}{t}\right|.
\end{equation}
Therefore by~\eqref{eq:Psisut}, dominated
convergence gives \eqref{eq:claim_uniform} with $a=0$ and the proof is concluded.
\end{proof}

\begin{proof}[Proof of Proposition~\ref{thm:prescrmcanalit}]
\textit{Step I: Construction of the approximants $\Sigma_j^\varepsilon$.} For \(\varepsilon>0\), we construct the surfaces
\(\Sigma_1^\varepsilon,\dots,\Sigma_N^\varepsilon\) by rotation around the
\(z\)-axis of the graphs \(\{z=f_j^\varepsilon\}\), where
each \(f_j^\varepsilon\in C^2(I_j)\cap C^0(\bar I_j)\) solves \eqref{eq:MCHeis-eps-lem} on \(I_j\) with
\(\sigma=\sigma_j\), for \(j=1,\dots,N\).

We start from an equator of \(\Omega\), namely from \(j=j_R\). 
We define 
\(f_{j_R}^\varepsilon\in C^2((a_{j_R-1},R))\cap C^0([a_{j_R-1},R])\) as the unique solution to \eqref{eq:MCHeis-eps-lem} with \(\sigma=\sigma_{j_R}=1\),  with boundary conditions chosen according to whether \(a_{j_R-1}=0\) or \(a_{j_R-1}>0\).

If \(j_R=1\), so that \(a_{j_R-1}=0\), we impose the boundary conditions~\eqref{eq:BC0} with \(z_b=f_{j_R}(R)\) and \(\delta_b=-1\).

If \(j_R\neq 1\), so that \(a_{j_R-1}>0\), we consider boundary conditions~\eqref{eq:BCb}, choosing the boundary data so that the behavior of \(f_{j_R}^\varepsilon\) at \(r=a_{j_R-1}\) and \(r=R\) matches that of \(f_{j_R}\). Namely,
\begin{equation}
    a=a_{j_R-1}, \quad
    b=a_{j_R}=R, \quad
    z_b=f_{j_R}(R), \quad
    \delta_a=\delta_{j_R}(a_{j_R-1})
    , \quad
    \delta_b=-1,
\end{equation}
as follows from~\eqref{eq:bcbR}. 
Proposition~\ref{prop:existence_explicit_epsilon}, then yields the explicit formula~\eqref{eq:fepsb} fo rthe solution \(f_{j_R}^\varepsilon\).
    
We proceed recursively for \(j=j_R-1,\dots,2\). Having constructed
\(f_{j+1}^{\varepsilon}\), we define \(f_j^{\varepsilon}\) as the unique
solution of \eqref{eq:MCHeis-eps-lem} on \(I_j\), with
\(\sigma=\sigma_j\), satisfying the matching condition with
\(f_{j+1}^{\varepsilon}\) at the common endpoint \(a_j\), and whose
derivative has the same blow-up signs at both endpoints as \(f_j'\). Namely, if \(\sigma_{j}=1\), then \(I_{j}=(a_{j-1},a_{j})\), and we impose the boundary conditions \eqref{eq:BCb} with
\begin{equation}
    a=a_{j-1}, \quad
    b=a_{j}, \quad
    z_b=f_{j+1}^{\varepsilon}(a_{j}), \quad
    \delta_a=\delta_j(a_{j-1})
    , \quad
    \delta_b=\delta_j(a_{j})
    .
\end{equation}
If instead \(\sigma_{j}=-1\), then \(I_{j}=(a_{j},a_{j-1})\), and we impose the boundary conditions \eqref{eq:BCa} with
\begin{equation}
    a=a_{j}, \quad
    b=a_{j-1}, \quad
    z_a=f_{j+1}^{\varepsilon}(a_{j}), \quad
    \delta_a=\delta_j(a_j)
    , \quad
    \delta_b=\delta_j(a_{j-1})
    .
\end{equation}

Proceeding recursively, we eventually reach \(I_1=(0,a_1)\). There we solve \eqref{eq:MCHeis-eps-lem} with \(\sigma=\sigma_1\), and impose the boundary conditions \eqref{eq:BC0} with \(b=a_1\) and \(\delta_b=\operatorname{sgn}(\gamma_2'(\alpha_1))\).

Arguing in the same way on the remaining intervals \(I_j\), \(j=j_R+1,\dots,N\), we define \(f_j^\varepsilon\) for every \(j=1,\dots,N\). The boundary data are chosen recursively so that adjacent profiles agree at each common endpoint \(a_j\). Thus the adjacent profiles glue continuously at every junction radii. For $j=1,\dots,N$ we then set
\begin{equation}
    \Sigma_j^\varepsilon
    =
    \left\{
        z=f_j^\varepsilon(r(x,y)),
        \ r(x,y)\in I_j
    \right\}.
\end{equation}
\smallskip
    
\textit{Step II: Uniform convergence and absence of self intersections among the $\Sigma_j^\varepsilon$.} We first claim that for any $j=1,\dots,N$ the function $f_j^\varepsilon$ converges to the function $f_j$ uniformly on $\bar I_j$ as $\varepsilon\to 0^+$.
If $j=j_R$ this is proved in Proposition~\ref{prop:existence_explicit_epsilon} as the boundary condition at $b=R$ is $z_b=f_{j_R}(a_{j_R})$. 
We proceed by (finite) induction on $j$ running backwards from $j_R-1$ to $1$. The same reasoning applies to $j$ running forwards from $j_R+1$ to $N$. 

Assume that $f_j^\varepsilon$ converges to $f_j$ uniformly on $\bar I_j$ as $\varepsilon\to 0^+$. 
We now consider $\bar I_{j-1}$. 
If $\sigma_{j-1}=1$, we have  $\bar I_{j-1}=[a_{j-2},a_{j-1}]$ and by~\eqref{eq:fepsb} and \eqref{eq:tesiiii}, we obtain
\begin{equation}\label{eq:cvunif}
\begin{split}
    \sup_{\bar I_{j-1}}\left|f_{j-1}^\varepsilon-f_{j-1}\right|&
    =\sup_{r\in \bar I_{j-1}}\left|f_j^\varepsilon(a_{j-1})-f_j(a_{j-1})+\int_r^{a_{j-1}} f_{j-1}'(t)\left(1-\sqrt{1+\frac{4\varepsilon^2}{t^2}}\right)\;dt\right|\\
    &\le |f_j^\varepsilon(a_{j-1})-f_j(a_{j-1})|+\sup_{r\in \bar I_{j-1}}\int_r^{a_{j-1}} |\Psi_{j-1}(t)|\left(\sqrt{1+\frac{4\varepsilon^2}{t^2}}-1\right)\;dt\\
    &\le |f_j^\varepsilon(a_{j-1})-f_j(a_{j-1})|+\int_{I_{j-1}}|\Psi_{j-1}(t)|\left(\sqrt{1+\frac{4\varepsilon^2}{t^2}}-1\right)\;dt,
\end{split}
\end{equation}
where $\Psi_{j-1}\in L^1(I_{j-1})$ is the function defined in~\eqref{eq:Psi} with $a=a_{j-2}$, $b=a_{j-1}$ and $\sigma=\sigma_{j-1}$. Hence the right-hand side of~\eqref{eq:cvunif} tends to (0): the first term vanishes by the induction hypothesis, while the integral term tends to (0) by Lemma~\ref{prop:proprietaperodeversgenerale}, iii., and the dominated convergence theorem.

If $\sigma_{j-1}=-1$, the same reasoning applies using formula~\eqref{eq:fepsa} instead of \eqref{eq:fepsb} and the claim is proved.

We now show that, for $\varepsilon>0$ small enough, distinct pieces intersect only at the prescribed junctions. Let $j\neq h$. If $I_{j}\cap I_{h}=\emptyset$, $\Sigma_{j}^\varepsilon\cap\Sigma_{h}^\varepsilon=\emptyset$. We therefore assume 
 \begin{equation}\label{eq:nonempty}
     I_{j}\cap I_{h}\neq \emptyset. 
 \end{equation}
 
 We first treat adjacent profiles near the junction radii, namely $h=j+1$. By formulas~\eqref{eq:fepsa} and \eqref{eq:fepsb} we have that $\operatorname{sgn}({f_j^\varepsilon}')=\operatorname{sgn}({f_j}')$ does not depend on $\varepsilon>0$.
 By~\eqref{eq:monotonaj}, we thus have that there exists $\rho=\min_j\rho_j>0$ such that for every $\varepsilon>0$
 \begin{equation}\label{eq:defrho}
\begin{aligned}
    \delta_j(a_j){f_j^\varepsilon}'>0 \quad &\text{on}\quad I_j\cap(a_j-\rho,a_j+\rho),\\
    \delta_{j+1}(a_j){f_{j+1}^\varepsilon}'>0 \quad &\text{on}\quad I_{j+1}\cap(a_j-\rho,a_j+\rho).
\end{aligned}
\end{equation}
 Moreover, since $I_j\cap I_{j+1}\neq\emptyset$, \(a_j\) is a common endpoint of the two intervals, which lie on the same side of \(a_j\). 
Since \(I_j\) and \(I_{j+1}\) lie on the same side of \(a_j\),
we have \(\sigma_{j+1}=-\sigma_j\). Hence
\eqref{eq:junc_deltasigma} gives
\[
    \delta_j(a_j)\delta_{j+1}(a_j)=-1.
\]
 Therefore, calling 
 $z_{j}^\varepsilon
 =f_{j}^\varepsilon(a_{j})=f_{j+1}^\varepsilon(a_{j})$,
 we have that 
 \begin{equation}
     (f_{j}^\varepsilon-z_{j}^\varepsilon)(f_{j+1}^\varepsilon-z_{j}^\varepsilon)<0 \quad \text{in} \quad (a_j-\rho,a_j+\rho)\cap I_{j}\cap I_{j+1},
 \end{equation}
 so that $f_j^\varepsilon(r)\neq f_{j+1}^\varepsilon(r)$ for all $r\in (a_j-\rho,a_j+\rho)\cap I_{j}\cap I_{j+1}$ and for every $j=1,\dots, N-1$.
 
It remains to exclude intersections outside the neighborhoods
\(U_j=(a_j-\rho,a_j+\rho)\) and between non-consecutive profiles. For
\(j\neq h\), we introduce compact sets \(K_{j,h}\subset \bar I_j\cap\bar I_h\)
by setting
\[
K_{j,h}=
\begin{cases}
\bar I_j\cap\bar I_h\setminus U_j, & h=j+1,\\
\bar I_j\cap\bar I_h\setminus U_h, & h=j-1,\\
\bar I_j\cap\bar I_h, & |j-h|>1.
\end{cases}
\]
The limit profiles \(f_j\) and \(f_h\) have disjoint graphs on \(K_{j,h}\):
indeed, intersections between consecutive pieces can occur only at the
junction radii, which have been removed, while non-consecutive pieces are
disjoint by the simplicity of \(\gamma\). By compactness and since the number
of pairs is finite, there exists \(\eta>0\) such that
\[
|f_j(r)-f_h(r)|\ge \eta
\]
for every \(j\neq h\) and every \(r\in K_{j,h}\).
By the uniform convergence of $f_j^\varepsilon$ to $f_j$ on $\bar I_j$, and using again the finiteness of $N$, for $\varepsilon>0$ sufficiently small we have
\begin{equation}
    \sup_{\bar I_j}|f_j^\varepsilon-f_j|<\frac{\eta}{3}
    \qquad
    \text{for every } j=1,\dots,N.
\end{equation}
Therefore, for every $j\neq h$ and every $r\in K_{j,h}$, using the inequality $|A+B-C|\ge |A|-|B|-|C|$ with $A=f_j(r)-f_h(r)$, $B=f_j^\varepsilon(r)-f_j(r)$ and $C=f_h^\varepsilon(r)-f_h(r)$, we obtain
\begin{equation}
\begin{aligned}
    |f_{j}^\varepsilon(r)-f_{h}^\varepsilon(r)|
    &\geq |f_{j}(r)-f_{h}(r)|
    - |f_{j}^\varepsilon(r)-f_{j}(r)|
    - |f_{h}^\varepsilon(r)-f_{h}(r)| \\
    &\geq \eta-\frac{\eta}{3}-\frac{\eta}{3}>0.
\end{aligned}
\end{equation}
Thus no intersections occur outside the prescribed junctions. 

\smallskip

\textit{Step III: Gluing and regularity of the approximating boundary.}
Set
\begin{equation}\label{eq:boundary-epsilon}
    \Sigma_\varepsilon
    =
    \bigcup_{j=1}^N\overline{\Sigma_j^\varepsilon}.
\end{equation}
By construction, the closures of two consecutive pieces meet along
the corresponding junction circle:
\begin{equation}\label{eq:junction-circle-epsilon}
    \overline{\Sigma_j^\varepsilon}
    \cap
    \overline{\Sigma_{j+1}^\varepsilon}
    =
    C_j^\varepsilon,
    \qquad j=1,\ldots,N-1,
\end{equation}
where
\[
    C_j^\varepsilon
    =
    \bigl\{
        (x,y,z_j^\varepsilon)\in\R^3:
        x^2+y^2=a_j^2
    \bigr\},
    \qquad
    z_j^\varepsilon
    =
    f_j^\varepsilon(a_j)
    =
    f_{j+1}^\varepsilon(a_j).
\]
The closures of the first and last pieces contain the two poles on
the rotation axis. By Step~II, no other intersections occur. Hence
\(\Sigma_\varepsilon\) is an embedded topological sphere. After establishing \(C^2\)-regularity below, we define
\(\Omega_\varepsilon\) as the bounded connected component of
\(\R^3\setminus\Sigma_\varepsilon\). In particular,
\[
    \partial\Omega_\varepsilon=\Sigma_\varepsilon.
\]
In the interior of each interval \(I_j\), Proposition~\ref{prop:existence_explicit_epsilon}
gives \(f_j^\varepsilon\in C^2(I_j)\). Hence the associated surfaces of
revolution are \(C^2\) away from the junction circles. We now discuss the
regularity at the endpoints of the intervals \(I_j\), distinguishing between
the axis \(r=0\) and the junction radii \(a_j\), \(j=1,\dots,N-1\).

For \(j=1,N\), the interval \(I_j\) has one endpoint on the axis \(r=0\).
We show that \(\Sigma_j^\varepsilon\) is \(C^2\) there. Let
\(q=(q_1,q_2)\in\mathbb R^2\),  with \(r(q)\in I_j\). For
\(i,\ell\in\{1,2\}\) and \(r(q)>0\), the chain rule gives
\[
\partial_i\partial_\ell \big(f_j^\varepsilon(r(q))\big)
=
\frac{q_iq_\ell}{r(q)^2}
\left(
{f_j^\varepsilon}''(r(q))
-
\frac{{f_j^\varepsilon}'(r(q))}{r(q)}
\right)
+
\delta_{i\ell}\frac{{f_j^\varepsilon}'(r(q))}{r(q)}.
\]
Using the expansion of \({f_j^\varepsilon}'\) at \(0\), see
\eqref{eq:tayloreps0}, and since \(r(q)\to0\) as
\(r^{-1}(I_j)\ni q\to0\), we obtain
\[
\lim_{q\to0}
\partial_i\partial_\ell \big(f_j^\varepsilon(r(q))\big)
=
-\frac{\sigma_j\varphi(0)\varepsilon}{2}\,\delta_{i\ell}.
\]
Thus the second derivatives extend continuously at the axis, and
\(\Sigma_j^\varepsilon\) is \(C^2\) at \(r=0\).

It remains to prove regularity at the junction circles \(C_j^\varepsilon\), \(j=1,\dots,N-1\). Since \({f_j^\varepsilon}'\) and \({f_{j+1}^\varepsilon}'\) blow up at \(a_j\), see \eqref{eq:bc_int}, the boundary of \(\Omega_\varepsilon\) near \(C_j^\varepsilon\) must be described as the rotation of the graph of a function of the variable \(z\), and regularity has to be analyzed in this parametrization.

By \eqref{eq:defrho}, there exists \(\rho>0\) such that \(f^\eps_j\) and \(f^\eps_{j+1}\) are strictly monotone, hence invertible, on \(U_j=(a_j-\rho,a_j+\rho)\). Moreover, \eqref{eq:junc_deltasigma} implies that
\begin{equation}
    \overline{f_j^\varepsilon(I_j\cap U_j)}\cap \overline{f_{j+1}^\varepsilon(I_{j+1}\cap U_j)}=\{z_j^\varepsilon\},
\end{equation}
where $z_j^\varepsilon=f_j^\varepsilon(a_j)$.
Therefore, letting $H_j^\varepsilon
=
f_j^\varepsilon(I_j\cap U_j)
\cup
f_{j+1}^\varepsilon(I_{j+1}\cap U_j)$, the function $h_j^\varepsilon:\overline{H_j^\varepsilon}\to \R^+$ defined by 
\begin{equation}
    h_j^\varepsilon(z)=\begin{cases}
        (f_j^\varepsilon)^{-1}(z)&\text{ if }z\in \overline{f_j^\varepsilon(I_j\cap U_j)}\\
        (f_{j+1}^\varepsilon)^{-1}(z)&\text{ if }z\in \overline{f_{j+1}^\varepsilon(I_{j+1}\cap U_j)}
    \end{cases}
\end{equation}
is well defined and $C^2$-regular on $H_j^\varepsilon$.
It therefore remains to prove that \(h_j^\varepsilon\) is \(C^2\) at \(z_j^\varepsilon\). For this purpose, observe that \eqref{eq:bc_int} implies that ${h_j^\varepsilon}'(z_j^\varepsilon)=0$ ensuring $C^1$-regularity at $z_j^\varepsilon$. We compute the second derivative of $h_j^\varepsilon$ at $z\neq z_j^\varepsilon$. To this purpose, observe that a straightforward computation rewrites the prescribed $g_\varepsilon$-mean curvature equation~\eqref{eq:MCHeis-eps-lem} as 
\begin{equation}\label{eq:feps''}
    {f_j^\varepsilon}''(r)
    =-\frac{1}{\frac{r^2}{4}+\varepsilon^2}\left\{\sigma_j\varphi(r)\left({f_j^\varepsilon}'(r)^2+\frac{r^2}{4}+\varepsilon^2\right)^\frac{3}{2}+\frac{{f_j^\varepsilon}'(r)^3}{r}\left(1+\frac{\varepsilon^2}{{f_j^\varepsilon}'(r)^2}\right)\right\},
\end{equation}
so that, letting $r_z=h_j^\varepsilon(z)$ for $z\in f(I_j\cap U_j)$ we have 
\begin{equation}
\begin{split}
    {h_j^\varepsilon}''(z)&=
        -\frac{{f_j^\eps}''(r_z)}{{f_j^\eps}'(r_z)^3}\\
        &= \frac{1}{\frac{r_z^2}{4}+\varepsilon^2}
        \left\{
        \sigma_j\varphi(r_z)\,
        \operatorname{sgn}\!\left({f_j^\varepsilon}'(r_z)\right)
        \left(
        1+\frac{\frac{r_z^2}{4}+\varepsilon^2}
        {\left({f_j^\varepsilon}'(r_z)\right)^2}
        \right)^{\frac{3}{2}}
        +\frac{1}{r_z}
        \left(
        1+\frac{\varepsilon^2}
        {\left({f_j^\varepsilon}'(r_z)\right)^2}
        \right)
        \right\}.
\end{split}
\end{equation}

Since \(h_j^\varepsilon(z)\to a_j\) and
\(\lvert {f_j^\varepsilon}'(h_j^\varepsilon(z))\rvert\to+\infty\)
as \(z\to z_j^\varepsilon\), it follows that
\begin{equation}
\begin{split}
    \lim_{\substack{z\to z_j^\varepsilon\\
    z\in f_j^\varepsilon(I_j\cap U_j)}}
    {h_j^\varepsilon}''(z)
    &=
    \frac{1}{\frac{a_j^2}{4}+\varepsilon^2}
    \left(
        \sigma_j\delta_j(a_j)\varphi(a_j)
        +\frac{1}{a_j}
    \right),\\
    \lim_{\substack{z\to z_j^\varepsilon\\
    z\in f_{j+1}^\varepsilon(I_{j+1}\cap U_j)}}
    {h_j^\varepsilon}''(z)
    &=
    \frac{1}{\frac{a_j^2}{4}+\varepsilon^2}
    \left(
        \sigma_{j+1}\delta_{j+1}(a_j)\varphi(a_j)
        +\frac{1}{a_j}
    \right).
\end{split}
\end{equation}
By~\eqref{eq:junc_deltasigma} we therefore deduce that 
\begin{equation}
    \lim_{z\to {z_j^\varepsilon}^+}{h_j^\varepsilon}''(z)=\lim_{z\to {z_j^\varepsilon}^-}{h_j^\varepsilon}''(z),
\end{equation}
hence implying $C^2$-regularity of $h_j^\varepsilon$. This implies that $\partial\Omega_\varepsilon$ is also $C^2$-regular around $C_j$, since it can locally be described as the zero level set of the function $F(x,y,z) =x^2 + y^2 - (h_j^\varepsilon(z))^2$, which is $C^2$ with non-vanishing gradient around $C_j$ equals to $(2x, 2y, -2h_j^\varepsilon(z) {h_j^\varepsilon}'(z)) \neq 0$.

This proves that
\(\Sigma_\varepsilon=\partial\Omega_\varepsilon\) is an embedded
\(C^2\)-regular surface. Finally, the uniform convergence of
\(f_j^\varepsilon\) to \(f_j\) on \(\overline I_j\), for every
\(j=1,\ldots,N\), implies that
\[
    d_{\mathcal H}
    \bigl(\partial\Omega_\varepsilon,\partial\Omega\bigr)
    \longrightarrow0
    \qquad\text{as }\varepsilon\to0^+.
\]
This concludes the proof.
\end{proof}

\subsection{Proof of Theorem~\ref{thm:approx}}
\label{ss:analyticapprox}
We are now in a position to prove Theorem~\ref{thm:approx}. 
We include here the proof of a very well-known polynomial approximation result. We are going to employ it below in order to approximate $\gamma_1$ through polynomials, this way obtaining axysimmetric sets satisfying the assumptions of Proposition~\ref{thm:prescrmcanalit}.
In the following, given $k\in \N$, $f^{(k)}$ denotes the $k$-th order derivative of a $C^k$ function $f$.

\begin{lem}[Approximation with fixed boundary data]\label{lem:approxc2anal}
    Let $a<b$ and let $h\in C^2([a,b])$. 
    Then there exists a sequence of polynomials $\{P_n\}_{n\ge1}$ such that 
    \begin{equation}
    \label{eq:poly-punti}
    P_n^{(k)}(a)=h^{(k)}(a),\qquad P_n^{(k)}(b)=h^{(k)}(b)
    \quad\text{for } k=0,1,2,
    \end{equation}
    and 
    \begin{equation}
    \label{eq:limit-poly}
     \lim_{n \to +\infty} P_n = h \quad \text{in} \, \, C^2([a, b]).
    \end{equation}
    \end{lem}
    
    \begin{proof}
    By a generalization of \emph{Stone-Weierstrass Theorem} (see, for instance, \cite[Theorem~1.6.2]{Narasimhan1973}), if $h\in C^2([a,b])$ then there exist polynomials $Q_n$ such that 
    \[
    Q_n\to h \text{ in the  }C^2([a,b]).
    \]
However, these polynomials $Q_n$ do not necessarily match the boundary data of $h$. We define the boundary deviations for derivatives of order $k=0,1,2$ as:
    \[
    \delta_{k,n}^a = h^{(k)}(a) - Q_n^{(k)}(a), \qquad 
    \delta_{k,n}^b = h^{(k)}(b) - Q_n^{(k)}(b).
    \]
    Since $Q_n \to h$ in $C^2$, it follows that $\delta_{k,n}^a \to 0$ and $\delta_{k,n}^b \to 0$ as $n \to \infty$.
    
    We now construct a correction polynomial $R_n(x)$ to be added to $Q_n$ in order to match \eqref{eq:poly-punti}. 
    
    To this end, explicitly carrying out the standard construction of Hermite interpolation as presented in \cite[Theorem~2.1.5.2]{StoerBulirsch2002} and in the sequel, we introduce polynomials in $[0,1]$, denoted by $\eta_0, \eta_1, \eta_2$ such that
    \begin{equation}
    \label{eq:condizionieta}
    \eta_j^{(k)} = \delta_{jk} \quad (\text{Kronecker delta}), \qquad \text{for } j,k \in \{0,1,2\}
    \end{equation}
    at $t = 0$
    while vanishing along with their first two derivatives at $t=1$. One can check that
    \begin{align}
    \eta_0(t) &= 1-10t^3+15t^4-6t^5, \\
    \eta_1(t) &= t-6t^3+8t^4-3t^5, \\
    \eta_2(t) &= \tfrac12 t^2-\tfrac32 t^3+\tfrac32 t^4-\tfrac12 t^5
    \end{align}
    satisfy the requirements.
    Consider also the polynomials 
    \[
    \tilde{\eta}_j(t) = (-1)^j \eta_j(1-t)
    \]
    that by symmetry
    satisfy the conditions \eqref{eq:condizionieta} at $t = 1$,
    while vanishing along with their first two derivatives at $t=0$.
    
    We can thus define, for $L=b-a$,
    \[
    R_n(x) = \sum_{j=0}^2 \left( \delta_{j,n}^a \, L^j \, \eta_j\left(\tfrac{x-a}{L}\right) + \delta_{j,n}^b \, L^j \, \tilde{\eta}_j\left(\tfrac{x-a}{L}\right) \right).
    \]
  Observe that
    $
    R_n^{(k)}(a) = \delta_{k,n}^a,\quad$ and $R_n^{(k)}(b) = \delta_{k,n}^b.
    $
    Finally, we define $P_n = Q_n + R_n$, that thus satisfies \eqref{eq:poly-punti}.
     Since $\delta_{k,n}^{a,b} \to 0$, we have $R_n \to 0$ in $C^2([a, b])$ as $n \to +\infty$, implying \eqref{eq:limit-poly} and concluding the proof.
\end{proof}

We finally conclude the proof of Theorem \ref{thm:approx}, that from this point becomes fairly standard.
\begin{proof}[Proof of Theorem~\ref{thm:approx}]
Let
\(\gamma\in C^2([0,1],\R^2_+)\) be a generating profile for
\(\partial\Omega\). By regularity at the characteristic poles, there exist
\(0<a<b<1\) such that $\gamma_1'\neq0$ on $[0,a]\cup[b,1]$.

Let \(P_n\) be the polynomial approximations of \(\gamma_1\) given by
Lemma~\ref{lem:approxc2anal}, and set
\begin{equation}
    \gamma_n(t)
    =
    \begin{cases}
        \gamma(t),
        &t\in[0,a]\cup[b,1],\\
        (P_n(t),\gamma_2(t)),
        &t\in[a,b].
    \end{cases}
\end{equation}
For \(n\) sufficiently large, \(\gamma_n\) generates an axisymmetric
set \(\Omega^n\) satisfying the assumptions of
Proposition~\ref{thm:prescrmcanalit}. Moreover, after passing to a
subsequence,
\[
    d_{\mathcal H}(\partial\Omega^n,\partial\Omega)<\frac1n,
    \qquad
    c_1-\frac1n\leq H^n\leq c_2+\frac1n.
\]
where $H^n$ is the horizontal mean curvature of $\partial\Omega^n$.
For every \(n\), Proposition~\ref{thm:prescrmcanalit} provides sets
\(\Omega_\varepsilon^n\) such that
\[
    d_{\mathcal H}
    (\partial\Omega_\varepsilon^n,\partial\Omega^n)
    \longrightarrow0
    \qquad\text{as }\varepsilon\to0^+,
\]
and whose \(g_\varepsilon\)-mean curvature is \(H^n\).
Choose a decreasing sequence \(\varepsilon_n\to0\) such that
\[
    d_{\mathcal H}
    (\partial\Omega_\varepsilon^n,\partial\Omega^n)<\frac1n
    \qquad\text{for }0<\varepsilon\leq\varepsilon_n,
\]
and define \(\Omega_\varepsilon=\Omega_\varepsilon^n\) if $\varepsilon\in(\varepsilon_{n+1},\varepsilon_n]$.
It follows that
\[
    d_{\mathcal H}(\partial\Omega_\varepsilon,\partial\Omega)
    <\frac2n,
    \qquad
    c_1-\frac1n\leq H_\varepsilon\leq c_2+\frac1n
\]
whenever \(\varepsilon\in(\varepsilon_{n+1},\varepsilon_n]\). This
proves the statement.
\end{proof}

\begin{rem}\label{rem:pansu_tomter}
    If $\Omega=\mathcal P_R$ is the Pansu sphere of radius $R>0$ defined in~\eqref{eq:pansub}, then the family $\{\Omega_\varepsilon\}_{\varepsilon>0}$ obtained via the construction described above recovers the family of Tomter spheres~\cite{Tomter93,FranceschiMontefalconeMonti19}.
    For the surface $\partial\mathcal P_R$, the auxiliary function $\psi$ is given by
    \[
        \psi(r) = \int_0^r s\,\varphi(s)\,ds = \int_0^r s \frac{2}{R} \, ds = \frac{r^2}{R}.
    \]
    Applying our construction with the prescribed curvature $\varphi = {2}/{R}$, we derive the profile 
    \[
        f_\varepsilon(r) = \int_r^R \frac{s}{2}\, \sqrt{\frac{s^2 + 4\varepsilon^2}{R^2 - s^2}}\,ds, \qquad r \in [0,R].
    \]
    corresponding to the generating profile of the Tomter spheres of radius $R$ in $\mathbb{H}_\varepsilon$ (see \cite{Tomter93}).
\end{rem}

\section{The main existence results}\label{s:existence}
The proof of our main existence theorem boils down to obtaining gradient bounds for
$u_\varepsilon^\delta$, solutions to \eqref{RepsMCFproblreg}, that are stable in both $\delta$ and $\eps$. This will allow for a standard compactness argument that will easily lead to viscosity solutions to \eqref{eq:HMCFeq}.

\subsection{Right invariant maximum principle}\label{ss:maximumpr}
The first step we carry out  consists in a maximum principle for $|\tilde\grad_\eps u_\eps^\delta|$. In fact, $|\tilde\nabla_\eps u_\eps^\delta|^2$ turns out to be subsolution to an elliptic equation, due to $\tilde\nabla$ and $\nabla$ commuting, see formula~\eqref{eq:vfcomm}. We do not expect a similar phenomenon to hold for $|\nabla_\eps u_\eps^\delta|^2$. The main ingredient is a right invariant Bochner identity due to Garofalo \cite[Proposition~3.4]{Garofalo23}, that we extend here to include Riemannian Heisenberg metrics. 

\begin{prop}[Right Bochner-type identity]\label{prop:bochnermia}
    Let $\eps>0$, then for any function $f \in C^3(\R^3)$, the following identity holds:
    \begin{equation}\label{eq:bochnermia}
        \Delta_\varepsilon (|\tilde{\nabla}_\varepsilon f|^2) = 2\left[ |\tilde{\nabla}_\varepsilon (Xf)|^2 + |\tilde{\nabla}_\varepsilon (Yf)|^2 + \varepsilon^2|\tilde{\nabla}_\varepsilon (Zf)|^2 \right] + 2\langle \tilde{\nabla}_\varepsilon f, \tilde{\nabla}_\varepsilon (\Delta_\varepsilon f) \rangle.
    \end{equation}
\end{prop}
\begin{proof}
    By a direct computation
    \begin{align}
        \Delta_\eps (|\Tilde{\nabla}_\eps f|^2)&=XX(|\Tilde{\nabla}_\eps f|^2)+YY(|\Tilde{\nabla}_\eps f|^2)+\eps^2 ZZ(|\Tilde{\nabla}_\eps f|^2)=\\ 
        &=XX(\Tilde{\nabla}_\eps f\cdot\Tilde{\nabla}_\eps f)+YY(\Tilde{\nabla}_\eps f\cdot\Tilde{\nabla}_\eps f)+\eps^2 ZZ(\Tilde{\nabla}_\eps f\cdot\Tilde{\nabla}_\eps f)=\\
        &=X(2\Tilde{\nabla}_\eps f\cdot X(\Tilde{\nabla}_\eps f))+Y(2\Tilde{\nabla}_\eps f\cdot Y(\Tilde{\nabla}_\eps f))+\eps^2 Z(2\Tilde{\nabla}_\eps f\cdot Z(\Tilde{\nabla}_\eps f))=\\
        &=2\Big[|X(\Tilde{\nabla}_\eps f)|^2+ \Tilde{\nabla}_\eps f\cdot XX(\Tilde{\nabla}_\eps f)+|Y(\Tilde{\nabla}_\eps f)|^2+ \Tilde{\nabla}_\eps f\cdot YY(\Tilde{\nabla}_\eps f) +\\
        & \qquad + \eps^2 |Z(\Tilde{\nabla}_\eps f)|^2+ \eps^2 \Tilde{\nabla}_\eps f\cdot ZZ(\Tilde{\nabla}_\eps f) \Big].
    \end{align}
    As observed in \eqref{eq:vfcomm}, the left-invariant vector fields $\{X, Y, Z\}$ commute with the right-invariant gradient $\tilde{\nabla}_\varepsilon$. Consequently:
    \[
        X(\tilde{\nabla}_\varepsilon f) = \tilde{\nabla}_\varepsilon (Xf), \quad Y(\tilde{\nabla}_\varepsilon f) = \tilde{\nabla}_\varepsilon (Yf), \quad Z(\tilde{\nabla}_\varepsilon f) = \tilde{\nabla}_\varepsilon (Zf).
    \]
    It follows that $\Delta_\varepsilon$ commutes with $\tilde{\nabla}_\varepsilon$ as well. Substituting these commutation relations into the expression above yields:
    \begin{align}
        \Delta_\varepsilon (|\tilde{\nabla}_\varepsilon f|^2) &= 2\Big[ |\tilde{\nabla}_\varepsilon (Xf)|^2 + |\tilde{\nabla}_\varepsilon (Yf)|^2 + \varepsilon^2 |\tilde{\nabla}_\varepsilon (Zf)|^2 \Big] + 2\langle \tilde{\nabla}_\varepsilon f, \tilde{\nabla}_\varepsilon (\Delta_\varepsilon f) \rangle.
    \end{align}
\end{proof}

\begin{prop}
\label{thm:gradest_ed}
    Let $D\subset\R^3$ be a bounded, smooth open set. Let $u_\eps^\delta \in C^\infty(\overline{D})$ be a  solution to \eqref{RepsMCFproblreg} in $D$, then
    \begin{equation}
    \label{eq:gradest_edtilde}
    \|\tilde \nabla_\eps u_\eps^\delta\|_{L^\infty(D)}\le\|\tilde \nabla_\eps u_\eps^\delta\|_{L^\infty(\partial D)}.
    \end{equation}
\end{prop}

\begin{proof}
Applying \eqref{eq:bochnermia} to $u_\eps^\delta$ we get, plugging in \eqref{RepsMCFproblreg},
\begin{align}\label{eq:epsdeltauno}
        \Delta_\eps \varphi &= 2\mathcal{A} - \frac{2\mathcal{B}}{|\nabla_\eps u_\eps^\delta|^2+\delta^2}\\
        &\quad - 2\frac{\nabla_\eps\nabla_\eps u_\eps^\delta(\nabla_\eps u_\eps^\delta,\nabla_\eps u_\eps^\delta)}{(|\nabla_\eps u_\eps^\delta|^2+\delta^2)^2}\langle\nabla_\eps u_\eps^\delta,\nabla_\eps\varphi\rangle 
        + \frac{\nabla_\eps\nabla_\eps \varphi(\nabla_\eps u_\eps^\delta,\nabla_\eps u_\eps^\delta)}{|\nabla_\eps u_\eps^\delta|^2+\delta^2}+ \frac{2\nabla_\eps\nabla_\eps u_\eps^\delta(\nabla_\eps u_\eps^\delta,\nabla_\eps \varphi)}{|\nabla_\eps u_\eps^\delta|^2+\delta^2},
    \end{align}
   where we let $\varphi = |\Tilde{\nabla}_\eps u_\eps^\delta|^2$,
   \[
   \mathcal A = |\Tilde{\nabla}_\eps (X u_\eps^\delta)|^2+|\Tilde{\nabla}_\eps (Y u_\eps^\delta)|^2+\eps^2|\Tilde{\nabla}_\eps (Z u_\eps^\delta)|^2
   \]
 and
   \[
        \mathcal{B} = \langle \nabla_\eps u_\eps^\delta, \Tilde{\nabla}_\eps (X u_\eps^\delta) \rangle^2 + \langle \nabla_\eps u_\eps^\delta, \Tilde{\nabla}_\eps (Y u_\eps^\delta) \rangle^2 + \eps^2 \langle \nabla_\eps u_\eps^\delta, \Tilde{\nabla}_\eps (Z u_\eps^\delta) \rangle^2.
    \]
   We now verify that the positive term $2\mathcal{A}$ dominates the negative term involving $\mathcal{B}$.
    Applying the Cauchy-Schwarz inequality $\langle v, w \rangle^2 \le |v|^2 |w|^2$ to each summand in $\mathcal{B}$, we have:
    \[
        \langle \nabla_\eps u_\eps^\delta, \Tilde{\nabla}_\eps (V u_\eps^\delta) \rangle^2 \le |\nabla_\eps u_\eps^\delta|^2 \, |\Tilde{\nabla}_\eps (V u_\eps^\delta)|^2,
    \]
    for any vector field $V \in \{X, Y, \eps Z\}$. Summing these contributions yields:
    \[
        \mathcal{B} \le |\nabla_\eps u_\eps^\delta|^2 \, \mathcal{A}.
    \]

    Consequently we obtain the estimate:
\begin{equation}\label{eq:epsdelta2}
        \frac{2\mathcal{B}}{|\nabla_\eps u_\eps^\delta|^2+\delta^2} \le \frac{2|\nabla_\eps u_\eps^\delta|^2}{|\nabla_\eps u_\eps^\delta|^2+\delta^2} \mathcal{A} \le 2\mathcal{A},
     \end{equation}
   that, coupled with \eqref{eq:epsdeltauno} and \eqref{eq:epsdelta2}, results in
        \begin{align}
        \begin{split}
            \mathcal{L}\varphi &= \Delta_\eps \varphi - \frac{\nabla_\eps\nabla_\eps \varphi(\nabla_\eps u_\eps^\delta,\nabla_\eps u_\eps^\delta)}{|\nabla_\eps u_\eps^\delta|^2+\delta^2}
            + 2\frac{\nabla_\eps \nabla_\eps u_\eps^\delta (\nabla_\eps u_\eps^\delta ,\nabla_\eps u_\eps^\delta)}{(|\nabla_\eps u_\eps^\delta |^2+\delta^2)^2}\langle\nabla_\eps u_\eps^\delta,\nabla_\eps\varphi\rangle
             \\
            &\quad\quad\quad\quad\quad\quad\quad\quad - \frac{2\nabla_\eps\nabla_\eps u_\eps^\delta(\nabla_\eps u_\eps^\delta,\nabla_\eps \varphi)}{|\nabla_\eps u_\eps^\delta|^2+\delta^2} \geq 0.
        \end{split}
        \end{align}
        The operator $\mathcal L$ being linear and elliptic yields \eqref{eq:gradest_edtilde} by the maximum principle. 
    \end{proof}
    With the above result in mind, the following subsections will be devoted to finding uniform upper bounds for $\tilde{\nabla} u_\eps^\delta$ on the boundary.

    \subsection{Letting $\delta\to 0$: the Riemannian Heisenberg mean curvature flow.}
   In the following statement, we gather various properties of solutions $u_\eps^\delta$ to \eqref{RepsMCFproblreg} and of its limit $u_\eps$. They are well known,   yet we include an argument for the smooth convergence in the regular set, that is item (v). In fact, we were not able to find an explicit such statement for the arrival time formulation we are interested in. 
The function $u_\eps$ will turn out to be a viscosity solution to \eqref{RepsMCFprobl}. This has the very same meaning as in the horizontal case described in Definition \ref{dfn:viscsol}, where the Hessian, the Laplacian and the gradients are those induced by $g_\eps$ and $\eta$ is allowed to be any vector in $T_p\R^3$ with $|\eta| \leq 1$.
\begin{thm}
\label{thm:mcfriem}
Let $D \subset \R^3$ be a bounded open set with smooth boundary. Fix $\eps > 0$. Assume that $H_\eps (p) > 0$ for any $p \in \partial D$. Then, the following hold.
\begin{itemize}
\item[(i)]for any $\delta > 0$ there exists a unique smooth solution $u_\eps^\delta \in C^\infty(\overline{D})$ to \eqref{RepsMCFproblreg}. 
\item [(ii)] There exists a function $u_\eps \in C^0(\overline{D})$ such $u_\eps^\delta \to u_\eps$ in $C^0(K)$ as $\delta \to 0$ for any compact $K \subset D$.  
\item [(iii)]$u_\eps$ is the unique viscosity solution to \eqref{RepsMCFprobl}. 
\item [(iv)] $|\nabla _\eps^\delta u_\eps^\delta| \to |\nabla_\eps u_\eps|$ locally uniformly in $K \subset \mathrm{Reg}(u_\eps)$, where
\begin{equation}
\label{eq:regularset}
\mathrm{Reg}(u_\eps) = \{p \in D  \, | u_\eps \, \, \text{is smooth in $p$ and} \,|\nabla_\eps u_\eps|(p) \neq 0\}
\end{equation}
\item[(v)] There exists a tubular neighborhood $U_r = \{p \in D \, | d_\eps (p, \partial D) < r\}$ such that $u_\varepsilon \in C^\infty(\overline{U_r})$. In particular, $U_r \subset\mathrm{Reg} (u_\eps)$.
\end{itemize}
\end{thm}

 \begin{proof} 
 The Euclidean version of items (i), (ii), (iii) and (v) are all worked out in the fundamental paper by Evans-Spruck \cite[Section 7.3]{evansspruck91}. The claimed Riemannian version of (i), (ii) and (iii) follow with minor modifications of \cite[Theorem 7.4]{evansspruck91} once suitable  $C^1(\overline{D})$ estimates, independent of $\delta$, are proved for $u_\eps^\delta$.  Such estimates can be found combining \cite[Lemma 3.1, Lemma 3.2 and Lemma 3.3]{Schulze2008}, that in fact deal with more general flows. Property (v) follows from the consistency with the classical mean curvature flow, namely, there exists $\overline{t} > 0$ such that $\{u_\eps = t\} = \Sigma_t$ for $0 \leq t < \overline{t}$, where the $\{\Sigma_t\}_{t \in [0, \overline{t})}$ is the mean curvature flow of $\partial D$, see \cite[Lemma 4.2]{Schulze2008} for a precise statement.
   We focus now on the local smooth convergence on an arbitrary compact set $K \subset \operatorname{Reg}(u_\varepsilon)$. We detail an argument that one can find outlined, e.g., along the proof of \cite[Theorem 3]{white13}.
   Consider the product manifold $D \times \mathbb{R}$ equipped with the product metric $ g_\varepsilon + ds^2$, where $ds^2$ is the Euclidean metric on $\mathbb{R}$. We introduce the auxiliary functions $U, U_\delta :D \times \R \to \R$ 
    \begin{equation}
        U_\delta(p,s) = u_\varepsilon^\delta(p) - \delta s, \qquad U(p,s) = u_\varepsilon(p).
    \end{equation}
    
    As observed first in \cite[Section 7.1]{evansspruck91}, the level sets $\Gamma_t^\delta = \{U_\delta = t\}$ of $U_\delta$  smoothly evolve by mean curvature flow $(D \times \R, g_\eps + ds^2)$,
    that is $U_\delta$ satisfies 
    \begin{equation}
    \label{eq:mcflifted}
   -\Delta U_\delta + \frac{\nabla \nabla U_\delta(\nabla U_\delta, \nabla U_\delta)}{|\nabla U_\delta|^2} = 1,
\end{equation}
where the quantities are computed in terms of the Levi Civita connection of $g_\eps + ds^2$.
    By \cite[Section 5]{Schulze2008} (see \cite[Proposition 2.5]{metzgershulze08} for a precise statement in the setting of flat $\R^n$)  
  the sequence of flows $\{\Gamma_t^\delta\}_{t\geq 0}$ converges to $\{\Gamma_t\}_{t \geq 0}$ where $\Gamma_t =\{u_\eps = t\} \times \R$ in the so called weak Brakke flow sense. Let now $p \in \mathrm{Reg} (u_\varepsilon)$. Then, $(p,s,  t)$ is a smooth point for the hypersurface $\Gamma_t \times \R \times \R^+$ sitting inside $D \times \R \times \R^{+}$ for any $s \in \R$. By Brakke's regularity theorem, in particular by \cite[Theorem 7.3]{white05}, we deduce that there exists a neighborhood $V_{(p,t)}$ of $(p, t)$ 
  such that, at a fixed $\overline{s}$,
  $\Gamma_{\tau}^\delta \to \Gamma_\tau$   smoothly in any compact subset of $V_{(p,t)}$.
 Observe now that the mean curvature of the hypersurface $\Gamma_\tau^\delta$ induced by the product metric of $D \times \R$ is given, through \eqref{eq:mcflifted}, by
  \[
H_\eps^\delta = \frac{1}{\sqrt{|\nabla_\varepsilon u_\varepsilon^\delta|^2 + \delta^2}}.
  \]
Similarly, the mean curvature of the level set $\Gamma_\tau$ on regular points is computed through \eqref{RepsMCFprobl} by 
\[
H_\eps = \frac{1}{{|\nabla_\varepsilon u_\varepsilon|}}.
  \]
  We deduce, that, given $p \in \mathrm{Reg}(u_\eps)$, there exists a neighborhood $V_p $ of $p$ such that $|\nabla_\eps u_\eps^\delta|  \to |\nabla_\eps u_\eps|$ in any compact subset of $V_p$. With a covering we deduce the stated locally uniform convergence in $\mathrm{Reg}(u_\eps)$. 
  \end{proof}

  \begin{rem}
  It has been proved by Colding-Minicozzi \cite[Corollary 2.2]{colding-minicozzi2016-arrival} that the arrival time function $u$ of the mean curvature flow of a strictly mean-convex bounded smooth domain $D$ is $C^{1, 1}$, and that in fact
  \[
  \mathrm{Reg}(u) = \{p \in D \, | \nabla u(p) \neq 0  \}.
  \]
 We do not need this sharper description in our work.
\end{rem}

    \subsection{Stable gradient estimates}
We derive here suitable gradient bounds for $u_\eps^\delta$ that are stable both for $\eps$ and $\delta$, in the axisymmetric or non-characteristic case.

We are now ready to establish the uniform gradient estimate for the limit solution $u_\varepsilon$. The following is the main technical result leading to the existence of the mean-convex mean curvature flow in the axisymmetric case. 

\begin{prop}[Global gradient estimate-axisymmetric case]\label{stimagradeps}
    Let $ D \subset \mathbb{R}^3$ be a bounded, smooth, axisymmetric open set. Assume that $D$ is strictly mean convex with respect to the Riemannian metric $g_\varepsilon$, i.e., $H_\varepsilon(\partial D) \ge c_\varepsilon > 0$. 
    
    Let $u_\varepsilon$ be unique  continuous viscosity solution to \eqref{RepsMCFprobl} given by Theorem \ref{thm:mcfriem}. Then, 
    \begin{equation}
    \label{eq:lip-ueps-axis}
        |u_\varepsilon(p)-u_\varepsilon(q)| \le \left(\sup_{\partial D} |\nabla_\varepsilon u_\varepsilon|\right) d(p,q) \quad \text{for all } p,q \in \overline{ D}.
    \end{equation}
    Furthermore, if $c_\varepsilon \ge c_0 > 0$ with $c_0$ independent of $\varepsilon$, then there exists a uniform constant $C > 0$, independent of $\varepsilon$, such that:
    \begin{equation}
    \label{eq:stimalip-unif-ueps}
       |u_\varepsilon(p)-u_\varepsilon(q)| \le \frac{1}{c_0} \, d(p, q) \quad \text{for all } p,q \in \overline{ D}.
    \end{equation}
\end{prop}

\begin{proof}
 By Theorem \ref{thm:mcfriem}, (v), equation \eqref{RepsMCFprobl} holds classically in a neighborhood of $\partial D$ and up to $\partial D$, and thus
     \begin{equation}
     \label{eq:gradsulbordo}
    \sup_{\partial D}|\nabla_\eps u_\eps| = \frac{1}{\min_{\partial D} H_\eps} \le  \frac{1}{c_\eps}.
    \end{equation}
    By the continuity of $\nabla_\eps u_\eps$ in such neighborhood, there exists $\tilde\eta > 0$ such that the compact set
    \[
        \Gamma_\eta = \{p \in  D \mid u_\eps(p)= \eta\}
    \]
    is disjoint from the critical set for each $\eta<\tilde{\eta}$, i.e., $|\nabla_\eps u_\eps| \neq 0$ in $\Gamma_\eta$.\\

    We consider now the approximators $u_\eps^\delta$: for every $\eta<\tilde\eta$, applying \textit{Theorem~\ref{thm:gradest_ed}} in the set $\{u_\eps\ge\eta\}$, coupled with \eqref{rem:rotinvariancegradeps}, we get
    \begin{equation}
    \label{eq:pmaxinproofaxi}
    \sup_{\{u_\eps\ge\eta\}}|\nabla_\eps u_\eps^\delta|\le \sup_{\Gamma_\eta}|\nabla_\eps u_\eps^\delta|.
    \end{equation}  
    {The above inequality implies, letting $\delta \to 0$ and using (i),  (iv) and(v) in Theorem \ref{thm:mcfriem}, that, for any two points $p, q$ in $\{u_\eps > \eta\}$,
    \[
|u_\eps (p) - u_\eps(q)| \leq \sup_{\Gamma_\eta} |\nabla_\eps u_\eps|d_\eps (p, q) \leq \sup_{\Gamma_\eta} |\nabla_\eps u_\eps|d (p, q). 
    \]
    We also used Proposition \ref{prop:deps-d} in the second inequality.
Letting $\eta \to 0$ and using again (v) in Theorem \ref{thm:mcfriem} yields \eqref{eq:lip-ueps-axis}. Using that on $\partial D$ $|\nabla u_\eps|= H_\eps^{-1}$ \eqref{eq:stimalip-unif-ueps} follows.}
\end{proof}


In the noncharacteristic case, we can a priori obtain an upper bound on the right hand side of \eqref{eq:gradest_edtilde}, arguing similarly to the Euclidean case treated in \cite{evansspruck91}. In particular, we shall prove the following result.
\begin{prop}
    \label{thm:gradest_ednoncar}
    Let $ D \subset \R^3$ be a bounded, smooth non-characteristic open set. Assume that $ D$ is strictly horizontally mean convex, i.e., $H(\partial D) \ge c > 0$. 
    Let $u_\varepsilon$ be the viscosity solution to \eqref{RepsMCFprobl} with initial datum $ D$. Then, there exists a uniform constant $C > 0$, independent of $\varepsilon$ and $\delta$ such that
    \begin{equation}\label{eq:uniform_right_est}
    \|\tilde\nabla_\eps u_\eps^\delta\|_{L^\oo( D)}\le C.
    \end{equation}
\end{prop}

Once estimate \eqref{eq:uniform_right_est} is established, it implies that the family $\{u_\varepsilon\}$ is uniformly Lipschitz continuous with respect to the right-invariant distance $\tilde{d}_\varepsilon$.  We  can then invoke the Ascoli-Arzel\'a Theorem to conclude the existence of a uniformly convergent subsequence.\\

Now we can prove the uniform right-invariant Lipschitz estimates of the viscosity solutions $\{u_\eps\}$ to problem \eqref{RepsMCFprobl} in the non-characteristic case.

\begin{proof}[Proof of Proposition~\ref{thm:gradest_ednoncar}]
 By Theorem~\ref{thm:gradest_ed}, the maximum of the gradient is attained on the boundary.
    Thus, to conclude the proof, we only need to establish a uniform bound for $\|\tilde\nabla_\varepsilon u_\varepsilon^\delta\|_{L^\infty(\partial D)}$ independent of $\varepsilon$ and $\delta$.
    
    We consider the regularized operator
    \[
        M_\delta v = \Delta_\varepsilon v - \frac{\nabla_\eps \nabla_\eps v(\nabla_\varepsilon v,\nabla_\varepsilon v)}{|\nabla_\varepsilon v|^2+\delta^2}.
    \]
    By definition, $M_\delta u_\varepsilon^\delta = -1$. Let $d_E$ be the Euclidean distance function from $\partial D$ defined as
    \begin{equation}
    d_E(p) = \min_{q \in \partial D} |p - q|,
\end{equation}
    and let
    \[
        w = \lambda g(d_E).
    \]
    We are going to choose $\lambda > 0$, $s_0 > 0$  and $g = g(s) \in C^2([0, s_0))$ later in the proof, to ensure that $w$ acts as a barrier for $u_\eps^\delta$ close to $\partial D$.
    
    
  A direct computation gives:
    \begin{itemize}
        \item $\nabla_\varepsilon w = \lambda g'\nabla_\varepsilon d$;
        \item $\nabla_\eps \nabla_\eps w = \lambda\left[g'\nabla_\eps \nabla_\eps d + g'' \nabla_\varepsilon d \otimes \nabla_\varepsilon d \right]$;
        \item $\Delta_\varepsilon w = \lambda[g'\Delta_\varepsilon d + g''|\nabla_\varepsilon d|^2]$;
        \item $\nabla_\eps \nabla_\eps w(\nabla_\varepsilon w,\nabla_\varepsilon w) = \lambda^3\left[ g'^3\nabla_\eps \nabla_\eps d(\nabla_\varepsilon d,\nabla_\varepsilon d) + g''g'^2|\nabla_\varepsilon d|^4 \right]$,
    \end{itemize}
    that yield
    \begin{align}
    \begin{split}
    \label{eq:formulaMw}
        M_\delta w &= \lambda [g'\Delta_\varepsilon d+g''|\nabla_\varepsilon d|^2] - \lambda^3\frac{g'^3\nabla_\eps \nabla_\eps d(\nabla_\varepsilon d,\nabla_\varepsilon d) + g''g'^2|\nabla_\varepsilon d|^4}{\lambda^2g'^2|\nabla_\varepsilon d|^2+\delta^2} \\
        &= \lambda g'\left[\Delta_\varepsilon d - \lambda^2 g'^2\frac{\nabla_\eps \nabla_\eps d(\nabla_\varepsilon d,\nabla_\varepsilon d)}{\lambda^2g'^2|\nabla_\varepsilon d|^2+\delta^2}\right] + \frac{\lambda\delta^2 g''|\nabla_\varepsilon d|^2}{\lambda^2g'^2|\nabla_\varepsilon d|^2+\delta^2}.
        \end{split}
    \end{align}
   Let us analyze first the term inside the brackets, defining it as $F_\delta$:
    \[
        F_\delta(d_E) = \Delta_\varepsilon d_E - \frac{\lambda^2 g'^2}{\lambda^2g'^2|\nabla_\varepsilon d_E|_\eps^2+\delta^2} \nabla_\eps \nabla_\eps d_E(\nabla_\varepsilon d_E,\nabla_\varepsilon d_E).
    \]
    Observe that $F_0 (d_E)(p) = -|\nabla_\eps d_E|(p)H_\eps(p)$, where $H_\eps(p)$ is the $g_\eps$-mean curvature of $\{d_E = d_E(p)\}$ computed at $p$.
     Let $U_{s} = \{p \in \overline{D} \, | d_E(p) < s\}$ be such that $d_E$ is $C^2$ in $\overline{U}_{s}$ (see e.g. \cite[Lemma 14.16]{gilbargtrudinger2001}). From $\partial D$ being noncharacteristic, up to taking a smaller $s$, we can assume that $|H_\eps| \leq C$ and $C^{-1}\leq|\nabla_\eps d_E|_\eps \leq  C$ for $C > 1$ that does not depend on $\eps \leq  1$.

     Choosing $g$ in the definition of $w$ such that 
     \begin{equation}
     \label{eq:request1g'}
   g'(s) > a > 0 \quad  \text{for}\quad   0 \leq s < s_0, 
   \end{equation}
   we observe that
\[
|F_\delta(d_E) - F_0(d_E)| = \frac{|\nabla_\eps\nabla_\eps d_E(\nabla_\eps d_E,\nabla_\eps d_E)|}{|\nabla_\eps d_E|_\eps^2}\, \frac{\delta^2}{\lambda^2 g'(d_E)^2 |\nabla_\eps d_E|^2 + \delta^2} \leq C \delta^2
\]
on ${U}_{s_0}$ for some $C > 0$ that does not depend on $\eps$ nor on $\lambda \geq 1$. 

This bound on $C$ is genuinely uniform in $\eps\le1$: in the $g_\eps$-orthonormal frame $\{X,Y,\eps Z\}$, the only entries of $\nabla_\eps\nabla_\eps d_E$ carrying a factor $1/\eps$ are
\[
\nabla_\eps\nabla_\eps d_E(X,\eps Z)=\eps X(Zd_E)+\frac{1}{2\eps}Yd_E, \qquad \nabla_\eps\nabla_\eps d_E(Y,\eps Z)=\eps Y(Zd_E)-\frac{1}{2\eps}Xd_E.
\]
Contracted twice with $\nabla_\eps d_E$, these contribute
\[
\frac1\eps\big[(Xd_E)(Zd_E)(Yd_E)-(Yd_E)(Zd_E)(Xd_E)\big]=0,
\]
which cancels exactly. The remaining terms involve only Euclidean second derivatives of $d_E$, bounded on $\overline{U}_{s_0}$, times factors $\eps\le1$; hence $\nabla_\eps\nabla_\eps d_E(\nabla_\eps d_E,\nabla_\eps d_E)$, and so $C$ above, is bounded uniformly in $\eps\le1$.

Using again the request \eqref{eq:request1g'}, we deduce that the first summand in the rightmost hand side of \eqref{eq:formulaMw} can be bounded by
\begin{equation}
\label{eq:primoboundmwausiliare}
\lambda g'(d_E)\left[\Delta_\varepsilon d_E - \lambda^2 (g'(d_E))^2\frac{\nabla_\eps \nabla_\eps d_E(\nabla_\varepsilon d_E,\nabla_\varepsilon d_E)}{\lambda^2(g'(d_E))^2|\nabla_\varepsilon d_E|^2+\delta^2}\right] \leq \lambda g'(d_E)(-H_\eps |\nabla_\eps d_E| + C\delta^2). 
\end{equation}
Observe that, by the strict mean convexity of $\partial D$, we can assume up to taking a smaller $s_0$ that $H_\eps (p) > 2c >  0$ for any $p$ in $\overline{U}_{s_0}$ .
Let then $\delta_0 > 0$ be small enough such that the term in round brackets is smaller than $- c < 0$ for any $0 <\delta \leq \delta_0 $ and any $p \in U_{s_0}$. Then, for any such $\delta$, \eqref{eq:primoboundmwausiliare} yields
\begin{equation}
\label{eq:primobound}
\lambda g'(d_E)\left[\Delta_\varepsilon d_E - \lambda^2 (g'(d_E))^2\frac{\nabla_\eps \nabla_\eps d_E(\nabla_\varepsilon d_E,\nabla_\varepsilon d_E)}{\lambda^2(g'(d_E))^2|\nabla_\varepsilon d_E|^2+\delta^2}\right]\leq - C\lambda
\end{equation}
for any $\lambda \geq 1$, and some new constant $C >0$ that does not depend on $\eps$ nor on $\delta$. \\

    The second term in \eqref{eq:formulaMw} can be bounded assuming $g'' > 0$ by
    \begin{equation}
    \label{eq:secondobound}
        \frac{\lambda\delta^2 g''|\nabla_\varepsilon d|^2}{\lambda^2g'^2|\nabla_\varepsilon d|^2+\delta^2} \le \frac{\lambda\delta^2 g''|\nabla_\varepsilon d|^2}{\lambda^2g'^2|\nabla_\varepsilon d|^2} = \frac{\delta^2 g''}{\lambda g'^2}.
    \end{equation}
    So,  choosing $g$ as a solution to the ODE 
    \begin{equation}
    \label{eq:request2}
    g'' = (g')^2,
    \end{equation}
    we obtain, by plugging \eqref{eq:primobound} and \eqref{eq:secondobound} into \eqref{eq:formulaMw} 
    \[
         M_\delta w < C\left(\frac{1}{\lambda}-\lambda\right)
    \]
    for any $0 < \delta \leq \delta_0 < 1$, $\lambda \geq 1$ with a constant $C$ that is independent of $\delta$ and $\eps$. Taking now $\lambda$ big enough guarantees that
   \begin{equation}
   \label{eq:subsolM}
        M_\delta w < -1 = M_\delta u_\varepsilon^\delta \quad \text{in } \{0 < d < s_0\}.
    \end{equation}
    Finally, we request 
    \begin{equation}
    \label{eq:request3}
g= 0 \quad \text{on} \quad \!\!\partial D, \quad   \lim_{s \to s_0^-}g(s) \to +\infty.
    \end{equation}
  Thus, by \eqref{eq:subsolM} and \eqref{eq:request3} we can apply the comparison principle for $M_\delta$ (see e.g. \cite[Proposition 4.3]{FernandezRealRosOton22}) in $\overline{U}_s$ for any $s$ close to $s_0$, and deduce that 
  \[
  u_\eps^\delta \leq w \quad \text{in} \quad \overline{U}_s.
  \]
  An explicit function $g$ fulfilling the requests made in \eqref{eq:request1g'}, \eqref{eq:request2}, and \eqref{eq:request3} is $g: [0, s_0) \to \mathbb{R}$ defined by
    \[
        g(s) = \log\left(\frac{s_0}{s_0-s}\right).
    \]

       To evaluate the normal derivative at the boundary, consider an inward Euclidean normal curve $\gamma(t)$ starting at $p \in \partial D$ such that $d_E(\gamma(t)) = t$. Since $u_\varepsilon^\delta(p) = w(p) = 0$, we have
    \[
        \frac{u_\varepsilon^\delta(\gamma(t)) - u_\varepsilon^\delta(\gamma(0))}{t} \leq \frac{w(\gamma(t)) - w(\gamma(0))}{t} = \lambda \frac{g(t) - g(0)}{t}.
    \]
    Taking the limit as $t \to 0^+$ gives $\partial_{\nu_E} u_\varepsilon^\delta(p)\le \lambda g'(0) = \lambda/s_0$; since moreover $u_\varepsilon^\delta>0$ in $D$ (Theorem~\ref{thm:mcfriem}) and $u_\varepsilon^\delta(\gamma(0))=0$, the same quotient is also $\ge0$ as $t\to0^+$, so
    \[
        \left|\partial_{\nu_E} u_\varepsilon^\delta(p) \right|\leq \frac{\lambda}{s_0}.
    \]
    Because $u_\varepsilon^\delta$ is constant on $\partial D$, its tangential derivatives vanish, so $\tilde\nabla_\varepsilon u_\varepsilon^\delta(p) = \big(\partial_{\nu_E} u_\varepsilon^\delta(p)\big)\tilde\nabla_\varepsilon d_E(p)$. Since $\partial D$ is compact and $\tilde X,\tilde Y,Z$ are smooth, $|\tilde\nabla_\varepsilon d_E|_{\tilde g_\eps}\le C_0$ on $\partial D$ for some $C_0>0$ independent of $\eps\le1$. Consequently,
    \[
        \|\tilde\nabla_\varepsilon u_\varepsilon^\delta\|_{L^\infty(\partial D)} \leq \frac{\lambda}{s_0}\,C_0=:C,
    \]
   completing the proof.
\end{proof}

Proposition \ref{thm:gradest_ednoncar}, (ii) in Theorem \ref{thm:mcfriem} and Proposition \ref{prop:deps-d} immediately imply the following.
\begin{cor}[Right-invariant Lipschitz bound for the viscosity solution]\label{cor:lip_viscosity_nonchar}
    Let $u_\varepsilon$ be the viscosity solution to \eqref{RepsMCFprobl} in the non-characteristic smooth domain $D$. Then $u_\varepsilon$ satisfies the global right-invariant Lipschitz estimate:
    \begin{equation}
    \label{eq:rightinvarianteps}
        |u_\varepsilon(p) - u_\varepsilon(q)| \le C \, \tilde d(p,q) \quad \forall p,q \in \overline  D,
    \end{equation}
    where $C$ is the same uniform constant from Proposition~\ref{thm:gradest_ednoncar}, independent of $\varepsilon$.
\end{cor}

\subsection{Letting $\eps \to 0$: the sub-Riemannian mean-convex mean curvature flow}

Having established the uniform gradient estimates in both the axisymmetric and the non-characteristic scenarios, we are now in the position to prove the global convergence of the Riemannian solutions to suitable sub-Riemannian limit solutions. 

The following result holds in general for converging solutions to the Riemannian Heisenberg mean-convex mean curvature flow, and substantially follows \cite{evansspruck91} and \cite{capognacitti09}.

\begin{prop}
\label{prop:limitsueps}
    Let $D \subset \R^3$ be an open bounded set with $C^2$-boundary. Let, for $\eps > 0$, $u_\eps$ be a viscosity solution to 
    \[
-\Delta_\eps u_\eps + \frac{\grad_\eps \grad_\eps u_\eps(\grad_\eps u_\eps, \grad_\eps u_\eps)}{|\grad_\eps u_\eps|^2} = 1
    \]
    in $D$, and assume $u_\eps \to u$ as $\eps \to 0^+$ locally to $u \in C^0(D)$, possibly along a subsequence. Then, $u$ is a viscosity solution to \eqref{eq:HMCFeq} in the viscosity sense given by Definition \ref{dfn:viscsol}.
\end{prop}

\begin{proof}
We assume for notational simplicity that the whole sequence $u_\eps \to u$. We show that $u$ is a subsolution. Verifying that is also a supersolution will be completely analogous.
    Consider then a test function $\Phi \in C^2(D)$ such that $u-\Phi$ has a local  maximum point at $p\in D$. Assume, for the time being, that such maximum point is strict. Since $u_{\eps}\to u$ locally uniformly, there exists a sequence of points $\{p_\eps\} \subset D$ such that $u_{\eps}-\Phi$ has a local maximum at $p_\eps$, and $p_\eps \to p$ as $\eps\to 0^+$ (see \cite[Lemma 6.1]{CrandallIshiiLions92}).
Assume first that  $\nabla \Phi(p) \neq 0$. Then, for $\eps$ small enough $\nabla_{\eps}\Phi(p_\eps)\neq 0$ and so the $u_\eps$ being a viscosity solution yields
    \begin{equation}
    \label{eq:viscosubsol-eps}
        1 \ge -\Delta_{\eps} \Phi + \frac{\nabla\nabla_{\eps}\Phi(\nabla_{\eps} \Phi,\nabla_{\eps} \Phi)}{|\nabla_{\eps} \Phi|^2} \quad \text{at } p_\eps,
    \end{equation}
   that passed to the limit as $\eps \to 0^+$, gives
    \[
        1 \ge -\Delta \Phi(p) + \frac{\nabla\nabla\Phi(\nabla \Phi,\nabla \Phi)}{|\nabla \Phi|^2}(p). 
    \]
   \\
    On the other hand, assume $\grad\Phi(p)=0$. We have to show that there exists a horizontal vector $\eta$ such that \eqref{eq:subsol_eta} is satisfied. Up taking a further subsequence, we can suppose that either $\nabla_{\eps}\Phi(p_\eps) = 0$ for each $\eps > 0$ or $\nabla_{\eps}\Phi(p_\eps) \neq 0$ for each $\eps > 0$.
Moreover, one can compute
  \begin{equation}
  \label{eq:hessianeps}
   \nabla_\eps\nabla_\eps \Phi = 
   \begin{pmatrix}
        X^2\Phi & XY\Phi - \frac{1}{2}Z\Phi & \eps XZ\Phi + \frac{1}{2\eps}Y\Phi \\
        YX\Phi + \frac{1}{2}Z\Phi & Y^2\Phi & \eps YZ\Phi - \frac{1}{2\eps}X\Phi \\
        \eps ZX\Phi + \frac{1}{2\eps}Y\Phi & \eps ZY\Phi - \frac{1}{2\eps}X\Phi & \eps^2 Z^2\Phi
    \end{pmatrix}
    \end{equation}
    with respect to the orthonormal frame $X, Y, \eps Z$. 
    
\emph{Case 1.} $\nabla_{\eps}\Phi(p_\eps) = 0$ for each $\eps > 0$. Then the definition of viscosity solution for $u_{\eps}$  guarantees the existence of some $\eta_\eps$ with $|\eta_\eps| \le 1$ such that
\[
 1 \ge -\Delta_{\eps} \Phi + {\nabla\nabla_{\eps}\Phi(\eta_\eps,\eta_{\eps})}\quad \text{at } p_\eps,
\] 
    Observe that $|v|_1 \leq |v|_\eps$ for any $v \in T_p\R^3$ and $p \in \R^3$. Thus,
    by compactness, up to a subsequence, $\eta_\eps \to \tilde\eta$ for some vector $\tilde\eta$ with $|\tilde\eta|_1 \le 1$. Let $\eta$ be the $g_1$-orthogonal projection of $\tilde\eta$ onto the horizontal distribution $\mathcal{D}$. Clearly, $\eta \in \mathcal{D}$ and its norm satisfies $|\eta| \le |\eta|_1 \le 1$. It thus suffices to plug $X\Phi(p_\eps) = Y\Phi(p_\eps) = Z\Phi(p_\eps) = 0$ to easily check that
    \begin{equation}
    \label{eq:limhess}
        \lim_{\eps\to0} \nabla_\eps\nabla_{\eps}\Phi(p_\eps)(\eta_\eps,\eta_\eps) = \nabla\grad\Phi(p)(\eta, \eta),
    \end{equation}
    and this implies that \eqref{eq:subsol_eta} is satisfied for such $\eta$.
    
    \emph{Case 2.} $\nabla_{\eps}\Phi(p_\eps) \neq 0$ for each $\eps > 0$. Then \eqref{eq:viscosubsol-eps} holds again. Let $\eta_\eps : = \nabla_\eps\Phi/|\nabla_\eps \Phi|$ computed at $p_\eps$. As before, let $\eta$ be the $g_1$-orthogonal projection of the limit $\tilde\eta$ of some subsequence of the $\eta_\eps$'s. Letting for notational simplicity $T^\eps = \nabla_\eps\nabla_\eps  \Phi$ computed at $p_\eps$, and accordingly $T = \nabla\nabla  \Phi$, we claim that $T^\eps (\eta_\eps, \eta_\eps)$ converges  to $T (\eta, \eta)$ for any $i, j = 1, 2$. On the other hand, the following identity holds, by direct computation.
    \[
    T^\eps(\eta_\eps, \eta_\eps) = 
    \begin{pmatrix}
       X^2\Phi & XY\Phi  & \eps XZ\Phi  \\
        YX\Phi  & Y^2\Phi & \eps YZ\Phi \\
        \eps ZX\Phi & \eps ZY\Phi  & \eps^2 Z^2\Phi
    \end{pmatrix}(\eta_\eps, \eta_\eps).
    \]
    This implies that \eqref{eq:limhess} again holds for this choice of $\eta_\eps$, and in turn that \eqref{eq:subsol_eta} is satisfied. \\

    This concludes the proof in the case where $p$ is a strict maximum point. In the general case where $p$ is a not necessarily strict maximum point of $u-\Phi$, we define the perturbed test function:
    $
        \tilde{\Phi} = \Phi + \psi,
    $
    where $\psi(q)$ is given by $\psi(q) =|p^{-1}q|^4 $. 
    Hence, $p$ is a strict maximum point for $u - \tilde{\Phi}$. 
    Moreover, since $\psi$ has a minimum at $p$, we have
    \begin{equation}
    \label{eq:condizionipsi}
        \nabla \psi(p) = 0 \quad \text{and} \quad \nabla\nabla \psi(p) = 0,
    \end{equation}
    Moreover, it also holds that 
    \begin{equation}
    \label{eq:laplpsi}
    \Delta \psi(p) = 0.
    \end{equation}
    Using that \eqref{eq:subsol} or \eqref{eq:subsol_eta} is satisfied for $\tilde{\Phi}$, and plugging in \eqref{eq:condizionipsi} and \eqref{eq:laplpsi} shows \eqref{eq:subsol} or \eqref{eq:subsol_eta} is respectively satisfied for $\Phi$. This concludes the proof of $u$ being  a viscosity subsolution.

    The function $u$ being a viscosity supersolution is shown in full analogy.
    \end{proof}

We are finally in a position to prove our main Theorem \ref{thm:existence}.

\begin{proof}[Proof of Theorem \ref{thm:existence}]
 Let $\Omega_\eps$ be given by Theorem \ref{thm:approx}, with $\Omega_\eps$ axisymmetric if assumption $(A)$ is in force and \(\Omega_\varepsilon=\Omega\) if assumption $(B)$ holds. 
 
 We further regularize $\Omega_\eps$ through a family of smooth $\{\Omega_\eps^i\}_i$ such that $\Omega_\eps^i \to \Omega_\eps$ in $C^2$ as $i \to +\infty$. Moreover, we let $\Omega_\eps^i$ be axisymmetric in the $(A)$ assumptions: this is straightforward, for example by standard mollification of the distance function from $\partial \Omega_\eps$. We have in particular $H_\eps^i  \geq c > 0$ for $c$ that does not depend on $\eps, i$. In the assumption $(B)$ this follows from the uniform convergence in~\eqref{eq:HepstoH}. Let $K \subset \Omega$ be compact. Then, for any $\eps > 0$ small enough and $i$ big enough, $K \subset \Omega_\eps^i$, as a consequence of the Hausdorff convergence $\partial \Omega_\eps \to \partial \Omega$. Let $u_\eps^i$ be the Riemannian mean curvature flow of $\Omega_\eps^i$ given by Theorem \ref{thm:mcfriem}. The estimates given by Propositions \ref{stimagradeps} and  Corollary \ref{cor:lip_viscosity_nonchar} do not depend on $i, \eps$, and they thus provide as $i \to +\infty$ a family of functions $u_\eps$ that is equibounded and equicontinuous with respect to $\eps$, on $K$. By the arbitrariness of $K$, $u_\eps$ converges up to subsequences to $u$ in $C^0_{\mathrm{loc}} (\Omega)$. By Proposition \ref{prop:limitsueps} $u$ is a viscosity solution to \eqref{eq:HMCFeq}. Choosing now $p \in \Omega$ and $q_\eps \in \partial \Omega_\eps$, the  propositions yield
 $
|u_\eps(p)| \leq C d^*(p, q_\eps),
$
where
$d^*$ denotes $d$ or $\tilde{d}$ according to the validity of $(A)$ or $(B)$.
Using $u_\eps \to u$ pointwise in $\Omega$ and once again $\partial \Omega_\eps \to \partial \Omega$ in Hausdorff distance, we obtain  $
|u(p)| \leq C d^*(p, q)
$
for any $p \in \Omega$ and $q \in \partial \Omega$. This implies that $u$ achieves continuously the Dirichlet datum $0$ at $\partial \Omega$, completing the proof.
 \end{proof}

\section{Mean cuvature blow-up at characteristic points}\label{s:blowup}

In this section we will prove results in turn leading to Theorem~\ref{cor:noselfsim}. 
In fact, we are showing that any time a horizontal mean curvature flow forms a foliation that is, roughly speaking, $C^2$, then the horizontal mean curvature of the evolving sets blows up at characteristic points, see Theorem~\ref{thm:hmcfoliation}

Since $u$ is in general only a continuous viscosity solution, $H$ cannot be evaluated pointwise: the blow-up is instead obtained via sup- and inf-convolutions, building $C^2$ test functions that touch $u$ from one side.

This blow-up is the key obstruction to self-similarity: a self-similar solution reproduces rescaled copies of $\partial\Omega$, whose curvature is uniformly bounded, away from its characteristic set, so no such solution can exist.

\subsection{Generalized $C^2$-foliation}
The following generalized notion of $C^2$-foliation was not considered before in literare, as far as we know. 

\begin{defn}[Generalized $C^2$-foliation]
\label{def:genC2foliation}
Let  $F = E_2 \setminus \overline{E_1}\subset\mathbb{R}^3$, where $E_1$ and $E_2$ are open bounded subsets with $C^2$ boundary such that $\overline{E_1}\subset E_2$. A family $\{\Sigma_\lambda\}_{\lambda\in \Lambda}$ of subsets of $\overline F$ is called a \emph{generalized $C^2$-foliation} of $\overline F$ if there exists a continuous function
\[
    R:\overline F\longrightarrow[0,\infty)
\]
such that:
\begin{enumerate}
\item[(i)] $\Lambda=[\lambda_{\min},\lambda_{\max}]$ for some $0\le\lambda_{\min}<\lambda_{\max}$, and $\partial E_1 = \Sigma_{\lambda_{\min}}, \partial E_2 = \Sigma_{\lambda_{\max}}$;
    \item[(ii)] for every $p\in\overline F$ there exists a unique $\lambda\in\Lambda$ such that $p\in\Sigma_\lambda  = \{p\in\overline F : R(p)=\lambda\}$;
    \item[(iii)] there exists an open, dense set $A\subseteq F$ such that $R\in C^2(A)$, $\nabla_{\mathrm{eucl}}R\neq0$ on $A$, and
    \[
    A\cap\Sigma_\lambda\neq\emptyset \qquad \text{for every } \lambda\in\Lambda \text{ with } \Sigma_\lambda\subset F;
    \]
    \item[(iv)] setting $F_\lambda = \{p\in\overline F : R(p)<\lambda\}$, one has
    \[
    \Sigma_\lambda = \partial (E_1 \cup F_\lambda)
    \]
\end{enumerate}
\end{defn}
Observe that a classical $C^2$-foliation would consist in $A = F$.

We find convenient to isolate the following geometrically meaningful quantities.

{
 \begin{defn}[Dilation generator and transversal set]\label{def:transversalset}
Let $D(x,y,z)=(x,y,2z)$ denote the infinitesimal generator of the anisotropic dilations $\delta_\lambda$, i.e.\ $D(p)=\frac{d}{d\lambda}\big|_{\lambda=1}\delta_\lambda(p)$. For a bounded open set $\Omega\subset\R^3$ with $C^1$ boundary, and $\nu_{\mathrm{eucl}}$ the continuous Euclidean outward unit normal to $\partial\Omega$, we define
\[
    G = \big\{\, p\in\partial\Omega \mid \langle D(p),\nu_{\mathrm{eucl}}(p)\rangle \neq 0 \,\big\}.
\]
\end{defn}}

The motivation behind the introduction of the above weakened version lies in the following observation.


\begin{lem}
\label{lem:dilation-foliation}
Let $\Omega\subset\R^3$ be a bounded open set with $C^2$ boundary $\partial\Omega$. For $\lambda\in(0,1]$ set
\[
    \Sigma_\lambda = \delta_\lambda(\partial\Omega), \qquad F_\lambda = \delta_\lambda(\Omega),
\]
and assume
\[
    \overline{F_{\lambda_1}} \subset F_{\lambda_2} \qquad \text{for every } 0<\lambda_1<\lambda_2\le1.
\]
Then, for any $0<\bar \lambda<1$, $\{\Sigma_\lambda\}_{\lambda\in[\bar \lambda,1]}$ is a generalized $C^2$-foliation of $\overline\Omega\setminus F_{\bar \lambda}$, with 
 \begin{equation}
 \label{eq:defA-dilated}
        A = \bigcup_{\lambda \in (\overline{\lambda},1)} \delta_\lambda(G).
    \end{equation}
\end{lem}

\begin{proof}
The assumed nesting, condition (iv) is satisfied.

    \textit{Step 1}
    Here we show properties (i) and (ii) of Definition \ref{def:genC2foliation},
    Observe first that
    \[
        \bigcup_{\lambda\in[\bar \lambda,1]}\Sigma_\lambda = \overline\Omega\setminus F_{\bar \lambda}.
    \]
By the strict nesting $\overline{F_{\lambda_1}}\subset F_{\lambda_2}$ for $\lambda_1<\lambda_2$, the boundaries $\Sigma_\lambda=\partial F_\lambda$ are pairwise disjoint. Hence every $p\in\overline\Omega\setminus F_{\bar \lambda}$ lies on exactly one leaf $\Sigma_\lambda$, and we may unambiguously define
    \[
        R:\overline\Omega\setminus F_{\overline{\lambda}} \to[0,1], \qquad R(p)=\lambda \iff p\in\Sigma_\lambda.
    \]
    For $0\le\alpha<\beta\le1$, we have
    \[
        R^{-1}\big((\alpha,\beta)\big) = F_\beta\setminus\overline{F_\alpha},
    \]
    which is an open set (the difference between the open set $F_\beta=\delta_\beta(\Omega)$ and the closed set $\overline{F_\alpha}$). Therefore, $R$ is continuous. In particular, $\Sigma_1=\partial\Omega=\{R=1\}$, and $F_\lambda=\{R<\lambda\}$ by construction; the same identification applied at $\lambda=\bar\lambda$ gives $\Sigma_{\bar\lambda}=\partial F_{\bar\lambda}=\{R=\bar\lambda\}$.
    
    Setting $F=\Omega\setminus\overline{F_{\bar\lambda}}$  we have $\overline{F_{\bar\lambda}}=F_{\bar\lambda}\cup\partial F_{\bar\lambda}=\{R<\bar\lambda\}\cup\{R=\bar\lambda\}=\{R\le\bar\lambda\}$, so
    \[
        F=\{R<1\}\setminus\{R\le\bar\lambda\}=\{\bar\lambda<R<1\}, \qquad \overline F=\{\bar\lambda\le R\le1\}=\overline\Omega\setminus F_{\bar\lambda},
    \]
    and consequently
    \[
        \partial F=\overline F\setminus F=\{R=\bar\lambda\}\cup\{R=1\}=\Sigma_{\bar\lambda}\cup\Sigma_1.
    \]
    This is exactly property (i) of Definition~\ref{def:genC2foliation} (with $\lambda_{\min}=\bar\lambda$, $\lambda_{\max}=1$); property (ii) is precisely the unique-leaf assignment constructed above.

    \textit{Step 2 (the transversal set is dense on $\partial\Omega$).}
    Let $\nu_{\mathrm{eucl}}$ be the continuous Euclidean unit normal to the $C^2$-boundary $\partial\Omega$. The relatively open set
    \[
        G = \big\{\, p\in\partial\Omega \mid \langle D(p),\nu_{\mathrm{eucl}}(p)\rangle \neq 0 \,\big\}
    \]
    is dense in $\partial\Omega=\Sigma_1$; in particular, $G\neq\emptyset$. Indeed, if $G$ were not dense in $\Sigma_1$, there would exist an open set $B\subset\Sigma_1$ such that $\langle D(p), \nu_{\mathrm{eucl}}(p) \rangle = 0$ for all $p\in B$. This would imply that for $\lambda$ sufficiently close to 1, $\delta_\lambda(p)$ would stay in $\partial \Omega$, contradicting the strict nesting property $\Sigma_\lambda \subset\subset \Omega$.

    \textit{Step 3.} Here, we show that the union of the dilated copies of the set $G$ defined in the previous step form the regular set $A$.
    Let $g: \R^3 \to \mathbb{R}$ be a $C^2$ defining function for $\Omega$, such that $\partial\Omega=\{p : g(p)=0\}$, $g<0$ in $\Omega$, and $\nabla_{\mathrm{eucl}}g \neq 0$ on $\partial\Omega$. 
    Define the map $\mathcal{F}: (\mathbb{H}\setminus\{0\}) \times (0,1] \to \mathbb{R}$ by
    \begin{equation}
        \mathcal{F}(p, \lambda) = g(\delta_{1/\lambda}(p)).
    \end{equation}
    By definition, $\Sigma_\lambda = \delta_\lambda(\partial\Omega) = \{p \in \mathbb{H}\setminus\{0\} \mid \mathcal{F}(p, \lambda) = 0\}$. Since $R(p)$ is the unique parameter such that $p \in \Sigma_{R(p)}$, it identically satisfies
    \begin{equation}\label{eq:implicit_R}
        \mathcal{F}(p, R(p)) = 0 \qquad \text{for all } p \in \overline\Omega\setminus F_{\bar \lambda}.
    \end{equation}
    To prove that $R$ is $C^2$ via the implicit function theorem, we compute $\partial_\lambda \mathcal{F}$. Fix $p \in \Sigma_\lambda$ and set $q = \delta_{1/\lambda}(p) \in \partial\Omega$. Recalling that $q = (x/\lambda, y/\lambda, z/\lambda^2) =: (x_q, y_q, z_q)$, the chain rule yields:
    \begin{align}
        \partial_\lambda \mathcal{F}(p, \lambda) &= \partial_x g(q)\left(-\frac{x}{\lambda^2}\right) + \partial_y g(q)\left(-\frac{y}{\lambda^2}\right) + \partial_z g(q)\left(-\frac{2z}{\lambda^3}\right) \\
        &= -\frac{1}{\lambda} \left( x_q \partial_x g(q) + y_q \partial_y g(q) + 2z_q \partial_z g(q) \right) \\
        &= -\frac{1}{\lambda} \langle D(q), \nabla_{\mathrm{eucl}} g(q) \rangle.
    \end{align}
    Now, assume that $q \in G$. By the definition of $G$, the generator $D(q)$ is not tangent to $\partial\Omega$, meaning $\langle D(q), \nabla_{\mathrm{eucl}} g(q) \rangle \neq 0$. Consequently, $\partial_\lambda \mathcal{F}(p, \lambda) \neq 0$. 
    
    By the implicit function theorem applied to \eqref{eq:implicit_R}, the function $R(p)$ is $C^2$ in a neighborhood of $p$. Furthermore, differentiating the identity $\mathcal{F}(p, R(p)) = 0$ with respect to the coordinates $p$ gives:
    \begin{equation}
        \nabla_{\mathrm{eucl}} R(p) = - \frac{\nabla_{\mathrm{eucl}} \mathcal{F}(p, \lambda)}{\partial_\lambda \mathcal{F}(p, \lambda)}.
    \end{equation}
    By the chain rule, $\nabla_{\mathrm{eucl}} \mathcal{F}(p, \lambda) = \left(\frac{\partial_x g(q)}{\lambda}, \frac{\partial_y g(q)}{\lambda}, \frac{\partial_z g(q)}{\lambda^2}\right)$. Since $\nabla_{\mathrm{eucl}} g(q) \neq 0$ (as $g$ is a defining function), we have $\nabla_{\mathrm{eucl}} \mathcal{F}(p, \lambda) \neq 0$, which implies $\nabla_{\mathrm{eucl}} R(p) \neq 0$.

Let now $A$ be defined by \eqref{eq:defA-dilated}.
   
    By the results of Step 3, $R \in C^2(A)$ and $\nabla_{\mathrm{eucl}} R \neq 0$ on $A$.
    Since $G$ is open in $\partial\Omega$ and $(\overline\lambda,1)$ is an open interval, $G\times(\overline\lambda,1)$ is open in $\partial\Omega\times(0,1]$; as $(q,\lambda)\mapsto\delta_\lambda(q)$ is a diffeomorphism onto its image, $A$ is open, and by construction $A\subseteq\Omega\setminus\overline{F_{\overline\lambda}}=:F$, as required by property (iii). Moreover, $G\times(\overline\lambda,1)$ is dense in $\partial\Omega\times[\overline\lambda,1]$, so $A$ is dense in $\overline\Omega\setminus F_{\overline\lambda}$.
    Finally, by construction, $A\cap\Sigma_\lambda=\delta_\lambda(G)\neq\emptyset$ for every $\lambda\in(\overline\lambda,1)$, i.e.\ for every leaf interior to $F$. This fulfills property (iii). Property (iv)  follows from the very construction of $R$, and concludes the proof.
\end{proof}

\begin{lem}\label{lem:almostfoliationsim}
    Let $\Omega \subset \R^3$ be an open, bounded set and let $u \in C^0(\overline{\Omega})$ be a solution to \eqref{HMCFprobl}. Assume that the level sets of $u$, $\{\Gamma_t = \{u=t\} \}_{t \in [t_1, t_2]}$, form a generalized $C^2$-foliation of the compact region $\Omega_{[t_1,t_2]} = \{p \in \overline{\Omega} \mid u(p) \in [t_1,t_2]\}$, where the foliation is parameterized by $\lambda \in [\lambda_1, \lambda_2]$. 
    
    Then, there exists a continuous function $R \in C^0(\Omega_{[t_1,t_2]}; [\lambda_1, \lambda_2])$, fulfilling the conditions of Definition \ref{def:genC2foliation}, and a strictly monotonic continuous function $f: [\lambda_1, \lambda_2] \to [t_1, t_2]$ such that:
    \begin{equation}
        u = f \circ R \quad \text{in } \Omega_{[t_1,t_2]}.
    \end{equation}
\end{lem}

\begin{proof}
By Definition~\ref{def:genC2foliation}, the assumption that $\{\Gamma_t\}_{t\in[t_1,t_2]}$ is a generalized $C^2$-foliation of $\Omega_{[t_1,t_2]}$ provides a continuous function $R:\Omega_{[t_1,t_2]}\to[\lambda_1,\lambda_2]$ whose level sets $\Sigma_\lambda=\{R=\lambda\}$ coincide, as a family of subsets of $\Omega_{[t_1,t_2]}$, with the family $\{\Gamma_t\}_{t\in[t_1,t_2]}$: that is, for every $\lambda\in[\lambda_1,\lambda_2]$ there exists $t\in[t_1,t_2]$ such that $\Sigma_\lambda=\Gamma_t$.

Since the sets $\Gamma_t$ are pairwise disjoint and nonempty (being level sets of the single-valued function $u$ on $\Omega_{[t_1,t_2]}=u^{-1}([t_1,t_2])$), this $t$ is uniquely determined by $\lambda$: if $\Sigma_\lambda=\Gamma_{t}=\Gamma_{t'}$ with $t\neq t'$, then $\Gamma_t\cap\Gamma_{t'}\neq\emptyset$, contradicting disjointness. We may therefore define, unambiguously,
\[
    f:[\lambda_1,\lambda_2]\to[t_1,t_2], \qquad f(\lambda)=t \iff \Sigma_\lambda=\Gamma_t.
\]

By construction, for every $p\in\Omega_{[t_1,t_2]}$, letting $\lambda=R(p)$ we have $p\in\Sigma_\lambda=\Gamma_{f(\lambda)}$, i.e.\ $u(p)=f(\lambda)=f(R(p))$. Hence
\[
    u = f\circ R \quad \text{in } \Omega_{[t_1,t_2]}.
\]

We now prove that $f$ is a bijection. To prove the injectivity suppose that $f(\lambda_a)=f(\lambda_b)=t$, then $\Sigma_{\lambda_a}=\Gamma_t=\Sigma_{\lambda_b}$; since the sets $\Sigma_\lambda$ are pairwise disjoint and nonempty (each point of $\Omega_{[t_1,t_2]}$ lies on exactly one leaf, by definition of foliation), this forces $\lambda_a=\lambda_b$.

Moreover, for every $t\in[t_1,t_2]$, the level set $\Gamma_t$ is nonempty and, by hypothesis, coincides with some $\Sigma_\lambda$; hence $t=f(\lambda)$ for that $\lambda$, so $f$ is surjective.
We show that $f$ is also continuous. Suppose not: then there exist $\lambda,\lambda_j\in[\lambda_1,\lambda_2]$ with $\lambda_j\to\lambda$ as $j\to+\infty$ but $f(\lambda_j)\not\to f(\lambda)$, i.e.\ there is $\eta>0$ such that $|f(\lambda_j)-f(\lambda)|\ge\eta$ for every $j$. Let $p_j\in\{R=\lambda_j\}$. By compactness of $\Omega_{[t_1,t_2]}$, a further subsequence $p_{j_k}$ converges to some $p\in\Omega_{[t_1,t_2]}$. By continuity of $R$, $R(p)=\lim_k R(p_{j_k})=\lim_k\lambda_{j_k}=\lambda$. By continuity of $u$,
\[
f(\lambda_{j_k})=u(p_{j_k})\to u(p)=f(\lambda),
\]
contradicting $|f(\lambda_{j_k})-f(\lambda)|\ge\eta$ for all $k$. This proves that $f$ is continuous.


 It is a standard fact in real analysis that a continuous injective function on an interval is strictly monotonic. Hence $f$ is strictly monotonic on $[\lambda_1,\lambda_2]$.

This completes the proof: $R\in C^0(\Omega_{[t_1,t_2]};[\lambda_1,\lambda_2])$ and $A$ are as given by Definition~\ref{def:genC2foliation}, $f\in C^0([\lambda_1,\lambda_2];[t_1,t_2])$ is strictly monotonic, and $u=f\circ R$.
\end{proof}

\subsection{A viscosity approach to the curvature blow-up}
Let $u\in C^0(\bar \Omega)$ be a viscosity solution to \eqref{HMCFprobl} and let $\Gamma_t=\{p\in\overline{\Omega} : u(p)=t\}$. Moreover, we let $T^*=\max_{\bar{\Omega}} u$ be the \textit{extinction time}, representing the maximum time for which the mean curvature flow $\Gamma_t$ evolution is well-defined.
The main result of the section is the following.
\begin{thm}\label{thm:hmcfoliation}
    Let $\Omega\subset\R^3$ be an open, bounded set with $C^2$-boundary
    and $u\in C^0(\bar \Omega)$ be a solution to \eqref{HMCFprobl} which is locally Lipschitz with respect to the right-invariant sub-Riemannian distance $\tilde d$. Assume that the level sets $\{\Gamma_t\}_{t \in [t_1, t_2]}$ form a generalized $C^2$-foliation in $[t_1,t_2]\subset(0,T^*)$, with regular set $A$ as in Definition~\ref{def:genC2foliation}(iii).
    Then, for any $p_\mathrm{\mathrm{char}}\in A\cap\bigcup_{t\in[t_1,t_2]}\mathcal{C}(\Gamma_t)$ there exists $\{p_n\}_n\subset\Omega_{[t_1,t_2]}$ such that $p_n\to p_\mathrm{\mathrm{char}}$ and 
    \begin{equation}
        H_n(p_n)\to\infty \quad \text{as}\  n\to \infty,
    \end{equation}
    where $H_n$ is the horizontal mean curvature of the level $\Gamma_{\tau_n}$, with $\tau_n\in[t_1,t_2]$ such that $p_n\in \Gamma_{\tau_n}$.
\end{thm}

The proof of {Theorem \ref{thm:hmcfoliation}} proceeds by constructing a sequence of non-characteristic points approaching $p_\mathrm{\mathrm{char}}$ along which a quantitative lower bound forces the horizontal mean curvature to blow up. This requires pointwise evaluation of the horizontal mean curvature operator; however, the arrival time $u$ is only continuous in general, so in the end, we will perform such evaluations on suitable $C^2$ test functions that touch $u$ from one side.

Let $g: U \to \R$  be a  bounded function on a set $U\subset\R^3$, and let $s>0$. We define the \emph{sup-convolution} $g^s$ and the \emph{inf-convolution} $g_s$ as follows:
\begin{equation}\label{eq:subinfconv}
    g^s(p) = \sup_{q\in U} \biggl\{ g(q) - \frac{1}{s}|p-q|^2 \biggr\}\quad\text{and}\quad
    g_s(p) = \inf_{q\in U} \biggl\{ g(q) + \frac{1}{s}|p-q|^2 \biggr\},
\end{equation}
see e.g.~\cite[Chapter~3.5]{CannarsaSinestrari2004}.
Since the mapping $p \mapsto g^s(p) + \frac{1}{s}|p|^2$ is convex by construction, the regularized function $g^s$ is semi-convex (and analogously, $g_s$ is semi-concave), then by Alexandroff's Theorem $g_s$ and $g^s$ are twice differentiable almost everywhere.\\
Moreover, a standard fact is that assuming the domain to be compact implies that the minimizer in the inf-convolution converges to the evaluated point uniformly as $s\to 0$. In fact, letting $q_s\in\overline U$ realize the minimum in the definition of $g_s(p)$ (which exists, by continuity of $g$ and compactness of the domain), we have the quantitative estimate
\[
    |p - q_s| \;\le\; \sqrt{2Ms}, \qquad M=\sup_{\overline U}|g|<\infty,
\]
uniformly in $p$. Indeed, testing the infimum at $q=p$ gives $g_s(p)\le g(p)\le M$; on the other hand, since $g(q_s)\ge -M$,
\[
    g_s(p) = g(q_s) + \frac1s|p-q_s|^2 \ge -M + \frac1s|p-q_s|^2.
\]
Combining the two inequalities yields $|p-q_s|^2\le 2Ms$, and the claim follows by letting $s\to0$.

The following $1$-dimensional $C^2$-touching lemma ensures that, at any point of twice differentiability of the profile function, one can construct a valid global $C^2$ test function. It follows as a direct consequence of the general consistency result \cite[Corollary 12.14]{Calder2018}, specialized to the one-dimensional case.

\begin{lem}\label{lem:c2touching}
    Let $f: [a, b] \to \R$ be twice differentiable at $r_0 \in (a, b)$. Then, there exists a function $\varphi \in C^2([a, b])$ such that:
    \[
        \varphi(r_0) = f(r_0), \quad \varphi'(r_0) = f'(r_0), \quad \varphi''(r_0) = f''(r_0),
    \]
    and the difference $f - \varphi$ has a strict global minimum at $r_0$.
\end{lem}

\begin{lem}\label{lem:prelimexpchar}
    Let $f:[a,b]\to\R$ be a strictly monotone continuous function, and let $r\in(a,b)$. Assume $f$ is Lipschitz continuous on a neighborhood of $r$, with Lipschitz constant $L>0$. 
    Then there exist a sequence $(r_n)_{n\in\N}\subset(a,b)$, with $r_n\to r$ 
    as $n\to\infty$, and a sequence $(\psi_n)_{n\in\N}\subset C^2([a,b])$ such that, for every $n$, $f-\psi_n$ has a local minimum at $r_n$ and
    \[
        0 < |\psi_n'(r_n)| \le L, \qquad \forall n.
    \]
\end{lem}

In the proof of this Lemma we need a technical result about inf-convolutions of strictly decreasing functions, which is the following.

\begin{lem}\label{lem:strict_monotonicity_envelope}
Let $f: [a,b] \to \mathbb{R}$ be a continuous, strictly monotone function, and let $s > 0$. The inf-convolution $f_s$ cannot be constant on any open interval $I \subset [a,b]$.
\end{lem}

\begin{proof}
Assume by contradiction that there exists an open interval $I  \subset (a,b)$ and a constant $C \in \mathbb{R}$ such that $f_s(l) = C$ for all $l \in I$.

For any $l \in I$, let $m_l \in [a,b]$ be a point achieving the minimum, so that
\begin{equation}\label{eq:costinfconv}
    C = f_s(l) = f(m_l) + \frac{1}{s}(m_l - l)^2.
\end{equation}
For any other point $l' \in I$, since $f_s(l')$ is the global minimum for the parameter $l'$, evaluating the functional at $m_l$ yields the inequality:
\begin{equation}
     f_s(l') \le f(m_l) + \frac{1}{s}(m_l - l')^2. 
\end{equation}
Using the assumption that $f_s(l') = C$, we have:
\begin{equation}
    C \le f(m_l) + \frac{1}{s}(m_l - l')^2.
\end{equation}

Plugging \eqref{eq:costinfconv} into the inequality above, we obtain:
\begin{equation}
    C \le C - \frac{1}{s}(m_l - l)^2 + \frac{1}{s}(m_l - l')^2.
\end{equation}
Simplifying $C$ and multiplying by $s > 0$ yields:
\begin{equation}
    (m_l - l)^2 \le (m_l - l')^2. 
\end{equation}
Expanding the squares, and factorizing we get:
\begin{equation}
    (l - l')(l + l') \le 2m_l(l - l'). 
\end{equation}

This inequality must hold for any choice of $l, l' \in I$. Taking  the limit 
     $l' \to l^-$ we obtain  $m_l \ge l$, while taking the right limit we obtain the reverse inequality, concluding that
     $m_l = l$ for all $l \in I$. Finally, substituting $m_l = l$ back into \eqref{eq:costinfconv}, we find:
$$ C = f(m_l) + \frac{1}{s}(m_l - l)^2 = f(l) + \frac{1}{s}(l - l)^2 = f(l). $$
This implies that $f(l) = C$ for all $l \in I$, which contradicts the assumption that $f$ is strictly monotone on $[a,b]$. Therefore, the initial assumption must be false, and $f_s$ cannot be constant on any open interval.
\end{proof}

Now we give the proof of Lemma \ref{lem:prelimexpchar}.

\begin{proof}[Proof of Lemma~\ref{lem:prelimexpchar}]
    For $s>0$, consider the inf-convolution $f_s:[a,b]\to\mathbb R$ defined in \eqref{eq:subinfconv}, which is Lipschitz continuous and twice differentiable almost everywhere.

    Let $\rho\in(a,b)$ be a point of twice differentiability of $f_s$. Owing to the continuity of $f$ and the compactness of the domain, the infimum defining $f_s(\rho)$ is always attained, so that we let $m\in[a,b]$ satisfying
    \begin{equation}\label{eq:pippi}
        f_s(\rho) = f(m) + \frac1s(\rho-m)^2.
    \end{equation}
    By Lemma~\ref{lem:c2touching}, there exists $\varphi\in C^2([a,b])$ such that $f_s-\varphi$ attains a global minimum at $\rho$. Thus, for $l\in(a,b)$, $c\in[a,b]$,
    \begin{equation}\label{eq:minfunztesttouch}
        f_s(\rho)-\varphi(\rho) \;\le\; f_s(l)-\varphi(l) \;\le\; f(c)+\frac1s(l-c)^2-\varphi(l).
    \end{equation}
    Plugging \eqref{eq:pippi} into the latter and choosing $c\in[a,b]$ such that $l=c+\rho-m\in(a,b)$, the quadratic terms cancel:
    \begin{equation}\label{eq:stimavisclocblowup}
        f(m)-\varphi(\rho) \le f(c)-\varphi(c+\rho-m).
    \end{equation}
    Since $\rho$ is an interior point, define a neighborhood $\mathcal U_m$ of $m$ small enough that $\rho+c-m\in(a,b)$ for every $c\in\mathcal U_m$, and set $\psi(c)=\varphi(c+\rho-m)$ for $c\in\mathcal U_m$; then $\psi\in C^2(\mathcal U_m)$ extends to a function in $C^2([a,b])$. Since $\psi(m)=\varphi(\rho)$, \eqref{eq:stimavisclocblowup} gives
    \[
        f(m)-\psi(m) \le f(c)-\psi(c) \qquad \forall c\in\mathcal U_m,
    \]
    so $m$ is a local minimum of $f-\psi$, with $\psi'(m)=\varphi'(\rho)=f_s'(\rho)$.

     For every fixed $s>0$, the set
    \[
        E_s = \{\rho\in(a,b) : f_s \text{ is twice differentiable at } \rho, \ f_s'(\rho)\neq0\}
    \]
    is dense in $(a,b)$. Indeed, let $I\subset(a,b)$ be an arbitrary open interval. Since $f_s$ is Lipschitz on $[a,b]$, it is absolutely continuous on $\overline I$, so for $\xi_0,\xi\in I$, $f_s(\xi)-f_s(\xi_0)=\int_{\xi_0}^\xi f_s'(t)\,dt$. If $f_s'=0$ a.e.\ on $I$, then $f_s$ is constant on $I$, contradicting Lemma~\ref{lem:strict_monotonicity_envelope} (using strict monotonicity of $f$). Hence $\{t\in I: f_s'(t)\neq0\}$ has positive measure; intersecting with the full-measure set of points of twice differentiability (Alexandroff's theorem) shows $E_s\cap I\neq\emptyset$. As $I$ was arbitrary, $E_s$ is dense in $(a,b)$.

    Set $\sigma_n=1/n$. By density of $E_{\sigma_n}$, we may choose
    \[
        \rho_n \in E_{\sigma_n} \cap \bigl(r-\tfrac1n,\, r+\tfrac1n\bigr),
    \]
    so that $\rho_n\to r$.
    Then, the construction above provides $r_n \in [a, b]$ and $C^2$ functions $\psi_n$ such that $f-\psi_n$ has a local minimum at $r_n$, with
    \[
        \psi_n'(r_n) = \varphi_n'(\rho_n) = f_{\sigma_n}'(\rho_n) \neq 0.
    \] 
    Recall that such $r_n$ realizes the minimum in the definition of $f_{\sigma_n}(\rho_n)$, and consequently
    $r_n\to r$ (see the discussion preceding Lemma \ref{lem:c2touching}).
     Since $f$ is $L$-Lipschitz on a neighborhood $\mathcal V$ of $r$, and $\rho_n,r_n\to r$, for $n$ large enough $\rho_n,r_n\in\mathcal V$. Fix such $n$, and let $l,l'\in\mathcal V$; taking $c=c^*+(l-l')$ as a competitor in the definition of $f_{\sigma_n}(l)$, where $c^*$ realizes the minimum for $f_{\sigma_n}(l')$, the quadratic terms cancel as in \eqref{eq:stimavisclocblowup}, yielding
    \[
        f_{\sigma_n}(l)-f_{\sigma_n}(l') \le f(c^*+(l-l'))-f(c^*) \le L|l-l'|.
    \]
   Then, $f_{\sigma_n}$ is $L$-Lipschitz near $r$, with constant independent of $n$. Hence $|f_{\sigma_n}'(\rho_n)|\le L$ for $n$ large enough, so
    \[
        |\psi_n'(r_n)| = |f_{\sigma_n}'(\rho_n)| \le L, \qquad \forall n.
    \]
\end{proof}

We now give the proof of Theorem~\ref{thm:hmcfoliation}.
The first step is to show that $f$ is Lipschitz: moving along a curve transverse to the leaves and using the local right-Lipschitz continuity of $u$, this regularity transfers to $f$. A one-dimensional touching argument then produces test functions $\psi_n$ touching $f$ at levels arbitrarily close to the characteristic one $p_{\mathrm{\mathrm{char}}}$, with uniformly controlled slope; testing $\psi_n\circ R$ on a suitable sequence of nearby noncharacteristic points $p_n \to p_{\mathrm{\mathrm{char}}}$, shows that along $p_n$, the horizontal mean curvature of the leaf $p_n$ belongs to blows up while approaching $p_{\mathrm{\mathrm{char}}}$.

\begin{proof}[Proof of Theorem \ref{thm:hmcfoliation}]
    Let $u=f\circ R$ where $f\in C^0([R_1,R_2];[t_1,t_2])$,\\ $R\in C^0(\Omega_{[t_1,t_2]};[R_1,R_2])$ and $A$ the open dense set as in Lemma~\ref{lem:almostfoliationsim}.
    For $r\in[R_1,R_2]$, we let 
    \begin{equation}
        \Sigma_r=\Gamma_{f(r)}.
    \end{equation}

    Suppose $m \in (R_1,R_2)$ and $\psi \in C^2([R_1,R_2])$ are such that $f-\psi$ has a local minimum at $m$, with $\psi'(m)\neq 0$. Let $\bar p \in \Sigma_m\cap A$ (i.e.\ $R(\bar p)=m$) be such that $\nabla R(\bar p) \neq 0$. Set $\Psi = \psi\circ R \in C^2$ in an neighborhood of $\overline{p}\in A$.
    
    Since $f-\psi$ has a local minimum at $m$, there is a neighborhood $\mathcal U_m \subset (R_1,R_2)$ of $m$ such that
    \[
        f(m) - \psi(m) \;\le\; f(c) - \psi(c) \qquad \forall\, c \in \mathcal U_m.
    \]
    By continuity of $R$, there is a neighborhood $\mathcal V_{\bar p}$ of $\bar p$ such that $R(p) \in \mathcal U_m$ for all $p \in \mathcal V_{\bar p}$. Evaluating the previous inequality at $c = R(p)$ and using $u = f\circ R$, $\Psi = \psi\circ R$, and $R(\bar p)=m$:
    \[
        u(\bar p) - \Psi(\bar p) = f(m) - \psi(m) \;\le\; f(R(p)) - \psi(R(p)) = u(p) - \Psi(p) \qquad \forall\, p \in \mathcal V_{\bar p},
    \]
    i.e.\ $u - \Psi$ has a local minimum at $\bar p$.
    
    By the chain rule, $\nabla \Psi = \psi'(m)\,\nabla R$; since $\psi'(m)\neq 0$ and $\nabla R(\bar p)\neq 0$, we get $\nabla\Psi(\bar p) \neq 0$. Since $u$ is a viscosity supersolution of \eqref{HMCFprobl} and $\Psi\in C^2$ touches $u$ from below at $\bar p$ with $\nabla\Psi(\bar p)\neq0$, the definition of viscosity supersolution and the identity~\eqref{eq:H} give
    \begin{equation}\label{eq:supersp1}
    \begin{split}
    1 &\le -|\nabla \Psi(\bar{p})| \operatorname{div} \left( \frac{\nabla \Psi}{|\nabla \Psi|} \right)(\bar{p}) \\
    &= -|\psi'(m)| |\nabla R(\bar{p})| \operatorname{div} \left( \frac{\psi'(m)}{|\psi'(m)|} \frac{\nabla R}{|\nabla R|} \right)(\bar{p}) \\
    &= -\psi'(m) |\nabla R(\bar{p})| \operatorname{div} \left( \frac{\nabla R}{|\nabla R|} \right)(\bar{p}).
    \end{split}
    \end{equation}

    Fix $p_{\mathrm{char}}\in\Omega_{[t_1,t_2]}$ with $R(p_{\mathrm{char}})=r_0\in (R_1,R_2)$ such that $p_{\mathrm{char}}\in \mathcal C(\Sigma_{r_0})$; by hypothesis, $p_{\mathrm{char}}\in A$.

    \noindent\textit{Step 1 ($f$ is locally Lipschitz near $r_0$)} We show that $f$ is locally Lipschitz near $r_0$. Since the right-invariant distribution $\tilde{\mathcal D}$ satisfies the same Hörmander (bracket-generating) condition as its left-invariant counterpart, the Frobenius Theorem guarantees that the right-characteristic locus of any $C^2$-hypersurface has relatively empty interior. We may therefore select a strictly regular point $q^*$ on $\Sigma_{r_0}\cap A$, where $\tilde\nabla R(q^*)\neq0$, since $\Sigma_{r_0}\cap A$ is not empty by assumption and, since $A$ is open, $\Sigma_{r_0}$ is a $C^2$ hypersurface by the implicit function theorem. In a compact neighborhood $\mathcal U\subset A$ of $q^*$, we define the $C^1$-smooth vector field
    \[
        \tilde W = \frac{\tilde\nabla R}{|\tilde\nabla R|^2}.
    \]
    In particular, there exists $K>0$ such that $|\tilde W|=1/|\tilde\nabla R|\le 1/K$ on $\mathcal U$. 
    
    For $r_1,r_2$ near $r_0$, flowing along $\tilde W$ connects a point $q_1\in\Sigma_{r_1}\cap\mathcal U$ to $q_2\in\Sigma_{r_2}\cap\mathcal U$ along a curve $\gamma$ which is tangent to $\tilde{\mathcal D}$ at every point and the support of $\gamma$ is contained in $A$. Since $\tilde\nabla R\cdot\tilde W \equiv 1$, the flow time required is
    \begin{equation}
        T=\int_0^T\tilde\nabla R(\gamma(t))\cdot\tilde W(\gamma(t))\;dt=\int_0^T(R\circ\gamma)'(t)\;dt =|r_1-r_2|.
    \end{equation}
    Hence, we have 
    \begin{equation}
        \int_0^T |\dot\gamma(t)|\,dt = \int_0^T |\tilde W(\gamma(t))|\,dt \le \Big(\sup_{\mathcal U}|\tilde W|\Big)\,T \le  \frac{|r_1-r_2|}{K}.
    \end{equation}
    By~\eqref{eq:CCdist} for the right invariant structure, this implies: $\tilde d(q_1,q_2)\le C|r_1-r_2|$, with $C=1/K$.
    Using the local Lipschitz continuity of $u$ with respect to $\tilde d$, and calling $L$ the Lipschitz constant on $\mathcal U$ we therefore obtain that 
    \begin{equation}\label{eq:lipstimatempo}
        |f(r_1)-f(r_2)| = |u(q_1)-u(q_2)| \le L\,\tilde d(q_1,q_2) \le L C\,|r_1-r_2|,
    \end{equation}
    i.e., $f$ is Lipschitz near $r_0$ with constant $M=L C$.

    \noindent\textit{Step 2 (touching sequence).} By Step~1, $f$ is Lipschitz near $r_0$ with constant $M$. Lemma~\ref{lem:prelimexpchar}, applied to $f$ with this Lipschitz constant $M$ and $r=r_0$, gives a sequence $r_n\to r_0$ and $\psi_n\in C^2([R_1,R_2])$ such that $f-\psi_n$ has a local minimum at $r_n$, with
    \[
        0 < |\psi_n'(r_n)| \le M \qquad \forall n.
    \]

    \noindent\textit{Step 3 (non-characteristic points).} Having produced, in Step~2, a sequence $r_n\to r_0$ of contact levels, we now need to exhibit, for each $n$, an actual point $p_n$ close to $p_{\mathrm{char}}$, on the leaf $\Sigma_{r_n}$, which is non-characteristic. Since $p_{\mathrm{char}}\in A$ already, the strategy is direct: flow from $p_{\mathrm{char}}$ itself for a short time to land exactly on $\Sigma_{r_n}$, then perturb slightly within $\Sigma_{r_n}\cap A$ to reach a non-characteristic point.

    Since $p_{\mathrm{char}}\in A$ and $A$ is open, there exists $\rho>0$ such that the closed Euclidean ball $\overline{B}(p_{\mathrm{char}},\rho)\subset A$. Since $R\in C^2(A)$ with $\nabla_{\mathrm{eucl}}R\neq0$ on $A$, the vector field
    \[
        W_{\mathrm{eucl}} = \frac{\nabla_{\mathrm{eucl}} R}{|\nabla_{\mathrm{eucl}} R|_{\mathbb R^3}^2}
    \]
    is well defined and of class $C^1$ on $\overline{B}(p_{\mathrm{char}},\rho)$, hence bounded there; set
    \[
        M_W = \sup_{\overline{B}(p_{\mathrm{char}},\rho)} |W_{\mathrm{eucl}}|_{\mathbb R^3} \;<\;\infty.
    \]
    Let $\gamma:(-\tau_0,\tau_0)\to\mathbb R^3$ be the (Euclidean) integral curve of $W_{\mathrm{eucl}}$ with $\gamma(0)=p_{\mathrm{char}}$, defined for as long as it remains in $\overline B(p_{\mathrm{char}},\rho)$. Since $\nabla_{\mathrm{eucl}}R\cdot W_{\mathrm{eucl}}\equiv1$,
    \begin{equation}\label{eq:propcurvaeuclr}
        \frac{d}{d\tau}\, R(\gamma(\tau)) = \nabla_{\mathrm{eucl}} R(\gamma(\tau))\cdot W_{\mathrm{eucl}}(\gamma(\tau)) = 1, \qquad\text{so}\qquad R(\gamma(\tau)) = r_0 + \tau,
    \end{equation}
    and, as long as $\gamma(\tau)\in\overline B(p_{\mathrm{char}},\rho)$,
    \begin{equation}\label{eq:flowbound}
        |\gamma(\tau)-p_{\mathrm{char}}|_{\mathbb R^3} \;\le\; \int_0^{|\tau|} |\dot\gamma(\sigma)|_{\mathbb R^3}\,d\sigma \;\le\; M_W\,|\tau|.
    \end{equation}
    In particular, for $|\tau|<\rho/M_W$ the trajectory has not left $B(p_{\mathrm{char}},\rho)\subset A$.

    Since $r_n\to r_0$, there is $\bar n$ such that $|r_n-r_0|<\rho/M_W$ for every $n\ge\bar n$; for such $n$, set
    \[
        q_n = \gamma(r_n-r_0) \in B(p_{\mathrm{char}},\rho)\subset A.
    \]
    By \eqref{eq:propcurvaeuclr}, $R(q_n)=r_n$, i.e.
    \[
        q_n \in \Sigma_{r_n}\cap A \neq \emptyset,
    \]
    and by \eqref{eq:flowbound},
    \[
        |q_n-p_{\mathrm{char}}|_{\mathbb R^3} \;\le\; M_W\,|r_n-r_0| \;\longrightarrow\; 0.
    \]

    Since $\Sigma_{r_n}\cap A$ is a non-empty $C^2$-hypersurface (implicit function theorem, using $R\in C^2(A)$, $\nabla_{\mathrm{eucl}}R\neq0$ on $A$), the Frobenius-type argument shows that the set of points where $\nabla R\neq0$ (horizontal, left-invariant gradient) is dense in $\Sigma_{r_n}\cap A$. We may therefore choose
    \[
         p_n \in (\Sigma_{r_n}\cap A)\setminus\mathcal C(\Sigma_{r_n}), \qquad | p_n - q_n|_{\mathbb R^3} < \frac1n.
    \]
    By the triangle inequality,
    \[
        | p_n - p_{\mathrm{char}}|_{\mathbb R^3} \le | p_n-q_n|_{\mathbb R^3} + |q_n-p_{\mathrm{char}}|_{\mathbb R^3} \longrightarrow 0,
    \]
    so $ p_n\to p_{\mathrm{char}}$, with $ p_n\in A$ non-characteristic for every $n$, i.e.\ $\nabla R( p_n)\neq0$; moreover $R\in C^2$ in a neighborhood of $ p_n$ (since $ p_n\in A$ open), so $\Psi_n=\psi_n\circ R$ is an admissible $C^2$ test function near $ p_n$.

        Since $\psi_n'(r_n)\neq0$ and $\nabla R( p_n)\neq0$, inequality~\eqref{eq:supersp1} holds at $ p_n$:
    \[
        1 \;\le\; -\psi_n'(r_n)\,|\nabla R(p_n)|\,\operatorname{div}\!\Big(\frac{\nabla R}{|\nabla R|}\Big)( p_n)
        \;\le\; |\psi_n'(r_n)|\,|\nabla R( p_n)|\,\Big|\operatorname{div}\!\Big(\frac{\nabla R}{|\nabla R|}\Big)( p_n)\Big|.
    \]
    Since $0<|\psi_n'(r_n)|\le M$ (Step~2), this yields the quantitative lower bound
    \begin{equation}\label{eq:quantlowerbound}
        \Big|\operatorname{div}\!\Big(\frac{\nabla R}{|\nabla R|}\Big)( p_n)\Big| \;\ge\; \frac{1}{|\psi_n'(r_n)|\,|\nabla R( p_n)|} \;\ge\; \frac{1}{M\,|\nabla R( p_n)|}.
    \end{equation}

    \medskip
        \noindent\textit{Step 4 (conclusion).} Since $p_{\mathrm{char}}\in A$ and $A$ is open, $R\in C^2$ in a full neighborhood of $p_{\mathrm{char}}$; as $p_n\to p_{\mathrm{char}}$, continuity of $\nabla R$ on this neighborhood gives $\nabla R(p_n)\to\nabla R(p_{\mathrm{char}})$. Since $R\in C^2$ near $p_{\mathrm{char}}$ with $\nabla_{\mathrm{eucl}}R\neq0$ there, $R$ is itself an admissible defining function for $\Sigma_{r_0}$ near $p_{\mathrm{char}}$; as the characteristic set does not depend on the choice of defining function and $p_{\mathrm{char}}\in\mathcal C(\Sigma_{r_0})$, this gives $\nabla R(p_{\mathrm{char}})=0$. Hence $|\nabla R(p_n)|\to0$, and \eqref{eq:quantlowerbound} gives directly
    \[
        \Big|\operatorname{div}\!\Big(\frac{\nabla R}{|\nabla R|}\Big)(p_n)\Big| \;\ge\; \frac{1}{M\,|\nabla R(p_n)|} \;\longrightarrow\;+\infty.
    \]
    Setting $\tau_n:=f(r_n)\in[t_1,t_2]$ (so that $p_n\in\Sigma_{r_n}=\Gamma_{\tau_n}$) this gives
    \[
        H_n(p_n)=\Big|\operatorname{div}\Big(\frac{\nabla R}{|\nabla R|}\Big)(p_n)\Big| \to \infty
        \quad\text{as } n \to \infty,
    \]
    which is the conclusion of Theorem~\ref{thm:hmcfoliation}.
\end{proof}

\subsection{Self-similar solutions}


We first observe that self similar solutions immediately yield a generalized foliation.
\begin{lem}
\label{thm:selfsimfoli}
Let $u$ be a self-similar solution to \eqref{HMCFprobl}, with $\lambda:[0,T^*)\to(0,1]$ satisfying \eqref{eq:defselfsimlambda}. Set $\Gamma_t=\{p \in \overline{\Omega} \mid u(p)=t\}$. Then. for every $\bar t\in(0,T^*)$, the level sets $\{\Gamma_t\}_{t\in[0,\bar t]}$ form a generalized $C^2$-foliation of $\overline\Omega\setminus F_{\lambda(\bar t)}$.
\end{lem}

\begin{proof}
By definition of self-similar solution, $\Gamma_t=\delta_{\lambda(t)}(\partial\Omega)=\Sigma_{\lambda(t)}$ for every $t\in[0,T^*)$, with $\lambda$ continuous, so the family $\{\Gamma_t\}_{t\in[0,\bar t]}$ coincides, as a family of sets, with $\{\Sigma_\lambda\}_{\lambda\in[\lambda(\bar t),1]}$. The hypersurfaces $\Sigma_\lambda=\delta_\lambda(\partial\Omega)$ satisfy the strict nesting $\overline{F_{\lambda_1}}\subset F_{\lambda_2}$ for $\lambda_1<\lambda_2$, since distinct time-slices of a solution to \eqref{HMCFprobl} are disjoint and nested, being level sets of $u$. The claim then follows directly from Lemma~\ref{lem:dilation-foliation}, applied with $\bar\lambda=\lambda(\bar t)$.
\end{proof}

We are then interested in understanding whether the characteristic points of a self-similar generalized foliation are in fact regular. We start with the following simple observation.

{
\begin{lem}\label{lem:charpointsinG}
Let $\Omega\subset\R^3$ be a bounded open set with $C^2$ boundary, and let $g:\R^3\to\R$ be a $C^2$ defining function for $\Omega$ as in Lemma~\ref{lem:dilation-foliation}, i.e. $\partial\Omega=\{g=0\}$ and $\nabla_{\mathrm{eucl}}g\neq0$ on $\partial\Omega$. If $p=(x,y,z)\in\mathcal C(\partial\Omega)$, then $p$ belongs to the set $G$ of Definition~\ref{def:transversalset} if and only if $z\neq0$.
\end{lem}

\begin{proof}
Since $p\in\mathcal C(\partial\Omega)$, $Xg(p)=Yg(p)=0$, i.e.\
\[
g_x(p)=\tfrac y2\,g_z(p), \qquad g_y(p)=-\tfrac x2\,g_z(p),
\]
so that
\[
\nabla_{\mathrm{eucl}}g(p) = g_z(p)\left(\tfrac y2,-\tfrac x2,1\right).
\]
As $\partial\Omega$ is $C^2$-regular, $\nabla_{\mathrm{eucl}}g(p)\neq0$; since the third component of $\left(\tfrac y2,-\tfrac x2,1\right)$ equals $1\neq0$, this forces $g_z(p)\neq0$. Consequently,
\[
\langle D(p),\nabla_{\mathrm{eucl}}g(p)\rangle
= g_z(p)\left\langle (x,y,2z),\left(\tfrac y2,-\tfrac x2,1\right)\right\rangle
= g_z(p)\left(\tfrac{xy}2-\tfrac{xy}2+2z\right)
= 2z\,g_z(p).
\]
Since $g_z(p)\neq0$, this vanishes if and only if $z=0$. By definition of $G$, this proves the claim.
\end{proof}}

We can finally give the proof of Theorem \ref{cor:noselfsim}.

\begin{proof}[Proof of Theorem \ref{cor:noselfsim}]
Let $u\in C^0(\bar \Omega)$ be a right-Lipschitz solution of~\eqref{HMCFprobl}  with initial datum $\Omega$. Assume by contradiction that $u$ is self-similar in the sense of Definition~\ref{defn:ss} and let  $\lambda:[0,T^*)\to(0,1]$ be the continuous function satisfying~\eqref{eq:defselfsimlambda}.

Fix $\bar t\in(0,T^*)$ and $\eta>0$ such that $\bar t+\eta<T^*$. 

By Lemma~\ref{thm:selfsimfoli}, the level sets $\{\Gamma_t\}_{t\in[0,\bar t+\eta]}$ form a generalized $C^2$-foliation of $\overline\Omega\setminus F_{\lambda(\bar t+\eta)}$, with regular set $A=\bigcup_{\lambda\in(\lambda(\bar t+\eta),1)}\delta_\lambda(G)$ (Lemma~\ref{lem:dilation-foliation}, applied with $\bar\lambda=\lambda(\bar t+\eta)$).

Moreover, by Corollary \ref{cor:axisimcarpoint}, the characteristic set $\mathcal{C}(\partial \Omega)$ consists of two points $p_1, p_2$ on $z$-axis. Both $p_1$ and $p_2$ are distinct from the origin: if $0\in\partial\Omega$, since $\delta_\lambda(0)=0$ for every $\lambda>0$, the origin would belong to $\Gamma_t=\delta_{\lambda(t)}(\partial\Omega)$ for every $t\in[0,T^*)$, contradicting the disjointness of distinct level sets of $u$. In particular, $p_1$ and $p_2$ are contained in the set $G$ of Definition \ref{def:transversalset}, and thus by Lemma \ref{lem:dilation-foliation} the characteristic points of $\Gamma_t$, given by $\delta_{\lambda(t)}(p_1)$ and $\delta_{\lambda(t)}(p_2)$ are contained in the regular set $A$.

By Theorem~\ref{thm:hmcfoliation} applied to $p_{\mathrm{\mathrm{char}}}\in A\cap\mathcal C(\Gamma_{\bar t})$, there exists a sequence $\{p_n\}_n\subset\Omega_{[0,\bar t+\eta]}$ such that
\[
    p_n \to p_{\mathrm{\mathrm{char}}}, \qquad H_n(p_n)\to\infty \quad\text{as } n\to\infty,
\]
where $H_n(p_n)$ denotes the horizontal mean curvature of the leaf $\Gamma_{\tau_n}$ through $p_n$, for some $\tau_n\in[0,\bar t+\eta]$.

 Since $\Gamma_{\tau_n}=\delta_{\lambda(\tau_n)}(\partial\Omega)$, write $p_n=\delta_{\lambda_n}(q_n)$ for a unique $q_n\in\partial\Omega$, where $\lambda_n=\lambda(\tau_n)$. The anisotropic dilation scales the horizontal mean curvature as
\[
    H(\Gamma_{\tau_n})(p_n) = \lambda_n^{-1}\, H(\partial\Omega)(q_n).
\]
Since the initial datum has bounded curvature, $H(\partial\Omega)\le c$ on $\partial\Omega\setminus\mathcal C(\partial\Omega)$, we obtain
\[
    H_n(p_n) = |H(\Gamma_{\tau_n})(p_n)| \le \frac{c}{\lambda_n}.
\]
Since $\tau_n\in[0,\bar t+\eta]$ and $\lambda$ is continuous and strictly positive on the compact interval $[0,\bar t+\eta]\subset[0,T^*)$, we have $\lambda_n\ge\min_{[0,\bar t+\eta]}\lambda =: \lambda_{\min}>0$ for every $n$, hence
\[
    H_n(p_n) \le \frac{c}{\lambda_{\min}} < \infty \qquad \text{for every } n,
\]
uniformly in $n$. This contradicts $H_n(p_n)\to\infty$.
\end{proof}

\bibliographystyle{abbrv}
\bibliography{references.bib}

\end{document}